\documentclass[leqno,a4paper]{article}
\usepackage{amsmath,amsthm,amssymb}

\PassOptionsToPackage{hyphens}{url}\usepackage{hyperref} 

\usepackage{enumerate}
\usepackage{graphicx}

\usepackage{subcaption}

\usepackage[svgnames,x11names]{xcolor} 
\hypersetup{
  colorlinks, linkcolor={DarkOrange4},
  citecolor={DarkOliveGreen4}, urlcolor={DarkOliveGreen4}
}

\usepackage{centernot}

\usepackage{enumitem}
\usepackage{color}

\usepackage{bbm}
 \usepackage{mathrsfs}

\newtheorem{Th}{Theorem}[section]
\newtheorem{Prop}[Th]{Proposition}
\newtheorem{Lemma}[Th]{Lemma}
\newtheorem{Cor}[Th]{Corollary}
\theoremstyle{definition}
\newtheorem{Remark}[Th]{Remark}
\newtheorem{Def}[Th]{Definition}
\newtheorem{Example}[Th]{Example}
\newtheorem{Qu}[Th]{Question}

\newcommand{\beq}{\begin{equation}}
\newcommand{\eeq}{\end{equation}}

\def\scalar(#1,#2){(#1\mid#2)}

\renewcommand{\hat}{\widehat}

\newcommand{\cb}{{\cal B}}

\newcommand{\cn}{{\cal N}}

\newcommand{\xbm}{(X,{\cal B},\mu)}

\newcommand{\ot}{\otimes}

\newcommand{\A}{\mathbb{A}}
\newcommand{\B}{\mathbb{B}}

\newcommand{\R}{{\mathbb{R}}}
\newcommand{\T}{{\mathbb{T}}}
\newcommand{\C}{{\mathbb{C}}}
\newcommand{\Z}{{\mathbb{Z}}}
\newcommand{\N}{{\mathbb{N}}}
\newcommand{\E}{{\mathbb{E}}}
\newcommand{\I}{\mathbb{I}}
\newcommand{\vep}{\varepsilon}

\newcommand{\M}{{\cal M}}

\newcommand{\mob}{\boldsymbol{\mu}}
\newcommand{\lamob}{\boldsymbol{\lambda}}
\newcommand{\tend}[3][]{\xrightarrow[#2\to#3]{#1}}
\newcommand{\raz}{\mathbf{1}}

\newcommand\egdef{:=}
\newcommand{\NS}{{\mathbb{N}^*}}

	\newcommand{\Id}{\mathop{\mbox{Id}}}

\makeatletter
\def\and{
  \end{tabular}%
  \begin{tabular}[t]{c}}
\makeatother

\providecommand{\noopsort}[1]{} 

\title{The Chowla and the Sarnak conjectures from ergodic theory point of view (extended version)}
\author{H.\ El Abdalaoui \and J.\ Ku\l aga-Przymus\thanks{Research supported by Narodowe Centrum Nauki grant 2014/15/B/ST1/03736.} \and  M.\ Lema\'nczyk\footnotemark[1] \and  T.\ de la Rue}

\begin{document}
\bibliographystyle{amsplain}

\maketitle

\begin{abstract}
 We rephrase the conditions from the Chowla and the Sarnak conjectures in abstract setting, that is, for sequences in $\{-1,0,1\}^{\N^\ast}$, and introduce several natural generalizations. We study the relationships between these properties and other notions from topological dynamics and ergodic theory. 
\end{abstract}
\thispagestyle{empty}

\tableofcontents

\section{Introduction}
A motivation for the present work comes from a dynamical point of view on some classical arithmetic functions taken up recently by Sarnak \cite{Sarnak}. Namely, we consider the following two functions: the M\"{o}bius function $\mob\colon\N^\ast:=\N\setminus\{0\}\to \{-1,0,1\}$ given by $\mob(1)=1$ and
\begin{equation}\label{def:mob}
\mob(n)=
\begin{cases}
(-1)^k& \text{ if $n$ is a product of $k$ distinct primes},\\
0& \text{ otherwise},
\end{cases}
\end{equation}
and the Liouville function $\lamob\colon \N^\ast\to \{-1,1\}$ defined by
$$
\lamob(n)=(-1)^{\Omega(n)},
$$
where $\Omega(n)$ is the number of prime factors of $n$ counted with multiplicities. The importance of these two functions in number theory is well known and may be illustrated by the following statement
\begin{equation}\label{E:la}
\sum_{n\leq N}\lamob(n)={\rm o}(N)=\sum_{n\leq N}\mob(n),
\end{equation}
which is equivalent to the Prime Number Theorem, see e.g.\ \cite{Ap}, p.~91. Recall also the classical connection of $\mob$ with the Riemann zeta function, namely
$$
\frac1{\zeta(s)}=\sum_{n=1}^{\infty}\frac{\mob(n)}{n^s} \text{ for any }s\in\mathbb{C}\text{ with }\Re(s)>1.
$$
In \cite{Titchmarsh}, it is shown that the Riemann Hypothesis is equivalent to the following: for each $\varepsilon>0$, we have
\[
\sum_{n\leq N}\mob(n)={\rm O}_\varepsilon\left(N^{\frac12+\varepsilon}\right) \text{ as } N \to \infty.
\]

In~\cite{Cho}, Chowla formulated a conjecture on the correlations of the Liouville function. The analogous conjecture for the M\"{o}bius function takes the following form: for each choice of $1\leq a_1<\dots<a_r$, $r\geq 0$, with $i_s\in \{1,2\}$, not all equal to~$2$, we have
\begin{equation}\label{cza}
\sum_{n\leq N}\mob^{i_0}(n)\cdot \mob^{i_1}(n+a_1)\cdot\ldots\cdot\mob^{i_r}(n+a_r)={\rm o}(N).
\end{equation}
Recently, Sarnak~\cite{Sarnak} formulated the following conjecture: for any dynamical system $(X,T)$, where $X$ is a compact metric space and $T$ is a homeomorphism of zero topological entropy, for any $f\in C(X)$ and any $x \in X$, we have
\begin{equation}\label{sar}
\sum_{n\leq N}f(T^nx)\mob(n)={\rm o}(N).
\end{equation}
From now on, we refer to~\eqref{sar} as the Sarnak conjecture.
Moreover, it is also noted in \cite{Sarnak} that for any measure-theoretic dynamical system $(X,\mathcal{B},\mu,T)$, for any $f \in L^2(X,\mathcal{B},\mu)$, the condition~\eqref{sar} holds for $\mu$-almost every $x\in X$.
As can be shown, this a.e.\ version of~\eqref{sar} is a consequence of the following Davenport's estimation \cite{Da}: for each $A>0$, we have
\begin{equation}\label{vin}
\max_{z \in \T}\left|\displaystyle\sum_{n \leq N}z^n\mob(n)\right|\leq C_A\frac{N}{\log^{A}N}\text{ for some }C_A>0\text{ and all }N\geq 2,
\end{equation}
combined with the spectral theorem (for a complete proof see Section~\ref{se:sae}). Finally,
Sarnak also proved that the Chowla conjecture~\eqref{cza} implies~\eqref{sar}.

The aim of this paper is to deal with the Chowla conjecture~\eqref{cza} and the Sarnak conjecture~\eqref{sar} in a more abstract setting. In Section~\ref{se:3.1}, we introduce conditions~\eqref{Ch} and~\eqref{S0} in the context of arbitrary sequences $z\in \{-1,0,1\}^{\N^\ast}$. They are obtained from~\eqref{cza} and~\eqref{sar} by replacing $\mob$ with $z$, respectively. In other words, we consider the sums of the form:
\begin{equation}\label{CZA}
\sum_{n\leq N}z^{i_0}(n)z^{i_1}(n+a_1)\cdot\ldots\cdot z^{i_r}(n+a_r)
\end{equation}
and
\begin{equation}\label{SAR}
\sum_{n\leq N}f(T^nx)z(n),
\end{equation}
and require that they are of order ${\rm o}(N)$ ($a_s$ and $i_s$ are as in~\eqref{cza}, $T$, $f$ and~$x$ are as in~\eqref{sar}). Finally, we define a new condition~\eqref{S}, formally stronger than~\eqref{S0}, by requiring that the sum given by~\eqref{SAR} is of order ${\rm o}(N)$ for any homeomorphism $T$ of a compact metric space $X$, any $f\in C(X)$ and any completely deterministic point $x\in X$.\footnote{Recall that $x\in X$ is said to be completely deterministic if for any accumulation point $\nu$ of $(\frac{1}{N}\sum_{n\leq N}\delta_{T^nx})_{N\in\N}$, the system $(X,\nu,T)$ is of zero entropy.} Note that if $h_{top}(T)=0$ then all points are completely deterministic.

We provide a detailed proof of the fact that~\eqref{Ch} implies~\eqref{S}, see Theorem~\ref{ChS} below. Classical tools from ergodic theory, such as joinings (see Section~\ref{se:Ch-to-S}), will be here crucial. This approach (for $z=\mob$ and~\eqref{S0} instead of~\eqref{S}) was suggested in~\cite{Sarnak}, together with a rough sketch of the proof.\footnote{Sarnak also announced a purely combinatorial proof of this result (and sent it to us in a letter). See also~\cite{Tao}.} Since~\eqref{S} implies~\eqref{S0} directly from the definitions, we obtain the following:
$$
\eqref{Ch} \implies \eqref{S}\implies \eqref{S0}.
$$

By replacing~\eqref{SAR} with the sums of the form
\begin{equation}\label{STRO}
\sum_{n\leq N}f(T^nx)z^{i_0}(n)\cdot z^{i_1}(n+a_1)\cdot\ldots\cdot z^{i_r}(n+a_r)
\end{equation}
in~\eqref{S0} and~\eqref{S}, we obtain conditions called~\eqref{S0-strong} and~\eqref{S-strong}, respectively. Notice that such sums generalize both~\eqref{CZA} and~\eqref{SAR}.
Clearly
$$
\eqref{S-strong}\implies \eqref{S0-strong}\implies \eqref{Ch}.
$$
In Section~\ref{se:CHSS0}, we show that the above three properties are, in fact, equivalent:
$$
\eqref{S-strong}\iff~\eqref{S0-strong}\iff~\eqref{Ch}.
$$
Section~\ref{se:SS0} is devoted to the proof of Theorem~\ref{thm:SequivS0} which says that although formally~\eqref{S} is stronger than~\eqref{S0}, in fact, we have
$$
\eqref{S}\iff\eqref{S0}.
$$

Section~\ref{se:4} answers some natural questions about possible relations between the properties under discussion. First, in Section~\ref{se:4.1}, we show that
$$
\eqref{S}\centernot\implies\eqref{Ch}.
$$
In Section~\ref{se:4.2}, we show that a sequence $z\in \{-1,0,1\}^{\N^\ast}$ satisfying~\eqref{Ch} need not be generic. 
In Section~\ref{se:4.3}, we give an example of a sequence satisfying a weakened version of~\eqref{Ch}, in which we consider only exponents $i_s=1$, but failing to satisfy~\eqref{Ch} in its full form. Finally, in Section~\ref{se:4.4} and Section~\ref{se:4.5}, we discuss the properties of recurrence and unique ergodicity for sequences satisfying~\eqref{Ch}.

Section~\ref{se:5} is motivated by the problem of describing the set
\beq\label{valent}
\{(h_{top}(z^2),h_{top}(z)): \;z\in \{-1,0,1\}^{\N^\ast}\; \mbox{satisfying}~\eqref{Ch}\}.
\eeq
For any sequence $w$ satisfying~\eqref{Ch} and such that $w^2=\mob^2$, we have (cf.\ \cite{Sarnak} and Remark~\ref{uw622} below) $(h_{top}(w^2),h_{top}(w))=(\frac{6}{\pi^2}, \frac{6}{\pi^2}\log3)$. Moreover, if $u\in \{-1,1\}^{\N^*}$ satisfies~\eqref{Ch} then $(h_{top}(u^2),h_{top}(u))=(0,1)$. We will discuss, in general, what are possible values of $(h_{top}(z^2),h_{top}(z))$ for sequences over $\{-1,0,1\}$ and provide further examples  of $z$ satisfying~\eqref{Ch} with $h_{top}(z^2)$ being an arbitrary number in $[0,1]$ using Sturmian sequences.

In Section~\ref{se:6}, we deal with Toeplitz sequences \cite{Do1,Ja-Ke} over the alphabet $\{-1,0,1\}$. Although Toeplitz sequences are obtained as a certain limit of periodic sequences (and periodic sequences are orthogonal to $\mob$), their behavior differs from the behavior of periodic sequences in the context of the Chowla and the Sarnak conjectures. Given a sequence $z\in \{-1,0,1\}^{\N^\ast}$ satisfying some extra assumptions (see Theorems~\ref{tw:abs2} and~\ref{tw:abs1}), we construct Toeplitz sequences $t$, that are not orthogonal to $z$ and are of positive topological entropy, providing also more precise entropy estimates. We apply this to $z=\mob$, $z=\mob_\mathscr{B}$ and to sequences satisfying~\eqref{Ch}, defined in Section~\ref{se:5.2.3} and Section~\ref{se:5.2.4}.
For further motivations and related results see \cite{Ab-Ka-Le,Do-Ka}.

\bigskip

The authors wish to address their thanks to the two referees for valuable remarks and comments which improved the quality of the present work.

\section{Preliminaries}

\subsection{Measure-theoretical dynamical systems}

\subsubsection{Factors and extensions}
Let $T\colon (X,\mathcal{B},\mu)\to (X,\mathcal{B},\mu)$ and $S\colon (Y,\mathcal{A},\nu)\to (Y,\mathcal{A},\nu) $ be automorphisms of standard Borel probability spaces.
\begin{Def}
We say that $S$ is a \emph{factor} of $T$ (or $T$ is an \emph{extension} of $S$) if there exists $\pi\colon (X,\mathcal{B},\mu)\to (Y,\mathcal{A},\nu)$ such that $S\circ \pi=\pi\circ T$. To simplify notation, we will identify the factor $S$ with the $\sigma$-algebra $\pi^{-1}(\mathcal{A})\subset \mathcal{B}$. Moreover, any $T$-invariant sub-$\sigma$-algebra $\mathcal{A}\subset \mathcal{B}$ will be identified with the corresponding factor $T|_\mathcal{A}\colon (X/\mathcal{A},\mathcal{A},\mu|_\mathcal{A})\to (X/\mathcal{A},\mathcal{A},\mu|_\mathcal{A})$.
\end{Def}
Let now $S_i\colon (Y_i,\mathcal{A}_i,\nu_i)\to (Y_i,\mathcal{A}_i,\nu_i)$, $i=1,2$, be factors of $T\colon (X,\mathcal{B},\mu)\to (X,\mathcal{B},\mu)$, with the factoring maps $\pi_i\colon X\to Y_i$, $i=1,2$. We will denote by $(Y_1,\mathcal{A}_1,\nu_1)\vee (Y_2,\mathcal{A}_2,\nu_2)$ the smallest factor of $\mathcal{B}$ containing both $\pi_1^{-1}(\mathcal{A}_1)$ and $\pi_2^{-1}(\mathcal{A}_2)$.\footnote{This factor can be viewed as a joining of $S_1$ and $S_2$, see Section~\ref{se:joi}.}

\subsubsection{Entropy}

Let $T$ be an automorphism of a standard Borel probability space $(X,\mathcal{B},\mu)$. Recall that the measure-theoretic entropy of $T$ is defined in the following way. Given a finite measurable partition $Q=\{Q_1,\dots,Q_k\}$ of $X$, we define
$$
H(Q):=-\sum_{1\leq m\leq k} \mu (Q_m) \log \mu(Q_m).\footnote{We consider $2$ as the base of logarithm.}
$$
(We may also write $H_\mu(Q)$ if we need to underline the role of $\mu$.)
The measure-theoretic entropy of $T$ with respect to the partition $Q$ is then defined as
$$
h_\mu(T,Q) := \lim_{N \rightarrow \infty} \frac{1}{N} H\left(\bigvee_{n=0}^{N-1} T^{-n}Q\right),
$$
where $\bigvee_{n=0}^{N-1} T^{-n}Q$ is the coarsest refinement of all partitions $T^{-n}Q$, $n=0,\dots,N-1$.
\begin{Def}[Kolmogorov and Sinai]
The \emph{measure-theoretic entropy} of $T$ is given by
$$
h(T,\mu) = \sup_Q h_\mu(T,Q),
$$
where the supremum is taken over all finite measurable partitions.
\end{Def}
\begin{Def}
We say that $T\colon (X,\mathcal{B},\mu)\to (X,\mathcal{B},\mu)$ is a \emph{K-system} if any non-trivial factor of $T$ has positive entropy.
\end{Def}
\begin{Def}
Let $T_i\colon (X_i,\mathcal{B}_i,\mu_i)\to (X_i,\mathcal{B}_i,\mu_i)$, $i=1,2$, be such that $T_2$ is a factor of $T_1$. 
\begin{itemize}
\item
The quantity $h(T_1,\mu_1)-h(T_2,\mu_2)$ is called the \emph{relative entropy} of $T_1$ with respect to $T_2$.
\item
If the extension $T_1\to T_2$ is non trivial, and if for any intermediate factor $T_3\colon (X_3,\mathcal{B}_3,\mu_3)\to (X_3,\mathcal{B}_3,\mu_3)$ between $T_1$ and $T_2$, with factoring map $\pi_3\colon X_3\to X_2$, the relative entropy of $T_3$ with respect to $T_2$ is positive unless $\pi_3$ is an isomorphism, we say that the extension $T_1\to T_2$ is \emph{relatively~K}.
\end{itemize}
\end{Def}

\subsubsection{Joinings}\label{se:joi}
\begin{Def}
Given automorphisms of  standard Borel probability spaces
$$
T_i\colon (X_i,\mathcal{B}_i,\mu_i)\to (X_i,\mathcal{B}_i,\mu_i),\ i=1,\dots,k,
$$
let $J(T_1,\dots,T_k)$ be the set of all probability measures $\rho$ on $(X_1\times \dots\times X_k,\mathcal{B}_1\otimes \dots\otimes \mathcal{B}_k)$, invariant under $T_1\times \dots\times T_k$ and such that $(\pi_i)_\ast(\rho)=\mu_i$, where $\pi_i\colon X_1\times \dots\times X_k\to X_i$ is given by $\pi_i(x_1,\dots,x_k)=x_i$ for $1\leq i\leq k$. Any $\rho\in J(T_1,\dots,T_k)$ is called a~\emph{joining}.
\end{Def}
\begin{Def}
Following~\cite{Fur}, we say that $T_1$ and $T_2$ are \emph{disjoint} if $J(T_1,T_2)=\{\mu_1\otimes \mu_2\}$. We then write $T_1\perp T_2$.
\end{Def}

Suppose now that $T_3\colon (X_3,\mathcal{B}_3,\mu_3)\to (X_3,\mathcal{B}_3,\mu_3)$ is a common factor of $T_1$ and $T_2$. To keep the notation simple, we assume that $\mathcal{B}_3$ is a sub-$\sigma$-algebra of both $\mathcal{B}_1$ and $\mathcal{B}_2$. Given $\lambda\in J(T_3,T_3)$, we define the \emph{relatively independent extension} of $\lambda$, i.e. $\widehat{\lambda}\in J(T_1,T_2)$, by setting for each $A_i\in\mathcal{B}_i$, $i=1,2$:
$$
\widehat{\lambda}(A_1\times A_2):=\int_{X_3\times X_3}\E(\raz_{A_1}|\mathcal{B}_3)(x)
\E(\raz_{A_2}|\mathcal{B}_3)(y)\ d\lambda(x,y).
$$
Consider now those $\widetilde{\Delta}\in J(T_1,T_2)$ that project down to the diagonal joining $\Delta\in J(T_3,T_3)$ given by $\Delta(A\times B):=\mu_3(A\cap B)$. If $\widehat{\Delta}$ is the only such joining, we say that $T_1$ and $T_2$ are \emph{relatively independent} over their common factor $T_3$. We then write $T_1\perp_{T_3}T_2$.

\begin{Remark}[\cite{Thouvenot}, Lemme 3]\label{Thouvenot}
If the extension $T_1\to T_3$ is of zero relative entropy and $T_2\to T_3$ is relatively K then $T_1\perp_{T_3}T_2$.
In particular, (taking for $T_3$ the trivial one-point system) if $T_1$ has zero entropy and $T_2$ is K, then $T_1\perp T_2$.
\end{Remark}

\subsection{Topological dynamical systems}
\subsubsection{Invariant measures}
Let $T\colon X\to X$ be a continuous map of a compact metric space. We denote by $\mathcal{P}_T(X)$ the set of $T$-invariant probability measures on $(X,\mathcal{B})$ with $\mathcal{B}$ standing for the $\sigma$-algebra of Borel sets. The space of probability measures on $X$ is endowed with the (metrizable) weak topology:
$$
\nu_n\tend{n}{\infty} \nu\iff \int_X f\ d\nu_n \tend{n}{\infty} \int_X f\ d\nu \text{ for each }f\in C(X),
$$
where $C(X)$ denotes the space of continuous functions on $X$. The weak topology is compact, and $\mathcal{P}_T(X)$ is closed in it.

By the Krylov-Bogolyubov theorem, $\mathcal{P}_T(X)\neq \varnothing$. In fact, for any $x\in X$, if we set
$$
\delta_{N,x}:=\frac{1}{N}\sum_{n\leq N}\delta_{T^nx},\footnote{In what follows, we will also use the notation $\delta_{T,N,x}$ if confusion could arise.}
$$
and if, for some increasing sequence $(N_k)_{k\in\N}$ and some probability measure $\nu$, $\delta_{N_k,x}\tend{k}{\infty} \nu$,
then $\nu\in\mathcal{P}_T(X)$. In such a situation, we say that $x$ is \emph{quasi-generic} for $\nu$ along $(N_k)$, and we set
$$
\text{Q-gen}(x):=\left\{\nu\in\mathcal{P}_T(X):\delta_{N_k,x}\tend{k}{\infty}\nu\;\mbox{for a subsequence}\;(N_k)\right\}.
$$
If $\delta_{N,x}\tend{N}{\infty} \nu$, i.e.\ if $\text{Q-gen}(x)=\{\nu\}$, we say that $x$ is \emph{generic} for $\nu$.
\begin{Def}[\cite{We}, see also \cite{Ka}]
We say that $x\in X$ is {\em completely deterministic} if, for each $\nu\in\text{Q-gen}(x)$, we have
$h(T,\nu)=0$. We will then write
\begin{equation}\label{ent}
\text{Q-gen}(x)\subset [h=0].
\end{equation}
\end{Def}

\subsubsection{Symbolic dynamical systems}
Let $A$ be a nonempty finite set and $\I=\N^\ast$ or $\Z$. Then $A^\I$ endowed with the product topology is a compact metric space. Coordinates of $w\in A^\I$ will be denoted either by $w_n$ or by $w(n)$ for $n\in\I$.
\begin{Def}
The subsets of $A^\I$ of the form
$$
C_t(a_0,\dots,a_{k-1})=\{w\in A^\I\colon w_{t+j}=a_{j}\text{ for }0\leq j\leq k-1\},
$$
where $k\ge1$, $t\in \I$ and $a_0,\dots,a_{k-1}\in A$, are called \emph{cylinders} and they form a basis for the product topology.
\end{Def}
\begin{Def}
Any $C=(a_0,\dots,a_{k-1})\in A^k$, $k\geq 1$, is called a \emph{block of length $k$}. For any $0\leq i\leq k-1$, let $C(i):=a_i$.
\end{Def}
We will identify blocks with the corresponding cylinders:
$$
C=(a_0,\dots,a_{k-1})\in A^k\longleftrightarrow C_0(a_0,\dots,a_{k-1}).
$$
\begin{Def}
We say that a block $C=(a_0,\dots,a_{k-1})\in A^k$ {\em appears in} $w$ if $w\in C_t(a_0,\dots,a_{k-1})$ for some $t\in \I$.
\end{Def}

On $A^\I$ there is a natural continuous action by the \emph{left shift} $S$:
$$
S\colon A^\I\to A^\I,\ S((w_n)_{n\in\I})=(w_{n+1})_{n\in\I} \text{ for }w=(w_n)_{n\in\I}\in A^\I.
$$
(For $\I=\Z$, $S$ is clearly invertible and it is a homeomorphism.)
\begin{Def}
Let $C=(a_0,\dots,a_{k-1})\in A^k$. The following quantity is called the \emph{upper frequency} with which $C$ appears in 
$w$:
$$
\overline{\text{fr}}(C,w):=\limsup_{N\to \infty}\frac{1}{N}\sum_{n\le N}\mathbf{1}_C(S^nw)=\limsup_{N\to \infty}\int_{A^\I} \mathbf{1}_C\ d \delta_{N,w}.
$$
\end{Def}

We will denote by the same letter $S$ the action by the left shift restricted to any closed shift-invariant subset of $A^\I$ (such a subset is called a~\emph{subshift}). In particular, given $w\in A^\I$, we will consider the two following subshifts:
$$
X_w:=\{u\in A^\I\colon \text{ all blocks that appear in }u\text{ also appear in }w\}
$$
and
\begin{multline}\label{zplusem}
X_w^+:=\{u\in A^\I\colon \text{ all blocks that appear in }u\\
\text{ appear in }w\text{ with positive upper frequency}\}.
\end{multline}

Finally, let $F\in C(A^\I)$ be given by
\begin{equation}\label{EFF}
F(w):=w(1)\text{ for }w\in A^\I.
\end{equation}
We will use the same notation $F$, even if the domain of $F$ changes, e.g.\ when we consider a subshift.

\subsubsection{Topological entropy}
Let $T$ be a homeomorphism of a compact metric space $(X,d)$. For $n\in \N$, let
$$
d_n(x,y):=\max\{d(T^ix,T^iy)\colon 0\leq i<n\}.
$$
Given $\vep>0$ and $n\in \N$, let
$$
N(\vep, n)=\max\{|E| : E\subset X,d_n(x,y)\geq \vep\text{ for all }x\neq y \text{ in }E\}.
$$
\begin{Def}[Bowen and Dinaburg]
The \emph{topological entropy} $h_{top}(T)$ is defined as
$$
h_{top}(T)=h_{top}(T,X):=\lim_{\vep\to 0}\left(\limsup_{n\to\infty}\frac{1}{n}\log N(\vep,n) \right).
$$
\end{Def}
We consider now the special case of a subshift, namely, $S\colon X_w\to X_w$, where $w\in A^\I$.
Let
$$
p_n(w):=\left|\{B\in A^n\colon B\text{ appears in }w\}\right|.
$$
and put
$$
h_{top}(w):=\lim_{n\to \infty}\frac{1}{n}\log p_n(w).
$$
Then
\begin{equation}\label{wS}
\text{$h_{top}(w)=h_{top}(S, X_w)$}.
\end{equation}

In a similar way, given $\nu\in \mathcal{P}_S(A^\I)$, we denote by $h_{top}(\text{supp}(\nu))$ the following quantity:
$$
h_{top}(\text{supp}(\nu)):=\lim_{n\to \infty}\frac{1}{n}\log p_n({\text{supp}(\nu)}),
$$
where
$$
p_n({\text{supp}(\nu)}):=\left|\{B\in A^n\colon \nu(B)>0\}\right|.\footnote{It is not hard to see that $p_{m+n}({\text{supp}(\nu)})\leq p_m({\text{supp}(\nu)})\cdot p_n({\text{supp}(\nu)})$.}
$$
In particular, if $\text{Q-gen}(w)=\{\nu\}$, then $h_{top}(\text{supp}(\nu))=h_{top}(S,X_w^+)$ (see Lemma~\ref{lm:f4} below).

\subsubsection{Invariant measures in symbolic dynamical systems}
\begin{Remark}\label{natex}
Any $\nu\in \mathcal{P}_S(A^{\N^\ast})$ is determined by the values it takes on blocks, so it can be extended to a measure in $\mathcal{P}_S(A^\Z)$ taking the same value on each block as $\nu$. This measure will be also denoted by $\nu$.\footnote{The invertible dynamical system $S\colon (A^\Z,\nu)\to (A^\Z,\nu)$ is the natural extension of the non-invertible system $S\colon (A^{\N^\ast},\nu)\to (A^{\N^\ast},\nu)$.} Moreover, if $w\in A^{\N^\ast}$ is quasi-generic for $\nu\in\mathcal{P}_S(A^{\N^\ast})$ along $(N_k)$ then for any $\overline{w}\in A^\Z$ such that $\overline{w}[1,\infty]=w$, the point $\overline{w}$ is quasi-generic for $\nu\in\mathcal{P}_S(A^\Z)$ along $(N_k)$.
\end{Remark}

For any probability distribution $(p_1,\ldots,p_{|A|})$ on $A$, we denote by $B(p_1,\ldots,p_{|A|})$ the corresponding Bernoulli measure on $A^\I$.

The cases $A=\{-1,0,1\}$ or $A=\{0,1\}$ will be of special interest for us. Let $\pi:\{-1,0,1\}^{\I}\to \{0,1\}^{\I}$ be the coordinate square map:
\begin{equation}\label{squaremap}
(\pi(w))_{n}:=w_n^2,
\end{equation}
which is clearly $S$-equivariant.

Given $\nu\in \mathcal {P}_S(\{0,1\}^\I)$, let $\widehat{\nu}$ denote the corresponding {\em relatively independent extension of} $\nu$: for every block $B$, we set
\beq\label{es1}
\widehat{\nu}(B):=2^{-|\text{supp}(B)|}\nu(\pi(B))=2^{-|\text{supp}(B)|}\nu(B^2),
\eeq
where $\text{supp}(B):=\{i: B(i)\neq0\}$ and $B^2(i):=B(i)^2$. Clearly, $\widehat{\nu}\in\mathcal{P}_S(\{-1,0,1\}^\I)$.

\subsubsection{M\"{o}bius function and its generalizations}\label{se:gene}
The following generalization of the M\"{o}bius function $\mob\colon\N^\ast\to\{-1,0,1\}$ defined by~\eqref{def:mob} has been introduced in~\cite{B-Free}. Let $\mathscr{B}=\{b_k\colon k\geq 1\}\subset \{2,3,\dots\}$ be such that $b_k=a_k^2$ and $a_k$, $a_{k'}$ are relatively prime for $k\neq k'$. For $n\in \N$, let
$$
\eta_\mathscr{B}(n):=\begin{cases}
0& \text{ if }b_k|n\text{ for some }k\geq 1,\\
1& \text{ otherwise},
\end{cases}
$$
$$
\delta(n):=\left|\{k\geq 1\colon a_k|n\}\right|
$$
and
\begin{equation}\label{gen-mob}
\mob_\mathscr{B}(n):=(-1)^{\delta(n)}\cdot \eta_\mathscr{B}(n).
\end{equation}
The classical case $\mob$ corresponds to $\mathscr{B}$ being the set of squares of all primes.

\subsubsection{Sturmian sequences}
\begin{Def}
Let $A$ be a finite set. We say that $w\in A^\I$ is a \emph{Sturmian} sequence if $p_n(w)= n+1$ for all $n\in\N^\ast$ (in particular, $|A|=2$, i.e. without loss of generality, $A=\{0,1\}$). If $w$ is Sturmian or periodic, we will say that $w$ is a \emph{generalized Sturmian} sequence.
\end{Def}
\begin{Remark}
Any generalized Sturmian sequence can be obtained in the following way. Consider a line $L$ with an irrational slope in the plane (see Figure~\ref{fig:1} on page~\pageref{fig:1}). We build $w$ by considering the consecutive intersections of $L$ with the integer grid, putting a $0$ each time $L$ intersects a horizontal line and a $1$ each time it intersects a vertical line of the grid (if the line intersects a node, put either $0$ or $1$). In order to include also periodic sequences, we allow the slope of $L$ to be rational, provided that $L$ does not meet any node of the grid.
\end{Remark}
\begin{Remark}\label{stu}
Recall that any (generalized) Sturmian sequence $w$ is generic for a measure $\nu$ of zero entropy. Moreover, $\nu(B)>0$ for any block $B$ appearing in $w$.
\end{Remark}
For more information on Sturmian sequences, we refer the reader e.g.\ to~\cite{Fogg}.

\subsubsection{Toeplitz sequences}\label{dodacnazwe}
\begin{Def}
Let $t\in A^\I$, where $A$ is a finite set. We say that the sequence $t$ is \emph{Toeplitz} if for each $a\in\I$ there exists $r_a$ such that $t(a)=t(a+kr_a)$ for each $k\in\I$.
\end{Def}

Each Toeplitz sequence $t\in A^{\I}$ is obtained as a limit of some periodic sequences defined over the extended alphabet $A\cup\{\ast\}$. Namely, there exists an increasing sequence $(p_n)$, $p_n|p_{n+1}$ such that for each $n\geq1$,
$$
t_n:=T_n^\I,\;\;\lim_{n\to\infty}t_n(j)=t(j)\;\;\mbox{for each}\;j\in\N,
$$
where, for each $n\geq1$, $T_n$ is a block of length $p_n$ over the alphabet $A\cup\{\ast\}$ and $\ast$ at position $k$ at instance $n$ means that $t(k)$ has not been defined at the stage $n$ of the construction\footnote{As an illustration of the definition, consider $A=\{0,1\}$, $\I=\N^\ast$, and set inductively
$T_1:=0\ast$, $T_{2n}:=T_{2n-1}1T_{2n-1}\ast$, $T_{2n+1}:=T_{2n}0T_{2n}\ast$. In this example, $p_n=2^n$. The Toeplitz sequence obtained in this way is not periodic, but it is regular.}.

Whenever
$$
(\mbox{the number of $\ast$ in}\;T_n)/p_n\to 0\text{ when } n\to\infty,$$
we say that $t$ is {\em regular}.
The dynamical systems generated by regular Toeplitz sequences are uniquely ergodic and have zero entropy.

For non-regular Toeplitz sequences the entropy can be positive. Moreover, non-regular Toeplitz sequences can display extremely non-uniquely ergodic behavior.\footnote{Downarowicz~\cite{Do0} proved that each abstract Choquet simplex can be realized as the simplex of invariant measures for a Toeplitz subshift.}

For more information about Toeplitz sequences, we refer the reader to \cite{Do1,Ja-Ke,Wi}.

\section{Ergodic theorem with M\"{o}bius weights}\label{se:sae}
\begin{Prop}
Let $T$ be an automorphism of a standard Borel probability space $(X,\cb,\mu )$ and let $f \in L^1(X,\mathcal{B},\mu)$. Then, for almost every $x \in X$, we have
\[
\frac1{N}\sum_{n\leq N}f(T^nx) \mob (n) \tend{N}{\infty}0.
\]
\end{Prop}
\begin{proof}
We may assume without loss of generality that $T$ is ergodic. Fix $f\in L^2\xbm$. By the Spectral Theorem, we have
\[
\Big\|\frac1{N}\sum_{n\leq N} f(T^nx)\mob (n)\Big\|_2=\Big\|\frac1{N}\sum_{n\leq N} z^n \mob (n)\Big\|_{L^2(\sigma_f)},
\]
where $\sigma_f$ is the spectral measure of $f$.\footnote{Recall that $\sigma_f$ is a finite measure on the circle determined by its Fourier transform given by $\widehat{\sigma}_f(n)=\int f\circ T^n\cdot \overline{f}\ d\mu$, $n\in\Z$.} Hence, by Davenport's estimation~\eqref{vin}, for each $A>0$, we obtain
\begin{equation}
 \label{eq:Dav2}
 \Big\|\frac1{N}\sum_{n\leq N} f(T^nx)\mob (n)\Big\|_2 \leq \frac{C_A}{{\log^{A}N}},
\end{equation}
where $C_A$ is a constant that depends only on $A$. Take $\rho>1$, then for $N=[\rho^m]$ for some $m\ge1$, \eqref{eq:Dav2} takes the form
\[
\Big\|\frac1{N}\sum_{n\leq N}f(T^nx)\mob (n)\Big\|_2 \leq \frac{C_A}{{(m\log(\rho))}^{A}}\text{ for any }A>0.
\]
By choosing $A=2$, we obtain
\[
\sum_{ m \geq 1}\Big\|\frac1{[\rho^m]}\sum_{n\leq [\rho^m]}f(T^nx)\mob (n)\Big\|_2 <+\infty.
\]
In particular, by the triangular inequality for the $L^2$ norm,
\[
 \sum_{ m \geq 1} \left| \frac1{[\rho^m]}\sum_{n\leq [\rho^m]}f(T^nx)\mob (n) \right| \in L^2(X,\mathcal{B},\mu)
\]
and the above sum is almost surely finite. 
Hence, for almost every point $x \in X$, we have
\beq\label{aj}
\frac1{[\rho^m]}\sum_{n\leq [\rho^m]}f(T^nx) \mob (n) \tend{m}{\infty}0.
\eeq
Suppose additionally that $f\in L^\infty(X,\mathcal{B},\mu)$. Then, if $[\rho^m]\leq N < {[\rho^{m+1}]+1}$, we obtain
\begin{align*}
\Big|\frac1{N}\sum_{n\leq N} f(T^nx) \mob (n)\Big|&=\Big| \frac1{N}\sum_{n\leq [\rho^m]}f(T^nx) \mob (n)+ \frac1{N}\sum_{[\rho^m]+1\leq n\leq N}f(T^nx) \mob (n)\Big|\\
&\leq \Big| \frac1{[\rho^m]}\sum_{n\leq [\rho^m]}f(T^nx) \mob (n)\Big|+\frac{\|f\|_{\infty}}{[\rho^m]} (N-[\rho^m])\\
&\leq \Big|\frac1{[\rho^m]}\sum_{n\leq [\rho^m]}f(T^nx) \mob (n)\Big|+\frac{\|f\|_{\infty}}{[\rho^m]} ([\rho^{m+1}]-[\rho^m]).
\end{align*}
Since $\frac{\|f\|_{\infty}}{[\rho^m]} ([\rho^{m+1}]-[\rho^m])\tend{m}{+\infty}\|f\|_{\infty}(\rho-1)$, using~\eqref{aj} and the fact that $\rho$ can be taken arbitrarily close to 1, we obtain
\[
\frac1{N}\sum_{n\leq N} f(T^nx) \mob (n) \tend{N}{\infty}0 \text{ for a.e. }x\in X.
\]
To finish the proof, notice that for any $f \in L^1(X,\mathcal{B},\mu)$, and any $\varepsilon>0$, there exists $g \in L^{\infty}(X,\mathcal{B},\mu)$ such that
$\|f-g\|_1 < \varepsilon$. It follows by the pointwise ergodic theorem that for almost all $x \in X$, we have
\[
\lim_{N \longrightarrow \infty}\Big|\frac1{N}\sum_{n\leq N} (f-g)(T^nx)\Big| < \varepsilon.
\]
Hence,
\begin{multline*}
\limsup_{N \longrightarrow \infty}\Big|\frac1{N}\sum_{n\leq N} f(T^nx)\mob (n)\Big|\leq\\
 \lim_{N \longrightarrow \infty}\Big|\frac1{N}\sum_{n\leq N} (f-g)(T^nx)\Big|+\limsup_{N \longrightarrow \infty}\Big|\frac1{N}\sum_{n\leq N} g(T^nx)\mob (n)\Big|
<\varepsilon.
\end{multline*}
Since $\varepsilon>0$ is arbitrary, the proof is complete.
\end{proof}

\section{The Chowla conjecture vs. the Sarnak conjecture -- abstract approach}\label{se:3}

\subsection{Basic definitions}\label{se:3.1}
We will now introduce the necessary definitions concerning the Chowla conjecture and the Sarnak conjecture in the abstract setting, i.e.\ for arbitrary sequences, not only for $\mob$.

\begin{Def}[cf.\ \cite{Cho,Sarnak}]
We say that $z\in \{-1,0,1\}^\I$ satisfies the condition~\eqref{Ch} if
\begin{equation}\label{Ch}
\frac1N \sum_{n\leq N}z^{i_0}(n) \cdot z^{i_1}(n+a_1)\cdot \ldots \cdot z^{i_r}(n+a_r) \tend{N}{\infty} 0\tag{{\bf Ch}}
\end{equation}
for each choice of $1\leq a_1<\ldots<a_r$, $r\geq 0$, $i_s\in\{1,2\}$ not all equal to~$2$.
\end{Def}

Whenever~\eqref{Ch} is satisfied for $z$, we will also say that $z$ satisfies the Chowla conjecture.

\begin{Def} [cf.\ \cite{Sarnak}]
We say that $z\in\{-1,0,1\}^\I$ satisfies the condition~\eqref{S0} if, for each homeomorphism $T$ of a compact metric space $X$ with $h_{top}(T)=0$, for each $f\in C(X)$ and for each $x\in X$, we have
\begin{equation}\label{S0}
\frac1N\sum_{n\leq N} f(T^nx)z(n)\tend{N}{\infty}0\tag{{\bf S$_0$}}.
\end{equation}
\end{Def}

\begin{Def}
We say that $z\in \{-1,0,1\}^\I$ satisfies the condition~\eqref{S} if,
for each homeomorphism $T$ of a compact metric space $X$,
\begin{equation}\label{S}
\frac1N\sum_{n\leq N} f(T^nx)z(n)\tend{N}{\infty}0\tag{{\bf S}}
\end{equation}
for each $f\in C(X)$ and each $x\in X$ that is completely deterministic.
\end{Def}

Whenever~\eqref{S} is satisfied for $z$, we will also say that $z$ satisfies the Sarnak conjecture.

Note that by the variational principle, see e.g.\ \cite{Wa}, if the topological entropy of $T$ is zero, then all points are completely deterministic. Hence \eqref{S} implies \eqref{S0}.

\subsection{About \eqref{Ch}}\label{se:onCh}
Fix $z\in \{-1,0,1\}^{\N^\ast}$. Suppose that $z^2$ is quasi-generic for $\nu$ along $(N_k)$, i.e.\ we have
\beq\label{zbieznosc}
\delta_{N_k,z^2}:=\frac1{N_k}\sum_{n\leq N_k}\delta_{S^n z^2}\tend{k}{\infty} \nu\in \mathcal{P}_S(X_{z^2}).
\eeq
\begin{Remark}
In the classical situation $z=\mob$, $z^2$ is generic for the Mirsky measure~\cite{Mi}, cf.~\cite{Ce-Si,Sarnak}. Moreover, the Mirsky measure on $X_{z^2}$ has full topological support, cf.~\eqref{cond}. In a more general framework, similar results hold for so called $\mathscr{B}$-free systems, see \cite{B-Free}.
\end{Remark}

Recall that the function $F$ was given by the formula~\eqref{EFF}, i.e.\ $F(w)=w(1)$.
\begin{Lemma}\label{jaf}
Let $1\leq a_1<\ldots<a_r$, $r\geq 0$ and $i_s\in\{1,2\}$, $0\leq s\leq r$. Then the following equalities hold:
$$
\int_{\{-1,0,1\}^{\N^\ast}}F^{i_0}\cdot F^{i_1}\circ S^{a_1}\cdot\ldots\cdot F^{i_r}\circ S^{a_r}\ d\widehat{\nu}=0,
$$
when not all $i_s$ are equal to~$2$.\footnote{Recall that $\widehat{\nu}$ was defined in \eqref{es1}.} Moreover,
$$
\int_{\{-1,0,1\}^{\N^\ast}}F^{2}\cdot F^{2}\circ S^{a_1}\cdot\ldots\cdot F^{2}\circ S^{a_r}\ d\widehat{\nu}=\int_{\{0,1\}^{\N^\ast}}F\cdot F\circ S^{a_1}\cdot\ldots\cdot F \circ S^{a_r}\ d\nu.
$$
\end{Lemma}
\begin{proof}
The assertion follows directly by the calculation:
\begin{align*}
\int_{\{-1,0,1\}^{\N^\ast}}&F^{i_0}\cdot F^{i_1}\circ S^{a_1}\cdot\ldots\cdot F^{i_r}\circ S^{a_r}\ d\widehat{\nu}\\
&=\sum_{j_0,j_1,\ldots,j_r=\pm1}j_0^{i_0}\cdot j_1^{i_1}\cdot\ldots\cdot j_r^{i_r}\\
&\cdot \widehat{\nu}\left(\left\{y\in\{-1,0,1\}^{\N^\ast}\colon (y(1),y({1+a_1}),\ldots,y(1+a_r))=(j_0,j_1,\ldots,j_r)\right\}\right)\\
&= \Big(\sum_{j_0,j_1,\ldots,j_r=\pm1}j_0^{i_0}\cdot j_1^{i_1}\cdot\ldots\cdot j_r^{i_r}\Big)\\
&\cdot\frac1{2^{r+1}}\nu\left(\left\{u\in\{0,1\}^{\N^\ast}\colon u(1)=u(1+a_1)=\ldots=u(1+a_r)=1\right\}\right).
\end{align*}
\end{proof}

\begin{Lemma}[cf.\ \cite{Sarnak} for $\mob$]\label{rownow}
Let $(N_k)$ be such that~\eqref{zbieznosc} holds. Then
\beq\label{m:chs}
\frac1{N_k}\sum_{n\leq N_k} z^{i_0}(n)\cdot z^{i_1}(n+a_1)\cdot\ldots\cdot z^{i_r}(n+a_r)\tend{k}{\infty}0
\eeq
for each choice of $1\leq a_1<\ldots<a_r$, $r\geq 0$, $i_s\in\{1,2\}$ not all equal to~$2$, if and only if
\beq\label{zbieznosc1}
\delta_{N_k,z}\tend{k}{\infty} \widehat{\nu}.
\eeq
\end{Lemma}
\begin{proof}
Note that, for each $k\geq 1$,
\begin{multline}\label{faj}
\frac1{N_k}\sum_{n\leq N_k}z^{i_0}(n)\cdot z^{i_1}(n+a_1)\cdot\ldots\cdot z^{i_r}(n+a_r)\\
=\frac1{N_k}\sum_{n\leq N_k}\big(F^{i_0}\cdot F^{i_1}\circ S^{a_1}\cdot\ldots\cdot F^{i_r}\circ S^{a_r}\big)(S^{n-1}z).
\end{multline}
Suppose that~\eqref{zbieznosc1} holds. Then it follows from~\eqref{faj} that
\begin{multline*}
\frac1{N_k}\sum_{n\leq N_k}z^{i_0}(n)\cdot z^{i_1}(n+a_1)\cdot\ldots\cdot z^{i_r}(n+a_r)\\
\tend{k}{\infty}\int_{\{-1,0,1\}^{\N^\ast}}F^{i_0}\cdot F^{i_1}\circ S^{a_1}\cdot\ldots\cdot F^{i_r}\circ S^{a_r}\ d\widehat{\nu}.
\end{multline*}
Therefore, in view of Lemma~\ref{jaf}, we obtain~\eqref{m:chs}.

Suppose now that~\eqref{m:chs} holds. Without loss of generality, we may assume that
\beq\label{zbieznosc2}
\delta_{N_k,z}\tend{k}{\infty} \rho.
\eeq
In view of~\eqref{faj}, this implies
\begin{multline*}
\frac1{N_k}\sum_{n\leq N_k}z^{i_0}(n)\cdot z^{i_1}(n+a_1)\cdot\ldots\cdot z^{i_r}(n+a_r)\\
\tend{k}{\infty} \int_{\{-1,0,1\}^{\N^\ast}}F^{i_0}\cdot F^{i_1}\circ S^{a_1}\cdot\ldots\cdot F^{i_r}\circ S^{a_r} \ d\rho.
\end{multline*}
It follows from~\eqref{m:chs} that
\begin{equation}\label{czy1}
\int_{\{-1,0,1\}^{\N^\ast}}F^{i_0}\cdot F^{i_1}\circ S^{a_1}\cdot\ldots\cdot F^{i_r}\circ S^{a_r} \ d\rho=0,
\end{equation}
whenever not all $i_t$ are equal to $2$. Moreover, since $F^2(u)=F(u^2)$ for any $u\in \{-1,0,1\}^{\N^\ast}$, we deduce from~\eqref{zbieznosc} that
\begin{equation}\label{czy2}
\int_{\{-1,0,1\}^{\N^\ast}}F^{2}\cdot F^{2}\circ S^{a_1}\cdot\ldots\cdot F^{2}\circ S^{a_r} \ d\rho=\int_{\{0,1\}^{\N^\ast}}F\cdot F \circ S^{a_1}\cdot\ldots\cdot F \circ S^{a_r} \ d\nu.
\end{equation}
In view of Lemma~\ref{jaf},~\eqref{czy1} and~\eqref{czy2}, we have
$$
\int_{\{-1,0,1\}^{\N^\ast}}G\ d\widehat{\nu}=\int_{\{-1,0,1\}^{\N^\ast}}G\ d\rho
$$
for any
$$
G\in \mathcal{A}:=\{F^{i_0}\cdot F^{i_1}\circ S^{a_1}\cdot\ldots\cdot F^{i_r}\circ S^{a_r}\colon 1\leq a_1<\dots<a_r, r\geq 0, i_s\in \N\}.
$$
Since $\mathcal{A}\subset C(\{-1,0,1\}^{\N^\ast})$ is closed under taking products and separates points, we only need to use the Stone-Weierstrass theorem to conclude that $\rho=\widehat{\nu}$.
\end{proof}

\label{porba}
The above lemma can be also viewed from the probabilistic point of view. Indeed, let ${(X_n)}_{n\geq 1}$ (or ${(X_n)}_{n\in\Z}$) be a stationary sequence of random variables taking values in $\{-1,0,1\}$. Notice that whenever
\begin{equation}\label{echowla2}
\mathbb{P}\big(\{X_{a_1}=j_1,\ldots,X_{a_r}=j_r\}\big)\\
=\frac1{2^{k}}\mathbb{P}\big(\{X^2_{a_1}=j^2_1,\ldots,X^2_{a_r}=j^2_r\}\big),
\end{equation}
for each choice of $1\leq a_1<\ldots<a_r$ and $j_s\in\{-1,0,1\}$, where $k:=|\{s\in\{1,\ldots,r\}: j_s\neq0\}|$, then
\beq\label{echowla1}
\E ( X_{a_1}^{i_1}\cdot\ldots\cdot X_{a_r}^{i_r})=0
\eeq
for each choice of $1\leq a_1<\ldots<a_r$, $r\geq0$, $i_s\in\{1,2\}$ not all equal to~$2$ (the proof is the same as the one of Lemma~\ref{jaf} with notational changes only).\footnote{Condition~\eqref{echowla2} means that the distribution of the process ${(X_n)}_{n\geq1}$ is the relatively independent extension of the distribution of the (stationary) process ${(X_n^2)}_{n\geq1}$.}

In fact, the following holds:
\begin{Lemma}\label{prob}
Conditions~\eqref{echowla2} and~\eqref{echowla1} are equivalent.
\end{Lemma}
\begin{proof}
We have already seen that~\eqref{echowla2} implies~\eqref{echowla1}. Let us show the converse implication. In other words, we need to show that there exists at most one stationary process (that is, at most one $S$-invariant distribution on $\{-1,0,1\}^{\N^\ast}$) such that~\eqref{echowla1} holds. However, each stationary process $(X_n)$ is entirely determined by the family
$$
\{\E(\exp(i\sum_{j=1}^{n}t_jX_j))\colon n\geq1, (t_0,\ldots,t_{n-1})\in\R^n\}.
$$
Since $\E(\exp(i\sum_{j=1}^{n}t_jX_j))=
\sum_{k=0}^\infty\frac{i^k}{k!}\E\left(\sum_{j=1}^{n}t_jX_j\right)^k$, the result follows.
\end{proof}

 As the proof shows, the above lemma can be proved in a more general framework, namely, for stationary processes having moments of all orders.

\begin{Remark}\label{re:ch:4}
It follows immediately from Lemma~\ref{rownow} that each of the following conditions is equivalent to~\eqref{Ch}:
\begin{itemize}
	\item
	$\text{Q-gen}(z)=\left\{\widehat{\nu}\colon\nu\in \text{Q-gen}(z^2)\right\}$;
	\item
	$\delta_{N_k,z^2}\tend{k}{\infty}\nu$ if and only if $\delta_{N_k,z} \tend{k}{\infty}\widehat\nu$.
\end{itemize}
\end{Remark}
Now, we can completely characterize sequences $z\in\{-1,1\}^{\N^\ast}$ satisfying~\eqref{Ch}.
\begin{Prop}\label{lambda}
The only sequences $u\in\{-1,1\}^{\N^\ast}$ satisfying~\eqref{Ch} are generic points for the Bernoulli measure $B(1/2,1/2)$.\end{Prop}
\begin{proof}
Notice that $u^2$ is the generic point for the Dirac measure at $(1,1,...)$ and by Lemma~\ref{rownow}, $u$ is a generic point for the relatively independent extension of that Dirac measure, which is the Bernoulli measure $B(1/2,1/2)$.
\end{proof}

\subsection{\eqref{Ch} implies~\eqref{S}}\label{se:Ch-to-S}
In this section, we will provide a dynamical proof of the following theorem:
\begin{Th}[Sarnak]\label{ChS}
\eqref{Ch} implies~\eqref{S}.
\end{Th}
\begin{Remark}
In particular,~\eqref{Ch} implies~\eqref{S0} (see~\cite{Sarnak}), which has already been proved by Sarnak.
The proof of the implication~\eqref{Ch} $\implies$ \eqref{S} given below is to be compared with Sarnak's arguments on page~9 of~\cite{Sarnak}. Later, in Theorem~\ref{thm:SequivS0}, we show that~\eqref{S} and~\eqref{S0} are equivalent. Hence, another way to prove Theorem~\ref{ChS} is to use \eqref{Ch} $\implies$ \eqref{S0} and~\eqref{S} $\iff$~\eqref{S0}.
\end{Remark}

Fix some $\nu\in\mathcal{P}_S(\{0,1\}^\Z)$.

\begin{Lemma}\label{l1} The dynamical system $(S, \{-1,0,1\}^\Z,\widehat{\nu})$ is a factor of
$$
(S,\{0,1\}^\Z,\nu)\times(S,\{-1,1\}^{\Z}, B(1/2,1/2)).
$$
\end{Lemma}
\begin{proof}
It suffices to notice that, for $\xi\colon \{0,1\}^\Z\times\{-1,1\}^\Z\to \{-1,0,1\}^\Z$ given by
$$
\xi(w,u)(n):=w(n)\cdot u(n),
$$
we have
$$
\xi_\ast(\nu\ot B(1/2,1/2))=\widehat{\nu},
$$
which is straightforward by the definition of $\widehat{\nu}$.
\end{proof}

\begin{Lemma} \label{l2}
The extension $(S,\{-1,0,1\}^\Z,\widehat{\nu})\stackrel{\pi}{\to} (S,\{0,1\}^\Z,\nu)$ is either trivial (i.e.\ 1-1 a.e.) or relatively~K.\footnote{Recall that $\pi$ was defined in~\eqref{squaremap}.}
\end{Lemma}
\begin{proof}
Notice that since the extension
$$
(S,\{0,1\}^\Z,\nu)\times (S,\{-1,1\}^{\Z},B(1/2,1/2))\to (S,\{0,1\}^\Z,\nu)
$$
is relatively~K, so is any nontrivial intermediate factor (over $(S,\{0,1\}^\Z,\nu)$). To see that
$(S,\{-1,0,1\}^\Z,\widehat{\nu})$ is an intermediate factor, by the proof of Lemma~\ref{l1}, all we need to check is that $\pi\circ\xi$ equals to the projection on the first coordinate. The latter follows from the equality
$w=(w\cdot u)^2$ which holds for $w\in \{0,1\}^\Z$ and $u\in\{-1,1\}^\Z$.
\end{proof}

\begin{Remark}\label{jedyna}
It is possible that the extension $(S,\{-1,0,1\}^\Z,\widehat{\nu})\stackrel{\pi}{\to} (S,\{0,1\}^\Z,\nu)$ is trivial. In fact, it happens only if $\nu=\delta_{(\ldots,0,0,0,\ldots)}$.
For, suppose that $Y\subset \{-1,0,1\}^\Z$, $\widehat{\nu}(Y)=1$ is such that $\pi|_Y$ is 1-1. Fix a block $B\in\{0,1\}^k$ with $\nu(B)>0$. Then the set $\{x\in \pi(Y):x(n)=B(n),\;n=0,1,\ldots,k-1\}$ is of positive $\nu$-measure, and $|\pi^{-1}(x)\cap Y|\geq2^{\text{supp}(B)}$ as each block $C\in\{-1,0,1\}^k$, $C^2=B$, has positive $\widehat{\nu}$-measure (whence $\widehat{\nu}(Y\cap C)>0$). It follows immediately that the support of $B$ has to be empty.
\end{Remark}

\begin{Lemma}\label{l3}
$\E^{\widehat\nu}(F|\pi(w)=u)=0$ for $\nu$-a.e.\ $u\in\{0,1\}^{\Z}$.
\end{Lemma}
\begin{proof}
We have
$$
\E^{\widehat\nu}(F|\pi(w)=u)=\E^{\widehat\nu}(F|\{0,1\}^\Z)(u)= \int_{\pi^{-1}(u)} F\, d\, \widehat\nu_u,
$$
where $\widehat\nu_u$ denotes the relevant conditional measure in the disintegration of $\widehat\nu$ over $\nu$. Notice that $\widehat\nu_u$ is the product measure $(1/2,1/2)$ of all positions belonging to the support of $u$. If $u(1)=0$ then the formula holds. If $u(1)=1$ then $F$ on $\pi^{-1}(u)$ takes two values $\pm1$ with the same probability, so the integral is still zero.
\end{proof}

\begin{Lemma}\label{l5}
Let $T$ be a homeomorphism of a compact metric space $X$, let $x\in X$ be completely deterministic, and suppose that $z$ is a quasi-generic point for $\widehat{\nu}$ along the sequence $(N_k)$. Assume that
\begin{equation}\label{4a}
\delta_{T\times S, N_k, (x,z)}\to \rho
\end{equation}
weakly in $\mathcal{P}_{T\times S}(X\times \{-1,0,1\}^\Z)$. Then:
\begin{enumerate}[label=(\alph*)]
\item $\rho$ is a joining of $(T,X,\kappa)$ and $(S,\{-1,0,1\}^\Z,\widehat\nu)$ for some zero entropy measure $\kappa\in \text{Q-gen}(x)$;
\item the factors $(T,X,\kappa)\vee (S,\{0,1\}^\Z,\nu)$ and $(S,\{-1,0,1\}^\Z,\widehat{\nu})$ are relatively independent over $(S,\{0,1\}^\Z,\nu)$ as factors of $(T\times S,X\times \{-1,0,1\}^\Z,\rho)$.\label{Bb}
\end{enumerate}
\end{Lemma}
\begin{proof}
It follows from~\eqref{4a} that
$$
\kappa:=\rho|_X=\lim_{k\to\infty}\delta_{T,N_k,x},
$$
and $h(T,\kappa)=0$ since $x$ is completely deterministic. Hence $\rho$ is a joining of $(T,X,\kappa)$ and $(S,\{-1,0,1\}^\Z,\widehat\nu)$, and the extension
$$
(T,X,\kappa)\vee (S,\{0,1\}^\Z,\nu)\to (S,\{0,1\}^\Z,\nu)
$$
has relative entropy zero (by the Pinsker formula, see e.g.\ \cite{Parry}, Theorem 6.3).
On the other hand, by Lemma~\ref{l2}, the extension
$$
(S,\{-1,0,1\}^\Z,\widehat\nu) \to (S,\{0,1\}^\Z,\nu)
$$
is relatively K.
To complete the proof, we only need to use Remark~\ref{Thouvenot}.
\end{proof}

\begin{proof}[Proof of Theorem~\ref{ChS}]
Assume that $z\in\{-1,0,1\}^{\N^\ast}$ satisfies~\eqref{Ch}, let $T$ be a homeomorphism of the compact metric space $X$, and let $x\in X$ be a completely deterministic point. Fix $(N_k)$ such that
\beq\label{k2}
\delta_{T\times S, N_k,(x,z)}\tend{k}{\infty} \rho
\eeq
for some measure $\rho$. Then by Remark~\ref{re:ch:4}, the projection of $\rho$ onto the second coordinate is of the form $\widehat\nu$ for some $\nu\in \text{Q-gen}(z^2)$. Take a function $f\in C(X)$. It follows from~\eqref{k2} that
\begin{equation}\label{w}
\frac1{N_k}\sum_{n\leq N_k}f(T^nx)z(n)=\frac1{N_k}\sum_{n\leq N_k}f(T^nx)F(S^nz)\tend{k}{\infty} \int f\ot F\ d\rho.
\end{equation}
 Using Lemma~\ref{l3}, we have
\begin{equation}\label{w1}
\E^\rho(F|\{0,1\}^\Z)=\E^{\widehat{\nu}}(F|\{0,1\}^\Z)=0.
\end{equation}
By this and using also Lemma~\ref{l5} (\textit{b}), we obtain
$$
\E^\rho(f\ot F|\{0,1\}^\Z) = \E^\rho(f|\{0,1\}^\Z)\ \E^\rho(F|\{0,1\}^\Z)= 0.
$$
This yields
$
\int f\ot F\ d\rho=0$.
\end{proof}

\subsection{\eqref{Ch},~\eqref{S0-strong} and~\eqref{S-strong} are equivalent}\label{se:CHSS0}

In this section, we will throw some more lights on Theorem~\ref{ChS}, by considering some strengthening of properties of~\eqref{S}-type.

\begin{Def}
A sequence $z\in \{-1,0,1\}^\I$ is said to satisfy the condition \eqref{S0-strong} if for each homeomorphism $T$ of a compact metric space $X$, with $h_{top}(T)=0$, we have
\begin{equation}\label{S0-strong}
\frac1N\sum_{n\leq N}f(T^nx)z^{i_0}(n)\cdot z^{i_1}(n+a_1)\cdot\ldots\cdot z^{i_r}(n+a_r)\tend{N}{\infty}0\tag{{\bf S$_0$-strong}}
\end{equation}
for each $f\in C(X)$, each $x\in X$ and each choice of $1\leq a_1<\ldots<a_r$, $r\geq 0$, $i_s\in\{1,2\}$ not all equal to~$2$.
\end{Def}

\begin{Def}
A sequence $z\in \{-1,0,1\}^\I$ is said to satisfy the condition \eqref{S-strong} if for each homeomorphism $T$ of a compact metric space $X$, we have
\begin{equation}\label{S-strong}
\frac1N\sum_{n\leq N}f(T^nx)z^{i_0}(n)\cdot z^{i_1}(n+a_1)\cdot\ldots\cdot z^{i_r}(n+a_r)\tend{N}{\infty}0\tag{{\bf S-strong}}
\end{equation}
for each $f\in C(X)$, each completely deterministic $x\in X$ and each choice of $1\leq a_1<\ldots<a_r$, $r\geq 0$, $i_s\in\{1,2\}$ not all equal to~$2$.
\end{Def}
If the above holds, we will also say that $z$ satisfies the strong Sarnak conjecture. In particular, for $z=\mob$ the strong Sarnak conjecture takes the form
$$\label{strongS}
\frac1N\sum_{n\leq N}f(T^nx)\mob^{i_0}(n)\cdot\mob^{i_1}(n+a_1)
\cdot\ldots\cdot\mob^{i_r}(n+a_r)\tend{N}{\infty}0
$$
for $f,T,x,r,a_s,i_s$ as above.

\begin{Prop}\label{ChS1}
The conditions~\eqref{Ch},~\eqref{S0-strong} and~\eqref{S-strong} are equivalent.
\end{Prop}

For the proof, we will need the following result.
\begin{Lemma}\label{lemat1}
Let $z\in \{-1,0,1\}^\I$ and let
$$
u(n):=z^{i_0}(n)\cdot z^{i_1}(n+a_1)\cdot\ldots\cdot z^{i_r}(n+a_r),\;n\in \I,
$$
for some natural numbers $1\leq a_1<a_2<\ldots<a_r$ and $i_s\in\{1,2\}$.
Then the following holds:
\begin{enumerate}[label=(\alph*)]
\item\label{p:a}
If $z$ satisfies~\eqref{Ch} then $u$ satisfies~\eqref{Ch} provided that not all $i_s$ are equal to~$2$.
\item\label{p:c}
If $z$ is completely deterministic, then so is $u$.\footnote{In particular, this holds, if we replace $z$ with $z^2$ and $u$ with $u^2$. Note in passing that we can have $z$ satisfying~\ref{p:a} while $z^2$ satisfies~\ref{p:c}.}
\end{enumerate}

\end{Lemma}
\begin{proof} We write
\begin{multline*}
\frac1N\sum_{n\leq N}u^{j_0}(n)\cdot u^{j_1}(n+b_1)\cdot\ldots\cdot u^{j_t}(n+b_t)\\
=\frac1N\sum_{n\leq N}\prod_{\alpha=0}^r\prod_{\beta=0}^t z^{i_\alpha j_\beta}(n+a_\alpha+b_\beta)
\end{multline*}
with $a_0=b_0=0$. Consider then the smallest $\alpha$ and $\beta$ such that $i_\alpha=j_\beta=1$. Since both sequences $(a_i)$, $(b_j)$ are strictly increasing, the sum $a_\alpha+b_\beta$ can be obtained only as
$a_\gamma+b_\delta$ with either $a_\gamma<a_\alpha$ or $b_\delta<b_\beta$. It follows that, in the above sum, the
term $z(n+a_\alpha+b_\beta)$ appears with the power
$$
i_\alpha j_\beta+\text{ even number},
$$
that is, an odd power, which completes the proof of part~\ref{p:a} of the lemma.

We will show now that assertion~\ref{p:c} also holds (see~\cite{Ka}, Lemma~5.1). Suppose that
$$
\delta_{S,N_k,u}\to \rho
$$
and consider the following sequence of measures on $\{-1,0,1\}^\I\times \dots \times \{-1,0,1\}^\I$:
\begin{align*}
\widetilde{\rho}_k:=&\frac{1}{N_k}\sum_{n\leq N_k}\delta_{S^nz }\otimes\delta_{S^nS^{a_1}z }\otimes \dots\otimes \delta_{S^nS^{a_r}z }\\
=&\frac{1}{N_k}\sum_{n\leq N_k}\delta_{\big(\underbrace{S\times \ldots\times S}_{r+1}\big)^n(z ,S^{a_1}z ,\dots,S^{a_r}z )}, k\geq 1.
\end{align*}
Passing to a subsequence if necessary, we may assume that $\widetilde{\rho}_k$ converges to $\widetilde{\rho}$. Then $\widetilde{\rho}$ is a joining of $(S,\kappa_0),(S,\kappa_1),\ldots, (S,\kappa_r)$,
where $\kappa_s\in \text{Q-gen}(z )$ for $0\leq s\leq r$. Hence $h(S,\kappa_s)=0$ for $0\leq s\leq r$ and it follows that $h(S^{\times(r+1)},\widetilde{\rho})=0$. Notice that
$S\colon (\{-1,0,1\}^\I,\rho)\to (\{-1,0,1\}^\I,\rho)$
is a factor of
$$
S^{\times(r+1)}\colon ((\{-1,0,1\}^\I)^{\times (r+1)},\widetilde{\rho})\to ((\{-1,0,1\}^\I)^{\times (r+1)},\widetilde{\rho}),
$$
with the factoring map $(x_0,\dots,x_r)\mapsto x_0\cdot\ldots\cdot x_r$. Therefore, we obtain $h(S,\rho)=0$ and the assertion follows.
\end{proof}
\begin{Remark}
Part~\ref{p:c} of Lemma~\ref{lemat1} will not be used in this section. We will need it later, in the proof of Proposition \ref{pr:2a}.
\end{Remark}

\begin{proof}[Proof of Proposition~\ref{ChS1}]
Since clearly~\eqref{S-strong} implies~\eqref{S0-strong}, which, in turn, implies~\eqref{Ch}, it suffices to show that~\eqref{Ch} implies~\eqref{S-strong}. This however follows immediately from Theorem~\ref{ChS} and Lemma~\ref{lemat1}.
\end{proof}

Moreover, in view of Proposition~\ref{ChS1} and Propositon~\ref{lambda}, we immediately obtain the following:
\begin{Cor}\label{lambda1}
If~\eqref{Ch} holds for the Liouville function $\lamob$ then
for each homeomorphism $T$ of a compact metric space $X$, we have
$$
\frac1N\sum_{n\leq N}f(T^nx)\lamob(n)\cdot \lamob(n+a_1)\cdot\ldots\cdot \lamob(n+a_r)\tend{N}{\infty}0
$$
for each $f\in C(X)$, each completely deterministic $x\in X$ and for each choice of $1\leq a_1<a_2<\ldots<a_r$, $r\geq0$.\footnote{Clearly, since $\lamob$ takes values in $\{-1,1\}$, we can remove the exponents $i_s$ appearing originally in the condition~\eqref{Ch}.}
\end{Cor}

\begin{Remark}
Since the Bernoulli shifts are disjoint from all zero entropy transformations, arguments similar to those used in the proof of Theorem~\ref{ChS}, together with Lemma~\ref{lemat1}, can be used to obtain another proof of Corollary~\ref{lambda1}.
\end{Remark}

\subsection{\eqref{S0} and~\eqref{S} are equivalent}\label{se:SS0}

\newcommand{\ou}{\overline{u}}
\newcommand{\onu}{\overline{\nu}}
The purpose of this section is to prove the following result.
\begin{Th}
\label{thm:SequivS0}
Properties \eqref{S0} and \eqref{S} are equivalent.
\end{Th}

The first part of the proof deals with the symbolic case and shows that if a sequence $u$ is quasi-generic for some shift-invariant measure of zero entropy, then $u$ can be well approximated by a sequence that has zero topological entropy. In~\cite{We2}, the following characterization of completely deterministic points was stated without a proof:
\begin{quote}
 A sequence $u$ is completely deterministic if and only if, for any $\vep>0$ there exists $K$ such that, after removing from $u$ a subset of density less than $\vep$, what is left can be covered by a collection $\mathcal{C}$ of $K$-blocks
 such that $|\mathcal{C}|<2^{\vep K}$.
\end{quote}
The following lemma is a reformulation of this criterion in a language suitable for our needs.

\begin{Lemma}
 \label{lemma:approach-deterministic-sequences-step1}
 Let $ A$ be finite nonempty set, and let $(N_k)_{k\ge1}$ be an increasing sequence of integers, with $N_k|N_{k+1}$ for each $k$. Assume that $u\in A^{\N^\ast}$ satisfies
 \begin{equation}
  \label{eq:convergence-to-nu}
  \delta_{N_k,u}=\frac{1}{N_k}\sum_{n\le N_k} \delta_{S^nu}\tend{k}{\infty} \nu,
 \end{equation}
 where $\nu$ is such that $h(S,\nu)=0$.

 Then, for any $\varepsilon>0$, we can find an arbitrarily large integer $k$ and a map $\varphi: A^{N_k}\to A^{N_k}$, satisfying the following properties:
 \begin{itemize}
  \item $\left|\varphi\bigl( A^{N_k}\bigr)\right|<2^{\varepsilon N_k}$;
  \item the sequence $\ou$ obtained from $u$ by replacing, for each $j\ge 0$, the block $u|_{jN_k+1}^{(j+1)N_k}$ by its image by $\varphi$, is such that for each $s\ge0$,
  \begin{equation}
   \label{eq:approach_u_step1}
   \dfrac{1}{N_{k+s}}\left|\bigl\{1\le n\le N_{k+s}: u_n\neq\ou_n\bigr\}\right| < \varepsilon;
  \end{equation}
 \item the first symbol occuring in $\ou$ is the same as in $u$.
 \end{itemize}
\end{Lemma}

We will need the following lemma taken from~\cite{Shields}, (Lemma~1.5.4 p.~52).
\begin{Lemma}
\label{lemma:Shields}
For $0<\delta<1$, set
\[
 H(\delta):=-\delta\log\delta -(1-\delta)\log(1-\delta).
\]
Then, for any integer $N\ge1$ and any $0<\delta\le 1/2$,
\[
 \sum_{k\le\delta N}{N\choose k} \le 2^{N H(\delta)}.
\]
\end{Lemma}

\begin{proof}[Proof of Lemma~\ref{lemma:approach-deterministic-sequences-step1}]
 Let $P$ be the finite partition of $A^{\N^\ast}$ determined by the values of the first coordinate. Then $\bigvee_{j=0}^{n-1}S^{-j}P$ is the partition of $A^{\N^\ast}$ according to the $n$-block appearing in coordinates from 1 to $n$. Since the entropy of $(S,\nu)$ vanishes, given an arbitrary $\delta>0$, we can take $n$ large enough so that
 \begin{equation}
 \label{eq:small-entropy-n}
  \frac{1}{n} H_\nu\left(\bigvee_{j=0}^{n-1}S^{-j}P\right) <\delta.
 \end{equation}
 Now, let us say that an $n$-block is \emph{heavy} if the $\nu$-measure of the corresponding cylinder set is larger than $2^{-\varepsilon n}$, and say it is \emph{light} otherwise. We claim that the $\nu$-measure of the union of all light $n$-blocks is arbitrarily small whenever $\delta$ is chosen small enough. Indeed, for any light $n$-block $B$, we have
 \begin{equation}
  \nu(B)\log \nu(B) \leq \nu(B)\log\frac{1}{2^{n\vep}}=-\nu(B) \cdot n\vep.
 \end{equation}
This and~\eqref{eq:small-entropy-n} imply
 \begin{equation*}
  \vep n\sum_{\text{light }n\text{-blocks }B} \nu(B) \leq -\sum_{\text{light }n\text{-blocks }B} \nu(B) \log(B)<\delta n,
 \end{equation*}
 which gives
 $$
 \sum_{\text{light }n\text{-blocks }B}\nu(B)<\frac{\delta}{\vep}.
 $$
 Observe also that the number of heavy $n$-blocks cannot exceed $2^{\varepsilon n}$.

 Say that an integer $j\ge1$ is \emph{good in $u$} if the $n$-block $u|_j^{j+n-1}$ is heavy. By~\eqref{eq:convergence-to-nu} (applied to the characteristic function of the union of all light $n$-blocks), and assuming $\delta$ is small enough, we can take $k$ large enough so that, for each $s\ge0$,
 \begin{equation}
  \label{eq:proportion_of_good_j}
  \dfrac{1}{N_{k+s}} \left| \{1\le j\le N_{k+s}: j\text{ is not good in }u \}\right| < \varepsilon^2/2.
 \end{equation}
 We can also assume that $k$ is large enough so that
 \begin{equation}
  \label{eq:end_of_Nk_blocks}
  \dfrac{n}{N_k}<\frac{\varepsilon}{2}.
 \end{equation}
 Let us now define the map $\varphi\colon A^{N_k}\to A^{N_k}$. Let $W\in A^{N_k}$; we say that $j\in\{1,\ldots,N_k\}$ is \emph{good in $W$} if $j+n-1\le N_k$ and the $n$-block $W|_j^{j+n-1}$ is heavy. We say that $W$ is \emph{acceptable} if the proportion of $j\in\{1,\ldots,N_k\}$ which are good in $W$ is larger than $1-\varepsilon$. The definition of $\varphi(W)$ will depend on whether $W$ is acceptable or not. If $W$ is not acceptable, then we simply set $\varphi(W):=a^{N_k}$, where $a\in A$ is the first symbol occuring in the sequence $u$. If $W$ is acceptable, then we run the following algorithm. Let $j_1$ be the first integer which is good in $W$, and inductively, define $j_{i+1}$ as the smallest integer larger than or equal to $j_i+n$ which is good in $W$, provided such an integer exists. This algorithm outputs a finite list of integers $j_1,\ldots,j_r$ which are good in $W$, such that $j_i+n\leq j_{i+1}$, and such that the disjoint heavy $n$-blocks $W|_{j_i}^{j_i+n-1}$, $1\le j\le r$, cover a
proportion at least $1-\varepsilon$ of $W$ (because symbols which are not covered correspond to integers which are not good in $W$). Then, in $W$, replace by $a$ all symbols which are not covered by these heavy $n$-blocks, and define $\varphi(W)$ as the resulting $N_k$-block.

 The number of $N_k$-blocks which are images of some acceptable block $W$ by this procedure is bounded by the number of choices for the subset of $\{1,\ldots,N_k\}$ where we put the letter $a$, times the number of choices for the heavy blocks. The former is bounded by the number of subsets of $\{1,\ldots,N_k\}$ which have less than $\varepsilon\, N_k$ elements, which is at most $2^{H(\varepsilon)N_k}$ by Lemma~\ref{lemma:Shields}.
Since the number of heavy blocks is at most $2^{\varepsilon n}$, the latter is bounded by $(2^{\varepsilon n})^r$, which is less than $2^{\varepsilon N_k}$ (indeed, $nr\le N_k$ because in $W$ we see $r$ disjoint heavy blocks of length $n$).

Observe that, by the construction of $\varphi$ and by the choice of $a$, the first symbol in $\ou$ is the same as in $u$.

Now, it only remains to show that~\eqref{eq:approach_u_step1} holds. Let $s\ge 0$. Each $m\in\{0,\ldots,N_{k+s}/N_k-1\}$ such that $u|_{m N_k+1}^{(m+1)N_k}$ is not acceptable gives rise in the corresponding subblock to at least $\varepsilon N_k$ integers $j$ which are not good \emph{in this subblock}. But there are two reasons why this could happen:
\begin{itemize}                                                                                                                                                                                 \item either $j$ is one of the last $n$ positions of the subblock, which by~\eqref{eq:end_of_Nk_blocks} only concerns a number of integers bounded by $\varepsilon N_k/2$,
\item or $j$ is not good in $u$, which therefore concerns at least $\vep N_k/2$ integers $j$ in this subblock.                                                                                                                                                    \end{itemize}
Then, \eqref{eq:proportion_of_good_j} ensures that the proportion of integers $m\in\{0,\ldots,N_{k+s}/N_k-1\}$ such that $u|_{m N_k+1}^{(m+1)N_k}$ is not acceptable is less than $\varepsilon$. Moreover, observe that if $W$ is an acceptable $N_k$-block, then $\varphi(W)$ differs from $W$ in at most $\varepsilon\,N_k$ places. This concludes the proof of the lemma.
\end{proof}

\begin{Lemma}
 \label{lemma:still-completely-deterministic}
 Let $k$ and $\ou$ be produced as in Lemma~\ref{lemma:approach-deterministic-sequences-step1}. Let us consider $\ou$ as a sequence in $(A^{N_k})^{\N^\ast}$, and denote by $S_{N_k}$ the action of the shift map in this setting (that is, $S_{N_k}$ shifts $N_k$ letters in $A$ at the same time). Set also, for each integer $s\ge 0$, $M_s:=N_{k+s}/N_k$. Then there exists an increasing sequence of integers $(s_\ell)_{\ell\ge1}$, and an $S_{N_k}$-invariant probability measure $\onu$ on $(A^{N_k})^{\N^\ast}$ such that
 \begin{itemize}
  \item we have the weak convergence
	 \[
      \delta_{S_{N_k},M_{s_\ell},\ou} = \frac{1}{M_{s_\ell}}\sum_{n\le M_{s_\ell}} \delta_{S_{N_k}^n \ou} \tend{\ell}{\infty} \onu,
     \]
  \item $h(S_{N_k},\onu)=0$.
 \end{itemize}

\end{Lemma}

\begin{proof}
 First, let $\mu$ be any weak limit of a subsequence of the form
$\delta_{S^{N_k},M_{s_\ell},u}$, $\ell\geq1$.
 Then $\mu$ is $S^{N_k}$-invariant. Moreover, we have
 \[
  \nu=\frac{1}{N_k}(\mu+S_*\mu+\cdots+S^{N_k-1}_*\mu).
 \]
 Since $h(\nu,S_{N_k})=0$, we have also $h(\mu,S_{N_k})=0$. Let $\Phi$ be the continuous map defined by the $N_k$-block recoding $\varphi$ from Lemma~\ref{lemma:approach-deterministic-sequences-step1}. We get the announced result with $\onu$ the pushforward measure of $\mu$ by $\Phi$.
\end{proof}

\begin{Lemma}
 \label{lemma:approach-deterministic-sequences-step2}
 With the same assumptions as in Lemma~\ref{lemma:approach-deterministic-sequences-step1}, for any $\varepsilon>0$, we can find a sequence $\ou\in A^{\N^\ast}$ and a subsequence $(N_{k(\ell)})_{\ell\ge1}$ such that:
 \begin{itemize}
  \item $h_{top}(\ou)=0$;
  \item for each $\ell\ge1$, \eqref{eq:approach_u_step1} with $k$ replaced by $k(\ell)$, is satisfied.
 \end{itemize}
\end{Lemma}

\begin{proof}
 Let $\ou^{(1)}$ be the sequence we obtain applying Lemma~\ref{lemma:approach-deterministic-sequences-step1} with $\varepsilon/2$, and let $k(1)$ be the corresponding integer $k$. Then $\ou^{(1)}$ can be viewed as an infinite concatenation of at most $2^{N_{k(1)}\varepsilon/2}$ different $N_{k(1)}$-blocks.
 By Lemma~\ref{lemma:still-completely-deterministic}, and since all integers $N_k$, $k\ge k(1)$, are multiples of $N_{k(1)}$, we can apply Lemma~\ref{lemma:approach-deterministic-sequences-step1} to the new sequence $\ou^{(1)}$ itself, viewed as a sequence in $(A^{N_{k(1)}})^{\N^\ast}$.
Doing this with $\varepsilon/4$, we obtain a new sequence $\ou^{(2)}$ and an integer $k(2)$. If we consider both $\ou^{(1)}$ and $\ou^{(2)}$ as concatenation of $N_{k(1)}$-blocks, all blocks used in $\ou^{(2)}$ are already used in $\ou^{(1)}$, so that $\ou^{(2)}$ is itself an infinite concatenation of at most $2^{N_{k(1)}\varepsilon/2}$ different $N_{k(1)}$-blocks.
On the other hand, if we consider now both $\ou^{(1)}$ and $\ou^{(2)}$ as sequences in $A^{\N^\ast}$, they coincide on their first $N_{k(1)}$ symbols.

We go on in the same way by induction. At step $\ell$, we have constructed a sequence $\ou^{(\ell)}$ and we have an integer $k(\ell)$ satisfying
\begin{equation}
\label{eq:diff_with_u}
\frac{1}{N_{k(\ell)+s}} \left|\{1\le n\le N_{k(\ell)+s}: u_n\neq \ou^{(\ell)}_n\}\right|<\frac{\varepsilon}{2}+\cdots+\frac{\varepsilon}{2^\ell} \text{ for all }s\ge 0,
\end{equation}
and for each $1\le j\le \ell$,
\begin{equation}\label{eq:not_too_many_blocks}
 \parbox{0.8\linewidth}{$\ou^{(\ell)}$ is an infinite concatenation of at most $2^{N_{k(j)}\varepsilon/2^j}$ different $N_{k(j)}$-blocks.}
\end{equation}
Consider $\ou^{(\ell)}$ as a sequence on the alphabet $A^{N_{k(\ell)}}$, which is quasi-generic for some $S_{N_{k(\ell)}}$-invariant probability with zero entropy along a subsequence of the original sequence $(N_k)$. We apply on it
Lemma~\ref{lemma:approach-deterministic-sequences-step1} with $\varepsilon/2^{\ell+1}$ to get a new sequence
$\ou^{(\ell+1)}$ and an integer $k(\ell+1)$, satisfying the analogous properties to~\eqref{eq:diff_with_u} and~\eqref{eq:not_too_many_blocks} at level $\ell+1$,
and such that $\ou^{(\ell+1)}$ coincides with $\ou^{(\ell)}$ on their first $N_{k(\ell)}$ symbols.

The sequence $(\ou^{(\ell)})_{\ell\ge1}$ which is obtained in this way, converges to a sequence $\ou$, satisfying for all $\ell\ge1$,
\[
 \ou|_1^{N_{k(\ell)}} = \ou^{(\ell)}|_1^{N_{k(\ell)}}.
\]
By~\eqref{eq:diff_with_u}, this ensures that for each $\ell\ge1$,
\[
 \frac{1}{N_{k(\ell)}} \left|\{1\le n\le N_{k(\ell)}: u_n\neq \ou_n\}\right|<\varepsilon.
\]
Moreover, by~\eqref{eq:not_too_many_blocks}, for each $\ell\ge1$, $\ou$ is an infinite concatenation of at most $2^{N_{k(\ell)}\varepsilon/2^\ell}$ different $N_{k(\ell)}$-blocks. Therefore, there are at most $N_{k(\ell)}\cdot 2^{N_{k(\ell)}\varepsilon/2^\ell}$ different $N_{k(\ell)}$-blocks which appear in $\ou$. This implies that $h_{top}(\ou)=0$.
\end{proof}

To conclude the proof of the equivalence of \eqref{S} and \eqref{S0}, we need also some tool to pass from the continuous case of a general sequence $\bigl(f(T^nx)\bigr)_{n\in\N^\ast}$ to the discrete case of a symbolic sequence $x\in A^{\N^\ast}$ for some finite $A\subset \R$. This is the object of what follows.

For each finite subset $A\subset\R$, we denote by $\varphi_A$ the function from $[\min A,+\infty)$ to $A$ which maps $t\ge\min A$ to the largest element $a\in A$ satisfying $a\le t$. We also denote by $\Phi_A$ the function from $[\min A,+\infty)^{\N^\ast}$ to $A^{\N^\ast}$ which maps each sequence $(y_n)_{n\in\N^\ast}$ to $\bigl(\varphi_A(y_n)\bigr)_{n\in\N^\ast}$.

\begin{Lemma}
 \label{lemma:continuous_to_discrete}
 Let $y=(y_n)_{n\in\N^\ast}$ be a bounded sequence of real numbers, with values in some compact interval $[\alpha,\beta]$. We assume that, along some increasing sequence of integers $(N_k)$, the following weak convergence holds:
 \[
 \delta_{S,N_k,y} \tend{k}{\infty} \mu,
 \]
 where $\mu$ is a shift-invariant probability on $[\alpha,\beta]^{\N^\ast}$. Then, for each $\varepsilon>0$, we can find a finite subset $A\subset \R$ such that:
 \begin{itemize}
  \item $\forall t\in[\alpha,\beta]$, $|\varphi_A(t)-t|<\varepsilon$,
  \item we have the weak convergence
\begin{equation}
\label{eq:weak_convergence_Phi_A}
 \delta_{S,N_k,\Phi_A y} \tend{k}{\infty} (\Phi_A)_*\mu.
\end{equation}
 \end{itemize}

\end{Lemma}

\begin{proof}
 The first condition required on $A$ is easily satisfied: we just have to choose $A$ so that
\begin{itemize}
 \item $\min A<\alpha$,
 \item $\sup A\ge \beta$,
 \item the distance between two consecutive elements of $A$ is always less than $\varepsilon$.
\end{itemize}
Then, for such an $A$, observe that $\varphi_A$ is continuous on $[\min A,+\infty)\setminus A$ ($A$ is the set of discontinuity points of $\varphi_A$), and that $\Phi_A$ is continuous on
$$
[\min A,+\infty)^{\N^\ast} \setminus r(A),
$$
where
\[
 r(A) := \{y\in[\alpha,\beta]^{\N^\ast}: y_n\in A\text{ for some }n\in\N\}.
\]
Consider the pushforward measure of $\mu$ by the projection of $[\alpha,\beta]^{\N^\ast}$ to the first coordinate. This is a probability measure on the interval $[\alpha,\beta]$, hence with at most a countable number of atoms. Moreover, the pushforward measure of $\mu$ by the projection of $[\alpha,\beta]^{\N^\ast}$ to any other coordinate has the same atoms, since $\mu$ is shift-invariant. Choosing the elements of $A$ from the complement of this set of atoms is always possible, and ensures that
\begin{equation}
 \label{eq:discontinuity_of_measure_0}
 \mu\bigl(r(A)\bigr)=0.
\end{equation}
 Finally, note that for any $k$, the pushforward of $\delta_{S,N_k,y}$ by $\Phi_A$ is precisely $\delta_{S,N_k,\Phi_A y}$. Since by~\eqref{eq:discontinuity_of_measure_0}, the set of discontinuities of $\Phi_A$ has $\mu$-measure 0, we get~\eqref{eq:weak_convergence_Phi_A}.
\end{proof}

\begin{proof}[Proof of Theorem~\ref{thm:SequivS0}]
It is clear from the definitions that condition \eqref{S} implies \eqref{S0}.

Assume that $z\in \{-1,0,1\}^{\N^\ast}$ does not satisfy~\eqref{S}. Then there exist a homeomorphism $T$ of a compact metric space $X$, a continuous function $f:X\to\R$, and a completely deterministic point $x\in X$ such that
$
 \frac{1}{N}\sum_{n\le N} f(T^nx) z_n
$
does not converge to 0 as $N\to\infty$. We can thus find some increasing sequence of integers $(N_k)$, and some $\theta\neq0$ such that
\begin{equation}
 \label{eq:convergence_to_theta}
 \frac{1}{N_k}\sum_{n\le N_k} f(T^nx) z_n\tend{k}{\infty}\theta.
\end{equation}
Without loss of generality, we can further assume that $N_k|N_{k+1}$ for each $k$. Indeed, extracting a subsequence if necessary, we can always assume that
\[
 N_k/N_{k+1}\tend{k}{\infty}0,
\]
and then replace inductively each $N_{k+1}$ by the closest multiple of $N_k$.

We can also assume that
\[
 \delta_{T,N_k,x} \tend{k}{\infty} \nu,
\]
where $\nu$ is a $T$-invariant probability measure on $X$ satisfying $h(T,\nu)=0$ (because $x$ is completely deterministic).
Let $\alpha:= \min f$, $\beta:=\max f$. If we set $y=(y_n)_{n\in \N}:=\bigl(f(T^nx)\bigr)_{n\in\N^\ast}\in[\alpha,\beta]^{\N^\ast}$, then we also have
\[
 \delta_{S,N_k,y} \tend{k}{\infty} \mu,
\]
where $\mu$ is the pushforward of $\nu$ to $[\alpha,\beta]^{\N^\ast}$ by the topological factor map defined by $f$. In particular, we have $h(S,\mu)=0$.

Now, choose $\varepsilon>0$ small enough so that
\begin{equation}\label{eq:chose-vep}
\vep<|\theta|/4.
\end{equation}
Let $A$ be the finite set given by Lemma~\ref{lemma:continuous_to_discrete} applied to $y=\left(f(T^nx)\right)_{n\in\N^\ast}$ and $\varepsilon$, and set $u=(u_n)_{n\in\N^\ast}:=\Phi_A(y)$. Then, we have
\begin{equation}
 \label{eq:good_symbolic_approximation}
|u_n-f(T^nx)|<\varepsilon \text{ for all }n\in\N,
\end{equation}
and
\begin{equation*}
 \label{eq:u_generic}
\delta_{S,N_k,u} \tend{k}{\infty} (\Phi_A)_*\mu.
\end{equation*}
Moreover, since $h(S,\mu)=0$ and $\bigl(S,A^{\N^\ast},(\Phi_A)_*\mu\bigr)$ is a measure-theoretic factor of $\bigl(S,[\alpha,\beta]^{\N^\ast},\mu\bigr)$, we also have $h\bigl(S,(\Phi_A\bigr)_*\mu)=0$.

We apply now Lemma~\ref{lemma:approach-deterministic-sequences-step2} to $u$ and $\vep$, obtaining a sequence $\ou$ with $h_{top}(\ou)=0$ and a subsequence $(N_{k(\ell)})$ such that
\begin{equation}\label{EQQ}
   \dfrac{1}{N_{k(\ell)}}\left|\bigl\{1\le n\le N_{k(\ell)}: u_n\neq\ou_n\bigr\}\right| < \varepsilon.
\end{equation}
Then
\begin{multline*}
 \frac{1}{N_{k(\ell)}} \sum_{ n\le N_{k(\ell)}} \ou_n z_n = \frac{1}{N_{k(\ell)}} \sum_{ n\le N_{k(\ell)}} (\ou_n-u_n) z_n\\
 + \frac{1}{N_{k(\ell)}} \sum_{ n\le N_{k(\ell)}} (u_n-f(T^nx)) z_n+ \frac{1}{N_{k(\ell)}} \sum_{ n\le N_{k(\ell)}} f(T^nx) z_n.
\end{multline*}
It follows from~\eqref{EQQ}, by~\eqref{eq:good_symbolic_approximation} and by~\eqref{eq:convergence_to_theta} that for $\ell$ sufficiently large
$$
\left| \frac{1}{N_{k(\ell)}} \sum_{ n\le N_{k(\ell)}} \ou_n z_n\right|\geq |\theta|/2-2\vep.
$$
Therefore~\eqref{eq:chose-vep} implies that
$$
\frac{1}{N}\sum_{ n\leq N}\ou_n z_n\not\to 0,
$$
i.e.\ $z$ does not satisfy~\eqref{S0} and the assertion follows.
\end{proof}

\section{\eqref{Ch} vs. various properties}\label{se:4}
\subsection{\eqref{S} does not imply \eqref{Ch}}\label{se:4.1}

A natural question arises, whether it is possible to find a sequence which satisfies~\eqref{S} and does not satisfy~\eqref{Ch}. We will provide now such an example.

\begin{Example}[$z$ that satisfies~\eqref{S} but not~\eqref{Ch}]\label{ex:1}

Consider the shift on $\{0,1,2,3\}^{\N^\ast}$ with the Bernoulli measure $B(\frac14,\frac14,\frac14,\frac14)=:\kappa$ and
let
$$
\theta:\{0,1,2,3\}^{\N^\ast}\to\{-1,0,1\}^{\N^\ast}
$$
be given by the code of length~2: $\theta(01)=\theta(12)=-1$, $\theta(02)=\theta(23)=1$, all remaining blocks of length~2 sent to~0. Let $\omega\in \{0,1,2,3\}^{\N^\ast}$ be a generic point for $\kappa$ (such a point exists by the ergodic theorem). Then $\nu:=\theta_\ast(\kappa)$ is an invariant measure for the subshift $Y:=\theta(\{0,1,2,3\}^{\N^\ast})\subset\{-1,0,1\}^{\N^\ast}$. Moreover, $z:=\theta(\omega)$ is a generic point for $\nu$ and $(S,Y,\nu)$ is a Bernoulli automorphism \cite{Or}. Recalling that $F(u):=u(1)$ for $u\in Y$, we have
\begin{multline*}
\frac1N \sum_{n\leq N}z^{i_0}(n)\cdot z^{i_1}(n+a_1)\cdot\ldots\cdot z^{i_r}(n+a_r)\\
=\frac1N \sum_{n\leq N} \big(F^{i_0} \cdot F^{i_1}\circ S^{a_1} \cdot\ldots\cdot F^{i_r}\circ S^{a_r}\big)(S^{n-1} z).
\end{multline*}
In particular (by genericity),
\beq\label{s1}\lim_{N\to\infty}\frac1N \sum_{n\leq N} z(n)z(n+1)=\int_Y F\cdot F\circ S\ d\nu.
\eeq
Observe that \beq\label{s0}\int_YF\ d\nu=0.
\eeq
However, the function
$F\cdot F\circ S$ takes the value 1 with probability $2/4^3$ (given by the blocks $012$ and $023$), while the value -1 has probability $1/4^3$ (it is given by $123$). It follows that the integral in~\eqref{s1} is not equal to zero, i.e.\ \eqref{Ch} does not hold. Note that, in this construction, $z$ is a generic point (so the more $z^2$ is a generic point).

It remains to show that $z$ satisfies~\eqref{S}. This is however clear: for any topological dynamical system $(X,T)$ and any $x\in X$, each accumulation point, say $\rho$, of the sequence of empiric measures
$
\delta_{T\times S,N,(x,z)}$, $N\geq1$,
is a joining of $(T,X,\rho|_X)$ and $(S,Y,\nu)$ (the latter, since $z$ is generic for $\nu$). If $x$ is completely deterministic, then $\rho|_X\in \text{Q-gen}(x)$ has zero entropy, hence $(X,T,\rho|_X)$ is disjoint from any K-system. In particular, $\rho=\rho_X\ot \nu$ and~\eqref{S} follows from~\eqref{s0}.
\end{Example}

\begin{Remark} The point $z$ in the above example is clearly not completely deterministic. In fact, if $z$ satisfies~\eqref{S} and is completely deterministic, then 
$\frac{1}{N}\sum_{n\leq N}z^2(n)=\frac{1}{N}\sum_{n\leq N}z(n)\cdot z(n)\to 0$, so the support of $z$ has zero density and $z$ automatically satisfies~\eqref{Ch}.
\end{Remark}

\begin{Remark}
Example~\ref{ex:1} can be seen as a starting point for a construction of sequences such that the convergence in~\eqref{Ch} holds whenever $a_r< k_0$ ($k_0\geq 2$), and fails for some choice of $1\leq a_1<\dots<a_r=k_0$. Indeed, consider again the shift on $\{0,1,2,3\}^{\N^\ast}$ with the Bernoulli measure $B(1/4,1/4,1/4,1/4)$ and
let
$$
\theta:\{0,1,2,3\}^{\N^\ast}\to\{-1,0,1\}^{\N^\ast}
$$
be given by the code of length~$k_0$: $\theta(0\ast 1)=\theta(1\ast 2)=-1$, $\theta(0\ast 2)=\theta(2\ast 3)=1$, where $\ast$ stands for any sequence of symbols from $\{0,1,2,3\}$ of length $k_0-2$ and all remaining blocks of length $k_0$ sent to~$0$. Let $\nu, \omega$ and $z$ be as in Example~\ref{ex:1}. By genericity
\begin{multline*}
\lim_{N\to \infty}\frac{1}{N}\sum_{n\leq N}z^{i_0}(n)\cdot z^{i_1}(n+a_1)\cdot\ldots\cdot z^{i_r}(n+a_r)\\
=\int_Y F^{i_0} \cdot F^{i_1}\circ S^{a_1} \cdot\ldots\cdot F^{i_r}\circ S^{a_r}\ d\nu.
\end{multline*}
If $a_r<k_0$ then each of the functions $F\circ S^a$ in the above integral take the values $1$ and $-1$ with probability $2/4^2$ and these events (as $a$ varies from $0$ to $k_0-1$) are independent. Therefore, whenever $a_r<k_0$, then the corresponding integral equals zero (when one of the $i_s$ equals~1). However, the function $F\cdot F\circ S^{k_0}$ takes the value $1$ with probability $2/4^3$ (given by the blocks $0\ast 1\ast 2$ and $0\ast 2\ast 3$) while the value $-1$ has probability $1/4^3$ (it is given by $1\ast 2\ast 3$), so the integral is not equal to zero. In other words,~\eqref{Ch} fails for this sequence when $r=1$ and $a_1=k_0$.
\end{Remark}

\subsection{\eqref{Ch} without genericity}\label{se:4.2}
We will show that $z$ may satisfy~\eqref{Ch} without being a generic point (in fact, even $z^2$ may fail to be generic).
\begin{Example}[$z$ that satisfies~\eqref{Ch} with $z^2$ not generic]
Let $ w_0\in Y:=\{-1,0,1\}^{\N^\ast}$ be a generic point for the Bernoulli measure $\kappa_0:=B(1/3,1/3,1/3)$, and $ w_1\in Y$ a generic point for the Bernoulli measure $\kappa_1:=B(1/2,0,1/2)$. Since the measures $\kappa_0$ and $\kappa_1$ are mutually singular, up to a set of $(\kappa_0+\kappa_1)$-measure zero, we can represent $Y$ as a union $Y_0\cup Y_1$, $Y_0\cap Y_1=\varnothing$ with $Y_i$ being a set of full measure for $\kappa_i$, $i=0,1$.

Let $\Z\ni a_n\to\infty$ and set
\beq\label{ex21} M_1:=1,\;M_{n+1}:=a_{n+1}M_n.
\eeq
We define a new sequence $ w\in\{-1,0,1\}^{\N^\ast}$ by setting
$$
 w[M_{2i+1}, M_{2i+2}-1]:= w_0[0,M_{2i+2}-M_{2i+1}-1],\;i\geq0,
$$
$$
 w[M_{2i}, M_{2i+1}-1]:= w_1[0,M_{2i+1}-M_{2i}-1],\;i\geq1.
$$

\begin{Lemma}\label{ll1}
We have $\text{Q-gen}( w)=\{\alpha \kappa_0+(1-\alpha)\kappa_1: \alpha\in[0,1]\}$.
\end{Lemma}
\begin{proof}
Suppose that, for some increasing sequence $(P_i)$, $\delta_{P_i, w}\to \nu$. Then, for each $i\geq1$, there exists $s_i\geq1$, so that $M_{s_i}\leq P_i<M_{s_i+1}$. By considering subsequences, if necessary, we can assume that $M_{s_i}/P_i\to\alpha$ (moreover, for any $\alpha\in [0,1]$ the sequence $(P_i)$ can be chosen so that this convergence holds). Since $M_{s_i-1}/P_i=\frac{M_{s_i}}{a_{s_i}P_i}\to 0$,
the sequence of measures
$\frac{1}{P_i}\sum_{n<M_{s_i-1}}\delta_{S^n w}$ converges to 0 when $i\to\infty$.
Moreover, since $(M_{s_i}-M_{s_i-1}-1)/P_i\to \alpha$, the measure $\frac{1}{P_i}\sum_{M_{s_i-1}\leq n<M_{s_i}}\delta_{S^n w}$ is arbitrarily close to $\alpha\kappa_{j_i}$, where $\kappa_{j_i}$ is either $\kappa_0$ or $\kappa_1$
depending on the parity of $s_i$ (we pass again to a subsequence if necessary). In a similar way, $\frac{1}{P_i}\sum_{M_{s_i}\leq n<P_i}\delta_{S^n w}\to (1-\alpha)\kappa_{1-j_i}$ and the result follows.
\end{proof}

Clearly, $ w$ is not generic, and we can easily check that neither is $ w^2$ (we obtain that $ w^2$ is quasi-generic for all convex combinations of the Dirac measure at $(1,1,\ldots)$ and a Bernoulli measure).

Now, since~\eqref{Ch} holds for $ w_0$ and $ w_1$, the integral of $F^{i_0}\cdot F^{i_1}\circ S^{a_1}\cdot\ldots\cdot F^{i_r}\circ S^{a_r}$ with respect to $\kappa_i$ for $i=0,1$ is equal to zero for any choice of $1\leq a_1<\ldots<a_r$, $r\geq 0$, $i_s\in\{1,2\}$ not all equal to~$2$. Therefore, for any such choice we also have
$$
\int F^{i_0}\cdot F^{i_1}\circ S^{a_1}\cdot\ldots\cdot F^{i_r}\circ S^{a_r}\ d(\alpha\kappa_0+(1-\alpha)\kappa_1)=0,
$$
which shows that~\eqref{Ch} holds for $ w$.
\end{Example}

\subsection{The squares in~\eqref{Ch} are necessary}\label{se:4.3}
We will now show that the squares in~\eqref{Ch} are necessary. In other words, we will show that~\eqref{Ch} is not equivalent to the following condition:
\begin{equation}\label{Ch1}
\frac1N \sum_{n\leq N}z (n) \cdot z (n+a_1)\cdot \ldots \cdot z (n+a_r) \tend{N}{\infty} 0\tag{Ch1}
\end{equation}
for each choice of $1\leq a_1<\ldots<a_r$, $r\geq 0$. The example will be introduced in the probabilistic language (cf. the discussion on page~\pageref{echowla1}). In order to obtain a sequence $z\in \{-1,0,1\}^{\N^\ast}$ satisfying~\eqref{Ch1} and not satisfying~\eqref{Ch} it suffices to take a generic point for the distribution of the process ${(X_n)}_{n\in\N^\ast}$ considered in the following example.
\begin{Example}
Let ${(Y_n)}_{n\in\N^\ast}$ be a sequence of independent random variables, taking values $\pm 1$, each with probability $1/2$. Set
$$
X_n:=Y_n\mathbf{1}_{Y_{n+1}=1}.
$$
Then for each choice of $0\leq a_1<\dots<a_r$, we have
$$
\E(X_{a_1}\cdot\ldots\cdot X_{a_r})=\E(Y_{a_1}\mathbf{1}_{Y_{a_1+1}=1}Y_{a_2}\mathbf{1}_{Y_{a_2+1}=1}\ldots Y_{a_s}\mathbf{1}_{Y_{a_s+1}=1})=\E(Y_{a_1}Z),
$$
where $Z$ is measurable with respect to the $\sigma$-algebra generated by $Y_{a_1+1}, Y_{a_1+2},\ldots$, hence is independent from $Y_{a_1}$. Since $\E(Y_{a_1})=0$, we get
$$
\E(X_{a_1}\cdot\ldots\cdot X_{a_r})=0.
$$
However, since $X_1^2\cdot Y_2=\mathbf{1}_{Y_2=1}$, we have
$$
\E(X_1^2X_2)=\E(\mathbf{1}_{Y_2=1}\mathbf{1}_{Y_3=1})=\frac{1}{4}\neq 0.
$$
\end{Example}

\begin{Remark}
The dynamical system determined by ${(X_n)}_{n\in\N^\ast}$ is a non-trivial factor of the system determined by ${(Y_n)}_{n\in\N^\ast}$. Moreover, ${(Y_n)}_{n\in\N^\ast}$ is an independent process, so the associated dynamical system is K. Hence $h_{top}(z)>0$ for any $z\in\{-1,0,1\}^{\N^\ast}$ generic for the distribution of ${(X_n)}_{n\in\N^\ast}$.
\end{Remark}
\begin{Qu}\label{Qu:1}
Does there exist $z\in \{-1,0,1\}^{\N^\ast}$ with $h_{top}(z)=0$, for which~\eqref{Ch} fails but that satisfies~\eqref{Ch1}?
\end{Qu}

\begin{Remark}
 \label{rk17}Note that if $h_{top}(z)=0$, and if the density of the support of $z$ is positive, then~\eqref{Ch} automatically fails because~\eqref{S} fails.
\end{Remark}

\subsection{\eqref{Ch} vs. recurrence}\label{se:4.4}
In this section we discuss the recurrence properties of sequences satisfying~\eqref{Ch}. 
\begin{Def}
Let $A$ be a nonempty finite set. A sequence $w\in A^{\N^\ast}$ is said to be {\em recurrent} if
each block $B$ appearing in $w$ appears in it infinitely often.
\end{Def}
Note that, if
\begin{equation}\label{g:1}
X_w^+= X_w,
\end{equation}
then obviously $w$ is recurrent.

It is well-known (see, e.g.\ \cite{Do}, pp.\ 189-190) that under the recurrence assumption, one can construct the topological natural extension of the one-sided subshift generated by $w$. More precisely, under the assumption of recurrence of $w$, there exists $\widetilde{w}\in A^\Z$ such that:
\begin{itemize}
\item $\widetilde{w}[0,\infty]=w$;
\item each block appearing in $\widetilde{w}$ appears in $w$.
\end{itemize}
Our main result in this section is the following:
\begin{Prop}\label{lm:5}
Suppose that~\eqref{g:1} holds for $z^2$, i.e.\
\begin{equation}\label{cond}
X_{z^2}^+=X_{z^2}
\end{equation}
and $z$ satisfies~\eqref{Ch}. Then~\eqref{g:1} holds for $z$;
in fact,
\beq\label{cond1}
\bigcup_{\nu\in \text{Q-gen}(z^2)}\text{supp}(\widehat\nu)=X_{z}.
\eeq
In particular, $z$ is recurrent.
\end{Prop}
For the proof we will need two lemmas.

\begin{Lemma}\label{lm:f4}
Let $w\in A^{\N^\ast}$ and consider the subshift $X_w\subset A^{\N^\ast}$. Then given a block $B\in A^r$ (for some $r\geq 1$), the following two conditions are equivalent:
\begin{itemize}
\item
$\text{there exists }\nu\in\text{Q-gen}(w)\text{ such that }\nu(B)>0$,
\item
$B$ appears in $w$ with positive upper frequency.
\end{itemize}
In other words,
$X_w^+=\bigcup_{\nu\in\text{Q-gen}(w)}\text{supp}(\nu)$.
\end{Lemma}
\begin{proof}
Let $(N_k)$ and $\nu\in \text{Q-gen}(w)$ be such that
$\delta_{N_k,w}\tend{k}{\infty} \nu$
and let $B\in A^r$, $r\geq 1$, be such that $\nu(B)>0$. Since $\mathbf{1}_B\in C(A^{\N^\ast})$, it follows that
$$
\overline{\text{fr}}(B,w)\geq \lim_{k\to \infty}\int \mathbf{1}_B\ d\,\delta_{ N_k,w} = \int\mathbf{1}_B\, d\nu=\nu(B)>0.
$$

Suppose now that $B$ appears in $w$ with positive upper frequency, i.e.\ we have
\begin{equation}\label{f3}
\lim_{k\to \infty}\int \mathbf{1}_B\ d \delta_{N_k,w}>0,
\end{equation}
for some increasing sequence $(N_k)$.
Passing to a subsequence if necessary, we may assume that
$
\delta_{N_k,w}\tend{k}{\infty} \nu$ weakly;
in particular, $\nu\in \text{Q-gen}(w)$. Moreover, by~\eqref{f3}, $\nu(B)>0$.
\end{proof}

Fix $z\in\{-1,0,1\}^{\N^\ast}$, and set $X_{N}:=\pi^{-1}(X_{z^2})$,
$X_{N}^+:=\pi^{-1}(X_{z^2}^+)$.
\begin{Lemma}\label{lemat15}
If~\eqref{Ch} holds for $z$ then $X_z^+=X_N^+$.
\end{Lemma}
\begin{proof}
Clearly, $X_z^+\subset X_N^+$. Take $\widetilde{B}\in X_N^+$. Then $B:=\pi(\widetilde{B})\in X_{z^2}^+$ and by Lemma~\ref{lm:f4} there exists $\nu\in\text{Q-gen}(z^2)$ such that $\nu(B)>0$. Therefore
$$
\widehat{\nu}(\widetilde{B})=\frac{1}{2^{|\text{supp} B|}}\nu(B)>0.
$$
Since $z$ satisfies~\eqref{Ch}, it follows from Lemma~\ref{rownow} that $\widehat{\nu}\in \text{Q-gen}(z)$. Lemma~\ref{lm:f4} implies now that $\widetilde{B}\in X_z^+$ and the assertion follows.
\end{proof}

\begin{proof}[Proof of Proposition~\ref{lm:5}]
By Lemma~\ref{lm:f4} and Remark~\ref{re:ch:4} (which can be applied since $z$ satisfies~\eqref{Ch}), we have
\begin{equation}\label{cond2}
X_z^+=\bigcup_{\widetilde{\nu}\in\text{Q-gen}(z)}\text{supp}(\widetilde{\nu})=\bigcup_{\nu\in \text{Q-gen}(z^2)}\text{supp}(\widehat{\nu}).
\end{equation}
It follows from~\eqref{cond} that
$
X_N^+=X_N$.
This and Lemma~\ref{lemat15} imply
$$
X_z\subset X_N=X_N^+=X_z^+\subset X_z
$$
so \eqref{g:1} holds for $z$. Therefore, and using also~\eqref{cond2}, we conclude that~\eqref{cond1} holds.
\end{proof}

\begin{Remark}
Note that condition~\eqref{cond} is satisfied if $\text{Q-gen}(z^2)=\{\nu\}$ and $\text{supp } \nu=X_{z^2}$. This is the case for $z=\mob$; in particular, $\mob^2$ is recurrent. However, it is not known whether $\mob$ is recurrent (Sarnak, see also recent~\cite{Ma-Ra-Ta}).
\end{Remark}

It is possible to have $z$ satisfying~\eqref{Ch} and non-recurrent with $z^2$ being recurrent. Consider the following two examples:

\begin{Example}\label{AA}
For $i\geq1$, let $B_i$ be the block consisting of $10^i$ zeroes. Then set
$A_1:=1 B_1$,
$A_2:=A_1 A_1 B_2$ and in general $A_{s+1}:=A_sA_sB_s$ for $s\geq 2$ to obtain in the limit the sequence $z^2$ which is recurrent. Replace first 1 by -1 without changing
other positions to define $z$. Then $z$ satisfies~\eqref{Ch} and $z$ is not recurrent. A ``drawback'' of this example is that the density of zeroes is equal to~$1$.
\end{Example}
\begin{Example}
We will use that same idea as in Example~\ref{AA}. Let $(n_i)$ be an increasing sequence of natural numbers and let $B_i$ be a block of length $2n_i$ of alternating ones and zeroes: $B_i=1010\ldots 10$. Then set $A_1:=11 B_1$,
$A_2:=A_1 A_1 B_2$ and in general $A_{s+1}:=A_sA_sB_s$. In the limit, we obtain an infinite sequence $w$. If $(n_i)$ increases fast enough then $w$ differs from $w':=(1,0,1,0,\ldots)$ on a set of density zero, whence $w$ is generic for $\nu:=\frac{1}{2}(\delta_{(1,0,1,0,\ldots)}+\delta_{(0,1,0,1,\ldots)})$. Let $u\in \{-1,1\}^{\N^\ast}$ be generic for $B(1/2,1/2)$ and let $z':=\xi(w',u)$, where $\xi(a,b)(n):=a(n)\cdot b(n)$ for $n\in\N^\ast$. Since $w'$ is of zero entropy, $(w',u)$ is generic for $\nu\otimes B(1/2,1/2)$. Hence $z'$ is generic for $\xi_\ast(\nu\otimes B(1/2,1/2))=\widehat{\nu}$. The sequence $z'$ is a concatenation of blocks of length~3 of $\pm1$ separated by long blocks of $-1,0,1$ in which every second position is~0. To obtain $z$, we now modify $z'$ in the following way. The first block of 3 consecutive $\pm1$ (i.e.\ $z'(1)z'(2)z'(3)$) is replaced with $(-1,-1,-1)$ while all other 3-blocks of consecutive $\pm1$ are replaced with 111. Then $z$ differs from $z'$ on a
subset of density zero, so $z$ is still a generic point
for $\widehat{\nu}$, i.e.\ \eqref{Ch} holds for $z$. Clearly, $z$ is not recurrent, whereas $w=z^2$ has this property. Moreover, the density of $0$'s in $z$ is equal to $1/2$.
\end{Example}

\subsection{\eqref{Ch} vs.\ unique ergodicity}\label{se:4.5}

\begin{Prop}\label{uwaga2}
Let $z\in \{-1,0,1\}^{\N^\ast}$ be such that~\eqref{Ch} holds. Moreover, suppose that there exists a block $B$ with
\beq\label{GYGY}
\text{$\text{supp}(B)\neq \varnothing$ and $B$ appears in $z$ with positive upper frequency.}
\eeq
Then the subshift $X_z$ cannot be uniquely ergodic.
\end{Prop}
\begin{proof}
It suffices to show that the subshift $X^+_z$ is not uniquely ergodic. By Lemma~\ref{lemat15}, we have $X_z^+=X_N^+$, whence $X_{z^2}^{+}\subset X_z^{+}$. Therefore
\begin{equation}\label{niepusty}
\varnothing \neq \text{Q-gen}(z^2)\subset \mathcal{P}_S(X_z^+).
\end{equation}
 By Remark~\ref{re:ch:4}, we have $\text{Q-gen}(z)=\{\widehat{\nu} : \nu\in\text{Q-gen}(z^2)\}$. Let $B$ be the block with non-empty support given by~\eqref{GYGY}. Then it follows from Lemma~\ref{lm:f4} that there exists $\nu\in\text{Q-gen}(z^2)$ such that $\widehat{\nu}(B)>0$. Moreover, for any block $C$ with $B^2=C^2$ we have $\widehat{\nu}(C)=\widehat{\nu}(B)>0$. It follows that $\nu\neq\widehat{\nu}$, but $\{\nu,\widehat{\nu}\}\subset \mathcal{P}_S(X_z^+)$.
\end{proof}

For $z=\mob$ or $z=\mob_\mathscr{B}$ the fact that the subshift $X_z$ is not uniquely ergodic ``comes'' from $X_{z^2}$. To see this, we need first to recall the following definition \cite{Ke-Li}:
\begin{Def}
A subshift $X\subset \{0,\dots, k\}^{\N^\ast}$ is \emph{hereditary} if for any $x\in X$ and $y\in \{0,\dots, k\}^{\N^\ast}$ the condition $y(n)\leq x(n)$ satisfied for all $n$, implies that $y\in X$.
\end{Def}
\begin{Remark}
In view of~\cite{Sarnak} and~\cite{B-Free}, for $z=\mob$ or $z=\mob_\mathscr{B}$, the subshift $X_{z^2}$ consists of all sequences $w\in \{0,1\}^\Z$ which are $\mathscr{B}$-admissible, i.e.\ such that
$$
t(\text{supp}(w),b)<b \text{ for all }b\in\mathscr{B},
$$
where for $A\subset \Z$ and $b\geq 1$, $t(A,b):=|\{c\in \Z/b\Z : \exists n\in A,\ n=c \bmod b\}|$ is the number of classes modulo $b$ in $A$.
\end{Remark}
It follows immediately from the above remark that the subshift $X_{\mob_\mathscr{B}^2}$ is hereditary. Now, each hereditary system of positive topological entropy (and such are $(S,X_{\mob_\mathscr{B}^2})$ \cite{B-Free}) is not uniquely ergodic, e.g.\ \cite{Kw}.\footnote{A direct proof of non-unique ergodicity of $(S,X_{\mob_\mathscr{B}^2})$ follows from the fact that each hereditary system has a fixed point, whereas the Mirsky measure is positive on each non-empty open subset of $X_{\mob_\mathscr{B}^2}$.}

\begin{Remark} We can choose a generic point $z\in\{-1,1\}^{\N^\ast}$ for the Bernoulli measure $B(1/2,1/2)$ to obtain an example of $z$ satisfying~\eqref{Ch} and for which $(S,X_z)$ is not uniquely ergodic while $(S, X_{z^2})$ has this property.
\end{Remark}

\subsection{Characterization of completely deterministic sequences by orthogonality to~\eqref{Ch}}\label{se:cd-Ch}

In response to an interesting question asked by an anonymous referee, we include the following characterization of completely deterministic sequences by orthogonality to sequences satisfying~\eqref{Ch}. We express our thanks to the referee and to Teturo Kamae who helped us proving this result.

\begin{Prop}
 \label{prop:cd-Ch}
 A sequence $t\in\{-1,1\}^{\N^\ast}$ is completely deterministic if and only if, for each sequence $z\in\{-1,0,1\}^{\N^*}$ satisfying~\eqref{Ch}, we have
 \[
  \frac1{N}\sum_{n\le N}z_n t_n \tend{N}{\infty} 0.
 \]
\end{Prop}

\begin{proof}
 One side of the equivalence follows easily from the preceding results: if $t$ is completely deterministic, and if $z$ satisfies~\eqref{Ch}, then in particular $z$ satisfies~\eqref{S} by Theorem~\ref{ChS} and we get the desired orthogonality.
 
 (Kamae) Conversely, assume that $t\in\{-1,1\}^{\N^\ast}$ is not completely deterministic. Then there exists an increasing sequence $(N_k)$ such that 
 \[ \delta_{N_k,t}\tend{k}{\infty} \nu\in\mathcal{P}_S(\{-1,1\}^{\N^\ast}),
   \]
 with $h(S,\nu)>0$. 
 Consider the Cartesian square of $\{-1,1\}^{\N^\ast}$, and denote by $x=\bigr(x(1),x(2),\ldots\bigr)$
 and $y=\bigr(y(1),y(2),\ldots\bigr)$ the two coordinates in this space. Applying Lemma~3.1 in~\cite{Ka}, we get a joining $\gamma$ of $\nu$ and the Bernoulli measure $B(1/2,1/2)$ on $\{-1,1\}^{\N^\ast}$, under which $x(1)$ and $y(1)$ are not independent: setting $p:=\gamma\bigl( x(1) = 1 \mid y(1) =1\bigr)$ and $q:=\gamma\bigl( x(1) = 1 \mid y(1) =-1\bigr)$, we have $p\neq q$. It follows that 
 \begin{multline}
  \label{eq:p-q}
  \int_{\{-1,1\}^{\N^\ast}\times \{-1,1\}^{\N^\ast}} x(1) y(1) \, d\gamma(x,y)\\ =\frac1{2}p-\frac1{2}(1-p) -\frac1{2}q + \frac1{2}(1-q) = p-q \neq0.
 \end{multline}
Now, by Theorem~2 in~\cite{Ka}, we can find $z\in\{-1,1\}^{\N^\ast}$ that is generic for $B(1/2,1/2)$ (which is equivalent to the fact that $z$ satisfies~\eqref{Ch} by Proposition~\ref{lambda}), such that
\[ 
 \frac1{N_k}\sum_{n\le N_k} \delta_{(S^nt,S^nz)} \tend{k}{\infty} \gamma.
\]
But then, using by~\eqref{eq:p-q}, we get
\[ 
 \frac1{N_k}\sum_{n\le N_k} z_n t_n \tend{k}{\infty} p-q\neq0.
\]
\end{proof}

\section{Sequences satisfying~\eqref{Ch}}\label{se:5}

In this section our main goal is to give natural examples of sequences $z$ satisfying~\eqref{Ch}.
We begin in Section~\ref{se:comments} by discussing the possible values of the pair $(h_{top}(z^2),h_{top}(z))$ when $z$ satisfies~\eqref{Ch}. 
Without any more restriction, this problem has no satisfactory answer. Indeed, we show with the help of a replacement lemma that even the condition of having a support of density 0, which is clearly stronger than~\eqref{Ch}, does not restrict the possible values of this pair of entropies. However, we prove that if $z$ satisfies~\eqref{Ch}, then its topological entropy is bounded from below by the density of its support (Proposition~\ref{prop:Ch_with_positive_density}).
In Section~\ref{se:geme} we describe a method of obtaining sequences satisfying~\eqref{Ch}. Section~\ref{se:sturmian-a} contains background on Sturmian sequences. These tools are used in Section~\ref{exampla}, where we provide two classes of sequences $z$ satisfying~\eqref{Ch}: with $h_{top}(z^2)=0$ and $h_{top}(z^2)>0$.

\subsection{Entropy of sequences satisfying~\eqref{Ch}}\label{se:comments}
The authors would like to thank Benjamin Weiss for fruitful discussions which resulted in the material presented in this section and in the appendix.

Note that, under the assumption that the Chowla conjecture is true for $\mob$, we have in particular $(h_{top}(\mob^2),h_{top}(\mob))=(\frac{6}{\pi^2},\frac{6}{\pi^2}\log 3)$. A natural question arises, what kind of pairs of numbers can be obtained as $(h_{top}(z^2),h_{top}(z))$ for sequences $z\in \{-1,0,1\}^{\N^\ast}$ satisfying~\eqref{Ch}.

First, we observe that there are some natural restrictions for the values of the pair $(h_{top}(z^2),h_{top}(z))$ for $z\in\{-1,0,1\}^{\N^\ast}$. These restrictions are detailed in the appendix of the present paper.
The following replacement lemma is useful for further investigations.

\begin{Lemma}\label{p:extraCh}
Let $z,w\in \{-1,0,1\}^{\N^\ast}$. Then there exists $\overline{z}\in \{-1,0,1\}^{\N^\ast}$ such that:
\begin{itemize}
\item
$\lim_{k\to \infty}\frac{1}{N_k}\sum_{n\leq N_k}\delta_{S^n\overline{z}}=\lim_{k\to \infty}\frac{1}{N_k}\sum_{n\leq N_k}\delta_{S^nz}$ for each increasing sequence $(N_k)$ such that one of these limits exists,
\item
$h_{top}(\overline{z})=\max (h_{top}(z),h_{top}(w))$,
\item
$h_{top}(\overline{z}^2)=\max (h_{top}(z^2),h_{top}(w^2))$.
\end{itemize}
\end{Lemma}

\begin{proof}
For a sequence $x$ over a finite alphabet, we set
$$
\mathcal{C}_n(x):=\{B : |B|=n, B\text{ appears in } x\}, \text{ so that }p_n(x)=|\mathcal{C}_n(x)|\text{ for }n\in\N.
$$
The sequence $\overline{z}$ will be defined as a limit of sequences $\overline{z}_k$, which will be constructed inductively. Fix $0<\vep_k\to 0$. Let $\overline{z}_1:={z}$, and choose $d_1$ large enough so that $1/d_1<\vep_1$. Suppose that $d_1,\dots, d_k$ and $\overline{z_1},\dots, \overline{z}_k$ are already chosen. Let $d_{k+1}$ be large enough, so that
$$
\frac{\min \{i : \overline{z}_k(j)=z(j)\text{ for }j\geq i\}}{d_{k+1}}<\vep_{k+1}.
$$
Let $B_{k+1}\in \{-1,0,1\}^{3d_{k+1}}$ be a block which appears in $\overline{z}_{k}$ infinitely many times. We define $\overline{z}_{k+1}$ by replacing some of the occurrences of $B_{k+1}$ in $\overline{z}_{k}$ by blocks of the form
$$
\underbrace{0\ldots 0}_{d_{k+1}} B \underbrace{0\ldots 0}_{d_{k+1}},\text{ where }B\in\mathcal{C}_{d_{k+1}}(w)
$$
in such a way that
\begin{itemize}
\item
$\overline{z}_{k+1}[1,d_{k+1}]=\overline{z}_{k}[1,d_{k+1}]$,
\item
$\mathcal{C}_{d_{k+1}}(w)\cup \mathcal{C}_{d_{k+1}}(z)\subset \mathcal{C}_{d_{k+1}}(\overline{z}_{k+1})$.
\end{itemize}
It follows immediately that
$$
h_{top}(\overline{z})\geq \max (h_{top}(z),h_{top}(w)).
$$
On the other hand,
$$
p_{d_{k+1}}(\overline{z})\leq \vep_{k+1}d_{k+1}3^{\vep_{k+1}d_{k+1}}p_{d_{k+1}}(z)d_{k+1}+p_{d_{k+1}}(z)2d_{k+1}+p_{d_{k+1}}(w)2d_{k+1},
$$
whence
\begin{align*}
h_{top}(\overline{z})&\leq \max\left(\lim_{k\to \infty}\frac{1}{d_k}\log p_{d_k}(z),\lim_{k\to \infty}\frac{1}{d_k}\log p_{d_k}(w)\right)\\
&=\max(h_{top}(z),h_{top}(w)).
\end{align*}
Therefore
$$
h_{top}(\overline{z})=\max(h_{top}(z),h_{top}(w)).
$$
In a similar way, we conclude that
$$
h_{top}(\overline{z}^2)=\max(h_{top}(z^2),h_{top}(w^2)).
$$
Moreover, if the replacement of blocks made in course of the construction is scarce enough, the resulting sequence $\overline{z}$ will be such that
$$
\lim_{k\to \infty}\frac{1}{N_k}\sum_{n\leq N_k}\delta_{S^n\overline{z}}=\lim_{k\to \infty}\frac{1}{N_k}\sum_{n\leq N_k}\delta_{S^nz},
$$
for any increasing sequence $(N_k)$ such that one of the above limits exists. This completes the proof.
\end{proof}

Applying the above lemma with $z=(0,0,\dots)$, we get the following result.

\begin{Prop}\label{1106a}
For any $(h_{z^2},h_z)\in [0,1]\times [0,\log 3]$, such that for some $w\in\{-1,0,1\}^{\N^\ast}$, we have
$$
(h_{z^2},h_z)=(h_{top}(w^2),h_{top}(w)),
$$
there exists $\overline{z}$ whose support has density 0 (hence satisfying~\eqref{Ch}), such that
$$
(h_{z^2},h_z)=(h_{top}(\overline{z}^2),h_{top}(\overline{z})).
$$
\end{Prop}

Of course, one can object that the examples of sequences $z$ satisfying~\eqref{Ch} provided by the above propositions are rather trivial, since the density of nonzero terms vanishes. 
If we restrict ourselves to sequences $z$ for which the (upper) density of nonzero terms is positive, Remark~\ref{rk17} proves that the topological entropy of $z$ has to be positive if $z$ satisfies~\eqref{Ch}. In fact, we have the following more precise result.

\begin{Prop}
 \label{prop:Ch_with_positive_density}
 Let $z\in\{-1,0,1\}^{\N^*}$, satisfying~\eqref{Ch} and such that
 \begin{equation}\limsup_{N\to \infty}\frac{1}{N}\sum_{n\leq N}z^2(n)=\delta>0.
 \label{eq:positive_density}
\end{equation}
Then $h_{top}(z)\ge\delta$.
\end{Prop}

\begin{proof}
 By~\eqref{eq:positive_density}, $z^2$ is quasi-generic for some shift-invariant probability measure $\nu$ on $\{0,1\}^{\N^*}$ satisfying 
 $\nu\bigl([1]\bigr)=\delta$ (where $[1]$ stands for the cylinder set $\{w: w(1)=1\}$). In particular, the support of $\nu$ contains cylinder sets of arbitrarily large length, for which the density of $1$'s is at least $\delta$. Then, since $z$ satisfies~\eqref{Ch}, Remark~\ref{re:ch:4} shows that $z$ is quasi-generic for $\widehat{\nu}$, and we deduce that $h_{top}(z)\geq\delta$.
\end{proof}

\subsection{Tools}\label{se:general}

\subsubsection{General method}\label{se:geme}
Let $u\in \{-1,1\}^{\N^\ast}$ be a generic point for the Bernoulli measure $B(1/2,1/2)$.
\begin{Prop}\label{pr:1}
If $\eta\in \{0,1\}^{\N^\ast}$ is completely deterministic then~\eqref{Ch} holds for $z:=\eta\cdot u$.
\end{Prop}
\begin{proof}
Let $(N_k)$ be a subsequence such that
$$
\delta_{S\times S, N_k,(\eta,u)}\tend{k}{\infty} \widetilde{\rho}.
$$
Then
$$
\delta_{S,N_k,\eta}\tend{k}{\infty} \rho,
$$
where $\rho$ is the projection of $\widetilde{\rho}$ onto the first coordinate and, by the assumption on $\eta$, $h(S,\rho)=0$. Moreover, since $u$ is generic for $B(1/2,1/2)$, the measure $\widetilde{\rho}$ is a joining of $(S,\rho)$ and $(S,B(1/2,1/2))$. Since $(S,\rho)\perp (S,B(1/2,1/2))$, this must be the product joining, i.e. $\widetilde{\rho}=\rho\otimes B(1/2,1/2)$.

It follows that $\eta\cdot u$ is quasi-generic along $(N_k)$ for the image of $\rho\otimes B(1/2,1/2)$ via the map
\begin{equation}
 \label{eq:map_m}
 m\colon \{0,1\}^{\N^\ast}\times \{-1,1\}^{\N^\ast}\to \{-1,0,1\}^{\N^\ast},
\end{equation}
given by $m(a,b)(n):=a(n)\cdot b(n)$. Clearly, $m_\ast (\rho\otimes B(1/2,1/2))=\widehat{\rho}$. The assertion follows from Lemma~\ref{rownow}.
\end{proof}

\begin{Remark}\label{mu-zero}
Since $\mob^2$ yields a system with discrete spectrum~\cite{Ce-Si}, in particular $\mob^2$ is completely deterministic.
\end{Remark}

\begin{Cor}
Suppose that~\eqref{Ch} holds for the Liouville function $\lamob$. Then~\eqref{Ch} holds for $\mob$.
\end{Cor}
\begin{proof}
The assertion follows directly from the fact that $\mob(n)=\lamob(n)\cdot \mob^2(n)$, Remark~\ref{mu-zero} and Proposition~\ref{pr:1}.
\end{proof}

Proposition~\ref{pr:1} turns out to be a particular case of the following result.
\begin{Prop}\label{pr:2a}
Suppose that~\eqref{Ch} holds for $u\in \{-1,0,1\}^{\N^\ast}$ and that $\eta\in \{-1,0,1\}^{\N^\ast}$ is completely deterministic. Then~\eqref{Ch} holds for $z:=\eta\cdot u$.
\end{Prop}

\begin{proof}
For each $1\leq a_1<\ldots<a_k$, $k\geq 0$ and $i_s\in\{1,2\}$, $1\leq s\leq k$, we set
$$
w(n)=\eta^{i_0}(n)\cdot \eta^{i_1}(n+a_1)\cdot\ldots\cdot \eta^{i_k}(n+a_k).
$$
Then $w$ is completely deterministic by Lemma~\ref{lemat1}~\ref{p:c}, and
we obtain
\begin{multline*}
\frac{1}{N}\sum_{n\leq N}z^{i_0}(n)\cdot z^{i_1}(n+a_1)\cdot\ldots\cdot z^{i_k}(n+a_k)\\
=\frac{1}{N}\sum_{n\leq N}w(n)u^{i_0}(n)\cdot u^{i_1}(n+a_1)\cdot\ldots \cdot u^{i_k}(n+a_k)\\
=\frac{1}{N}\sum_{n\leq N}F(S^{n-1}w)u^{i_0}(n)\cdot u^{i_1}(n+a_1)\cdot\ldots \cdot u^{i_k}(n+a_k)
\tend{N}{\infty}0
\end{multline*}
by Proposition~\ref{ChS1}.
\end{proof}

\begin{Remark}\label{uw:8}
In Proposition~\ref{pr:2a}, the condition~\eqref{Ch} can be replaced by~\eqref{S}; the proof goes along the same lines.
\end{Remark}

\subsubsection{Sturmian sequences -- background}
\label{se:sturmian-a}
In this section, we give the necessary background on Sturmian sequences.

\begin{Lemma}\label{sturm}
For any $\delta\in [0,1]$ there exists a Sturmian sequence $\eta\in \{0,1\}^{\N^\ast}$ such that for $n$ large enough
$$
\delta n-3<\# 1(B_n) < \delta n+3\text{ for any $B_n\in \{0,1\}^n$ appearing in $\eta$},
$$
where $\# 1(B_n)=\left|\{0\leq k\leq n-1 : B_n(k)=1\}\right|$. Moreover, for any Sturmian sequence $\eta$ there exists a unique $\delta\in [0,1]$ such that the above inequalities hold.
\end{Lemma}
\begin{proof}
For $\delta\in \{0,1\}$ the proof is immediate: we consider the sequences $(0,0,0,\dots)$ and $(1,1,1,\dots)$, respectively. Thus, we may assume that $\delta\in (0,1)$.

Consider the integer grid and a line $L$ in the plane\footnote{Recall that among lines $L$ with a rational slope, we consider only those which do not intersect the nodes of the grid.} and build $\eta\in \{0,1\}^{\N^\ast}$ by writing down $0$ or $1$ depending on whether $L$ intersects a horizontal or a vertical line of the grid (if $L$ meets a node, we write either $0$ or $1$).
 \begin{figure}[h]
 \centering
 \includegraphics[height=5cm]{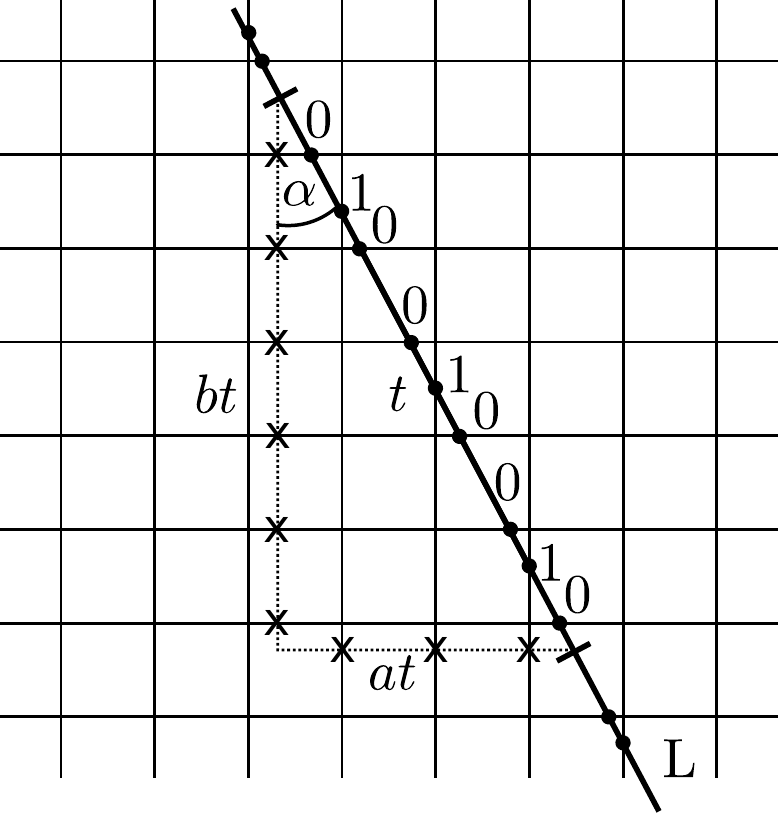}
 \caption{Sturmian sequence}
 \label{fig:1}
 \end{figure}
Given $t>0$, we fix a line segment $L_t$ of the line $L$ of length $t$. Denote by $a$ and $b$ the absolute values of the sine and cosine of the angle at which $L$ intersects the vertical lines of the grid, respectively. Note that $L_t$ intersects as many vertical lines of the grid as the side of the triangle opposed to angle $\alpha$ does (see Figure~\ref{fig:1}). Since this side has length $at$, therefore $L_t$ intersects either $[at]$ or $[at]+1$ vertical lines. This is the number of $1$'s in the corresponding block of $\eta$; we will denote it by $\# 1(L_t)$. In a similar way, the number of $0$'s is equal to $[bt]$ or $[bt]+1$; we denote it by $\# 0(L_t)$. 
Then
\begin{equation}\label{eq:a}
at<\# 1(L_t) \leq at+1 \text{ and }bt<\# 0(L_t)\leq bt+1.
\end{equation}
Therefore
$$
(a+b)t-2<n(L_t)\leq (a+b)t+2,
$$
where $n(L_t)=\# 1(L_t)+\# 0(L_t)$. It follows that
\begin{equation}\label{eq:b}
\frac{n(L_t)-2}{a+b}\leq t<\frac{n(L_t)+2}{a+b}.
\end{equation}
Now, fix $B_n\in \{0,1\}^n$. Then $n=n(L_t)$ for some $t>0$ and for some line segment $L_t$ of length $t$. It follows now by~\eqref{eq:a} and~\eqref{eq:b} that
\begin{multline}\label{eq:c}
\# 1(B_n)=\#1(B_{n(L_t)})=\#1(L_t)\leq at+1\\
<a\frac{n(L_t)+2}{a+b}+1=\frac{a}{a+b}(n+2)+1<\frac{a}{a+b}n+3
\end{multline}
and in a similar way
\begin{equation}\label{eq:d}
\# 1(B_n)>\frac{a}{a+b}(n-2)-1>\frac{a}{a+b}n-3.
\end{equation}
This completes the proof as $\frac{a}{a+b}=\frac{|\tan \alpha|}{1+|\tan \alpha|}$ takes any value between $0$ and $1$.
\end{proof}

\begin{Remark}\label{Uw24}
In particular, it follows from the above lemma that
$$
\delta-\frac3N=\frac{1}{N}(\delta N-3)\leq \frac{1}{N}\sum_{n\leq N}\eta(n)\leq \frac{1}{N}(\delta N+3)=\delta+\frac3N,
$$
whence
\begin{equation}\label{eq+}
\frac{1}{N}\sum_{n\leq N}\eta(n)\to \delta.
\end{equation}
Moreover,
\begin{equation}\label{eq++}
 \delta>0 \text{ for any Sturmian sequence other than }(0,0,0,\dots).
\end{equation}
\end{Remark}

\subsection{Examples}\label{exampla}

\subsubsection{Sequences $z$ with $h_{top}(z^2)=0$}\label{se:5.2.3}
Let $u\in\{-1,1\}^{\N^\ast}$ be a generic point for the $B(1/2,1/2)$ measure and let $\eta\in\{0,1\}^{\N^\ast}$ be a Sturmian sequence. Then, by Proposition~\ref{pr:1}, $z=\eta\cdot u$ satisfies~\eqref{Ch}. Moreover, we have $h_{top}(z^2)=0$. We will now calculate $h_{top}(z)$.

\begin{Prop}\label{pr:2}
For any Sturmian sequence $\eta\in \{0,1\}^{\N^\ast}$ and any $u\in \{-1,1\}^{\N^\ast}$ generic for the Bernoulli measure $B(1/2,1/2)$, the sequence $z:=\eta\cdot u$ satisfies~\eqref{Ch}, $h_{top}(z^2)=0$ and $h_{top}(z)=\delta$, where $\delta$ is uniquely determined by the second assertion of Lemma~\ref{sturm}. Moreover, for every $\delta\in[0,1]$, the pair $(0,\delta)$ is realized as $(h_{top}(z^2),h_{top}(z))$ for a sequence $z$ satisfying~\eqref{Ch}.
\end{Prop}
\begin{proof}
Let $\eta\in \{0,1\}^{\N^\ast}$ be a Sturmian sequence, and $\delta$ be as in Lemma~\ref{sturm}, i.e.
\begin{equation}\label{del}
\delta n-3<\# 1(B_n) < \delta n+3
\end{equation}
for any $B_n\in \{0,1\}^n$ appearing in $\eta$. Fix a generic point $u\in \{-1,1\}^{\N^\ast}$ for the measure $B(1/2,1/2)$. Then by Proposition~\ref{pr:1}, $z=\eta\cdot u$ satisfies~\eqref{Ch}.

Since $z$ satisfies~\eqref{Ch}, it is generic for the relatively independent extension of the measure given by the block frequencies in $z^2=\eta$. In particular, given a block $C$ appearing in $\eta=z^2$, and $B$ such that $B^2=C$, $B$ will appear in $z$. Hence
\begin{equation}\label{eq:eee}
p_{z^2}(n)\cdot 2^{\delta n-3}<p_z(n)< p_{z^2}(n)\cdot 2^{\delta n+3},
\end{equation}
which yields $h_{top}(z)=\delta$.
\end{proof}

\subsubsection{Sequences $z$ with arbitrary $h_{top}(z^2)>0$}\label{se:5.2.4}
We will now give examples of $z$ satisfying~\eqref{Ch} with arbitrary $h_{top}(z^2)>0$.
\begin{Prop}\label{PR5}
For any Sturmian sequence $\eta\in \{0,1\}^{\N^\ast}$ and any $u\in \{-1,0,1\}^{\N^\ast}$ generic for the Bernoulli measure $B(1/4,1/2,1/4)$, the sequence $z:=\eta\cdot u$ satisfies~\eqref{Ch}, $h_{top}(z^2)=\delta$ and $h_{top}(z)=\delta\log 3$, where $\delta$ is uniquely determined by the second assertion of Lemma~\ref{sturm}. Moreover, each pair $(\delta,\delta\log3)$ is realized as $(h_{top}(z^2),h_{top}(z))$ for a sequence $z$ satisfying~\eqref{Ch}.
\end{Prop}

\begin{proof}
Take $u\in \{-1,0,1\}^{\N^\ast}$ generic for the Bernoulli measure $B(1/4,1/2,1/4)$. Notice that $u^2$ is generic for $B(1/2,1/2)$, and that $B(1/4,1/2,1/4)$ is the relatively independent extension of $B(1/2,1/2)$. By Lemma~\ref{rownow}, $u$ satisfies~\eqref{Ch}. Let $\eta\in \{0,1\}^{\N^\ast}$ be a Sturmian sequence, and let $\delta$ be as in Lemma~\ref{sturm}. By Proposition~\ref{pr:2a}, $z:=u\cdot \eta$ also satisfies~\eqref{Ch}.

Notice that any block $B$ appearing in $z$ arises by replacing some of the $1$'s in a block $C$ appearing in $\eta$ by $0$'s or $-1$'s. Moreover, all blocks of this form appear in $z$. Thus,
$$
3^{\delta n-3}\leq p_z(n)\leq (n+1)\cdot 3^{\delta n +3},
$$
whence $h_{top}(z)=\delta\log3$. In a similar way, we obtain $h_{top}(z^2)=\delta$, which completes the proof.
\end{proof}

\begin{Remark}\label{uw622}
Suppose that $b_k=a_k^2$, $k\geq 1$ are pairwise relatively prime and let $\mob_\mathscr{B}$ be given by formula~\eqref{gen-mob}. Let $z\in\{-1,0,1\}^{\N^\ast}$ be a sequence satisfying~\eqref{Ch}, such that $z^2=\mob_\mathscr{B}^2$ (we can get such a sequence as the product of $\mob_\mathscr{B}^2$ and a sequence of $-1$'s and $1$'s which is generic for $B(1/2,1/2)$, noting that $\mob_\mathscr{B}^2$ is completely deterministic by~\cite{B-Free}, and using Propositin~\ref{pr:1}). By Theorem 5.3. in~\cite{B-Free}, we have
$h
_{top}(z^2)=\prod_{k\geq 1}\left(1-\frac{1}{b_k}\right).
$
Moreover, the same arguments yield
$
h_{top}(z)=\log 3\cdot\prod_{k\geq 1}\left(1-\frac{1}{b_k}\right).
$
Recall (cf.~\cite{Sarnak,Peckner}) that in the classical case when $z=\mob$, we have
$\prod_{k\geq 1}\left(1-\frac{1}{b_k}\right)=\frac{6}{\pi^2}$.
\end{Remark}

\section{Toeplitz sequences correlating with a given sequence, and their topological entropy}\label{se:6}

Since the Sarnak conjecture holds for periodic sequences, the following question arises:
\begin{quote}
Are all sequences that display some strong periodic structure orthogonal to $\mob$?
\end{quote}
Toeplitz sequences (see Section~\ref{dodacnazwe}) are a natural class to consider in this context, since they are explicitly given as some limits of periodic  sequences: indeed, any block appearing in a Toeplitz sequence, appears in it periodically (the period may vary, depending on the chosen block). It was however already shown in~\cite{Ab-Ka-Le} that there are Toeplitz sequences that are not orthogonal to $\mob$.\footnote{The entropy of such sequences was not computed in~\cite{Ab-Ka-Le}.} The aim of this section is to work in an abstract setting, dealing, instead of $\mob$, with a sequence $z\in\{-1,0,1\}^{\N^\ast}$ satisfying some additional assumptions. Under these assumptions, we will construct Toeplitz sequences $t=(t_n)$ such that
\begin{equation}\label{eq:top}
\frac{1}{N}\sum_{n\leq N}t_n\cdot z(n)\not\to 0
\end{equation}
and show that $h_{top}(t)>0$, giving more precise entropy estimates.

The starting point for our constructions is the following simple observation: if the upper density of $1$'s in $z^2$ is positive then
$$
\frac{1}{N}\sum_{n\leq N}z(n)\cdot z(n)\not\to 0.
$$
The underlying idea of the constructions is to find a Toeplitz sequence $t$ which has ``as much as possible in common'' with the sequence $z$ under consideration.

We apply our results to the following two classes of sequences:
\begin{enumerate}[label=(\alph*)]
\item\label{AAA}
sequences satisfying~\eqref{Ch}, related to Sturmian sequences (see Section~\ref{se:5.2.3} and Section~\ref{se:5.2.4});
\item\label{BBB}
$z=\mob$, $z=\mob_\mathscr{B}$ and any sequence $z$ such that $z^2=w^2$, where $w$ is as in~\ref{AAA}.
\end{enumerate}
Notice that in case~\ref{AAA}, in view of Theorem~\ref{ChS},~\eqref{eq:top} clearly implies that $t$ is not completely deterministic, so, in particular, $h_{top}(t)>0$. Therefore, what we are really interested in, are the obtained entropy (lower) estimates. In case~\ref{BBB}, we cannot refer to~\eqref{Ch} anymore to show that $h_{top}(t)>0$, it needs to be shown separately. Note however that
our entropy estimates are not as precise as in case~\ref{AAA} (the reason is that we have less knowledge about $z$). It is also unclear whether the constructed Toeplitz sequences are not completely deterministic.

\subsection{Abstract setting}
Let $z\in \{-1,0,1\}^{\N^*}$ be such that
\begin{equation}\label{aaa}
\liminf_{N\to\infty}\frac{1}{N}\sum_{n\leq N}z^2(n)>0.
\end{equation}
We will show that there exists a Toeplitz sequence $t\in \{-1,0,1\}^{\N^*}$ which correlates with $z$, i.e.
\beq\label{aaa?}
 \liminf_{N\to\infty} \dfrac{1}{N}\sum_{n\leq N} t_n\cdot z(n) >0
\eeq
(see Proposition~\ref{pr:7} below). Moreover, under some additional assumptions on $z$, we will give estimates for $h_{top}(t)$. More precisely, we will prove the following.

\begin{Th}\label{tw:abs2}
Suppose that~\eqref{aaa} holds. If, moreover, $z$ is such that:
\begin{enumerate}[label=(\alph*)]
\item\label{ca}
$z$ is quasi-generic for some $\widetilde{\nu}\in\mathcal{P}_S(\{-1,0,1\}^{\mathbb{N}^\ast})$,
\item\label{cb}
${H}:=h_{top}(\text{supp}(\widetilde{\nu}))>0$,
\item\label{cc}
there exist $q\geq 2$ and $b\geq 1$ such that, for all $m\geq 1$, the number of ergodic components of the action of $S^{q^m}$ on $(\{-1,0,1\}^{\N^*},\widetilde{\nu})$ is bounded by $b$,
\end{enumerate}
then, for any $\vep>0$, $t$ can be chosen so that $h_{top}(t)\geq (1-\vep){H}$ and~\eqref{aaa?} holds.
\end{Th}

\begin{Remark}
Note that condition~\ref{cc} above holds for an arbitrary $q\geq 2$ whenever
$(S,\widetilde\nu)$ is ergodic and there exists $b\geq 1$ such that for any rational eigenvalue $\lambda$ of $(S,\widetilde\nu)$, $\lambda^c=1$ for some $1\leq c\leq b$. In particular,~\ref{cc} holds if
$(S,\widetilde\nu)$ is totally ergodic.
\end{Remark}

\begin{Th}\label{tw:abs1}
Suppose that~\eqref{aaa} holds. If moreover $z$ is such that
\begin{enumerate}[label=(\alph*')]
\item\label{ca1}
$z^2$ is quasi-generic for some $\nu\in\mathcal{P}_S(\{-1,0,1\}^{\mathbb{N}^\ast})$,
\item\label{cb1}
${H}:=h_{top}(\text{supp}({\nu}))>0$,
\item\label{cc1}
there exist $q\geq 2$ and $b\geq 1$ such that, for all $m\geq 1$, the number of ergodic components of the action of $S^{q^m}$ on $(\{0,1\}^{\N^*},{\nu})$ is bounded by $b$,
\end{enumerate}
then, for any $\vep>0$, $t$ can be chosen so that $h_{top}(t)\geq (1-\vep){H}$ and~\eqref{aaa?} holds.
\end{Th}

\begin{Remark}\label{siodemki}
Although Theorem~\ref{tw:abs2} seems to give a better lower entropy estimation than Theorem~\ref{tw:abs1}, it cannot be applied in many interesting cases (see Section~\ref{zastM}) because of the assumption~\ref{cc} which we are not able to verify. In such cases, we apply Theorem~\ref{tw:abs1}. Independently of us, Downarowicz and Kasjan proved in~\cite{Do-Ka} a result similar to Theorem~\ref{tw:abs1} in the particular case $z=\mob$.
\end{Remark}

The proofs of Theorems~\ref{tw:abs2} and~\ref{tw:abs1} go along the same lines. Since they are quite technical, they will be split into several sections.

\subsubsection{A Toeplitz sequence correlating with $z$}
Fix some $q\ge2$ and, for each $j\ge1$, consider the arithmetic progression
\[
 A_j \egdef \{j+nq^j : n\ge0\}\subset\N^*.
\]
\begin{Def}
We say that $j\in\N^*$ is \emph{initial} if there is no $j'<j$ with $j\in A_{j'}$. Then,
\[
 \{A_j : j\text{ initial}\}\text{ is a partition of $\N^*$.}
\]
When $j$ is initial, we denote by $A_j^*$ the set $A_j\setminus\{j\}$. Elements of $A_j^*$ for some initial $j$ are said to be \emph{non-initial}. We denote the set of all non-initials by $\cn$.
\end{Def}

The Toeplitz sequence we are interested in, is the sequence $t=(t_n)_{t\in\NS}\in\{-1,0,1\}^{\NS}$ defined by
\begin{equation}\label{top}
 t_n \egdef \begin{cases}
       z(n) \text{ if $n$ is initial},\\
	   z(j) \text{ if $n\in A_j^*$ for some initial $j$}.
      \end{cases}
\end{equation}

\begin{Lemma}\label{lemma:density_non-initials}
For any $N\geq 1$, we have
$$
\frac{\mathcal{N}\cap\{1,\ldots,N\}}{N}\leq\frac{1}{q-1}.
$$
\end{Lemma}
\begin{proof}
Let $j$ be initial. Since the difference of two consecutive terms in $A_j^*$ is $q^j$, and since the first term of the arithmetic progression $A_j$ is missing in $A_j^*$, we have
$\left| A_j^* \cap \{1,\ldots,N\} \right| \le \dfrac{N}{q^j}$ for each $N\ge 1$.
Hence,
\[
 \dfrac{\left| \cn \cap \{1,\ldots,N\} \right|}{N} \le \sum_{j\not\in\cn} \dfrac{1}{q^j} \le \sum_{j\ge1} \dfrac{1}{q^j} = \dfrac{1}{q-1}.
\]
\end{proof}

\begin{Prop}\label{pr:7}
 \label{prop:corel}
 Suppose that~\eqref{aaa} holds.
Then, for $q$ sufficiently large, the Toeplitz sequence $t$ defined by~\eqref{top} correlates with $z$, \textit{i.e.}
\[
 \liminf_{N\to\infty} \dfrac{1}{N}\sum_{n\leq N} t_n\cdot z(n) >0.
\]
\end{Prop}

\begin{proof}
We have
\begin{equation}\label{eq:D1}
\frac{1}{N}\sum_{n\leq N}t_n\cdot z(n)=\frac{1}{N}\sum_{n\leq N,n\in\cn}t_n\cdot z(n)+\frac{1}{N}\sum_{n\leq N, n\not\in\cn}t_n\cdot z(n),
\end{equation}
where
\begin{equation}\label{eq:D2}
\left| \frac{1}{N}\sum_{n\leq N,n\in\cn} t_n\cdot z(n) \right|\leq \frac{|\cn\cap \{1,\dots,N\}|}{N}\leq \frac{1}{q-1}
\end{equation}
by Lemma~\ref{lemma:density_non-initials}. Moreover, using once more
Lemma~\ref{lemma:density_non-initials}, we have
\begin{align*}
\frac{1}{N}&\sum_{n\leq N, n\not\in\cn}t_n\cdot z(n)=\frac{1}{N}\sum_{n\leq N,n\not\in\cn}z^2(n)\\
&=\frac{|\text{supp}(z)\cap \cn^c\cap \{1,\dots,N\}|}{N}=1-\frac{|((\text{supp}(z))^c\cup \cn)\cap\{1,\dots,N\}|}{N}\\
&\geq 1-\frac{|(\text{supp}(z))^c\cap \{0,\dots, N\}|}{N}-\frac{|\cn\cap\{1,\dots,N\}|}{N}\\
&\geq \frac{|\text{supp}(z)\cap \{1,\dots,N\}|}{N}-\frac{1}{q-1}=\frac{1}{N}\sum_{n\leq N}z^2(n)-\frac{1}{q-1}.
\end{align*}
Therefore
$$
\frac{1}{N}\sum_{n\leq N}t_n\cdot z(n)\geq\frac{1}{N}\sum_{n\leq N}z^2(n)-\frac{2}{q-1}.
$$
By~\eqref{aaa}, the latter expression is bounded below by a fixed positive number whenever $q$ and $N$ are large enough, which completes the proof.
\end{proof}

\subsubsection{Two types of non-initial numbers}
Fix an integer $m\ge 1$. For any integer $k\ge0$, we consider the interval
$$
I_{m,k} \egdef \Bigl( k q^m ,\ (k+1) q^m \Bigr]\cap \NS.
$$
We distinguish two types of non-initials in $I_{m,k}$:
\begin{Def}
A non-initial in $I_{m,k}$ is said to be:
\begin{itemize}
\item
\emph{of type~1} if it belongs to some $A_j^*$ with $j\le m$,
\item
\emph{of type~2} if it belongs to some $A_j^*$ with $j> m$.
\end{itemize}
\end{Def}
\begin{Remark}\label{uw:31}
Observe that, if for some $k\ge1$ and some $1\le r\le q^m$, $kq^m+r$ is a non-initial of type~1 in $I_{m,k}$, then for any other $k'\ge1$, $k'q^m+r\in I_{m,k'}$ is also a non-initial of type~1 (since it belongs to the same $A_j^*$). Hence, the pattern formed by non-initials of type~1 inside $I_{m,k}$ does not depend on which $k\ge1$ we consider.

On the other hand, consider $A_{m+h}^*$ for some $h\ge1$. This set of non-initial numbers intersects $I_{m,k}$ every $q^h$-th integer $k$, and when it does, the single non-initial point of type 2 in the intersection is always of the form $kq^m+r$ for some $r$ depending on $h$ but not on $k$.
\end{Remark}
\subsubsection{The end of the interval}
Fix additionally an integer $1\le\ell<m$, set $L\egdef q^\ell$, and consider the last $L$ elements of $I_{m,k}$:
\[
 I_{m,k,L} \egdef \Bigl( (k+1)q^m-L ,\ (k+1) q^m \Bigr]\cap \NS.
\]

\begin{Def}
 We say that the integer $k$ is \emph{good} if the only non-initial integers in $I_{m,k,L}$ are of type~1.
\end{Def}

By Remark~\ref{uw:31}, for all good $k$'s, the pattern formed by non-initial integers inside $I_{m,k,L}$ is always the same.

\begin{Prop}
 \label{prop:not-good-k}
 The upper density of the set
 \[
  \M\egdef\{k\ge1: k\text{ is not good}\}
 \]
 is bounded from above by $1/q^{q^m-m-L}$.
\end{Prop}

\begin{proof}
Let $n\in I_{m,k,L}$ be a non-initial element of type 2. Then $n\in A_j^*$ for some initial $j> m$, and we have $n \equiv j\bmod q^j$, hence also $n \equiv j \bmod q^m$.
This and the definition of $I_{m,k,L}$ imply $j> q^m-L$, i.e.
\begin{equation}\label{eq:insideImkl}
 \parbox{0.8\linewidth}{the non-initials of type 2 inside $I_{m,k,L}$ belong to some $A_j^*$ with $j>q^m-L$.}
\end{equation}
Now, fix an initial $j>m$ and let $k_0$ be such that $j\in I_{m,k_0}$. Then
$$
\{k\geq 0 : I_{m,k}\cap A_j^* \neq \varnothing \}=\{k_0+i\cdot q^{j-m} : i\geq 1\}.
$$
Hence, for any $K\ge1$, we have
\[
 \dfrac{1}{K} \left| \{0\le k<K: I_{m,k}\cap A_j^*\neq\varnothing\}\right| \le \dfrac{1}{q^{j-m}}.
\]
It follows that
\begin{multline*}
 \dfrac{1}{K} \left| \{0\le k<K: I_{m,k}\cap A_j^*\neq\varnothing\text{ for some } j> q^m-L\}\right| \\
 \le \sum_{j>q^m-L}\dfrac{1}{q^{j-m}} 
 = \dfrac{1}{q^{q^m-m-L}}\sum_{h\ge 1}\dfrac{1}{q^h} < \dfrac{1}{q^{q^m-m-L}}.
\end{multline*}
In view of~\eqref{eq:insideImkl}, this ends the proof.
\end{proof}

\subsubsection{Density of non-initials of type~1 inside $I_{m,k,L}$}

We want now to bound the density of non-initials of type~1 inside $I_{m,k,L}$ (which are the only non-initials in this interval when $k$ is good).

\begin{Lemma}
 \label{lemma:type1}
 Let $n\in I_{m,k,L}$ be a non-initial of type~1. Then $n\in A_j^*$, where $j$ satisfies $j>q^j-L$ (\emph{cf.}~\eqref{eq:insideImkl}).
\end{Lemma}

\begin{proof}
 Let $n$ be a non-initial of type~1 inside $I_{m,k,L}$. Then, by the definition of type~1, there exists an initial $j$ with $j\le m$ such that $n\in A_j^*$. Thus $n\equiv j\bmod q^j$, and also
 \[
  n \equiv (k+1)q^m + j\bmod q^j.
 \]
 Since $n\le(k+1)q^m$, there exists an integer $s\ge1$ with
 \[
  n = (k+1)q^m + j-s\,q^j.
 \]
 But $n> (k+1)q^m - L$, hence
 \[
  (k+1)q^m - L < (k+1)q^m + j-s\,q^j \le (k+1)q^m + j-\,q^j,
 \]
and the assertion follows.
\end{proof}

\begin{Prop}
 \label{prop:density}
 For $k\ge 1$, the proportion of non-initial elements of type~1 inside $I_{m,k,L}$ is equal to
 \begin{equation}
  \label{eq:density}
  \dfrac{1}{q}+\dfrac{1}{q^2}+\cdots+\dfrac{1}{q^\ell}\cdot
 \end{equation}
\end{Prop}
\begin{proof}
 First, let us show that there are no non-initial elements of type~1 inside $I_{m,k,L}$ which are in some $A_j^*$ with $j>\ell$. Indeed, suppose that such an element exists. Then, we can write $j=\ell + s $ for some integer $s\ge1$, and Lemma~\ref{lemma:type1} gives
 \[
  \ell+s>q^{\ell+s}-q^\ell=q^\ell(q^s-1).
 \]
If $q^s-1=1$, then $q=2$ and $s=1$, and we get $\ell\ge 2^\ell$, which is impossible. Otherwise, using the inequality $\alpha\beta\ge\alpha+\beta$ for any $\alpha\ge2$, $\beta\ge2$, we obtain
\[\ell+s \ge q^\ell + q^s,\]
which is also impossible since $\ell<q^\ell$ and $s<q^s$.

It remains to estimate the contribution of non-initial elements of type~1 which are in some $A_j^*$ with $j\le\ell$. For each such $j$, since $q^j$ divides the length $L=q^\ell$ of $I_{m,k,L}$, we have
 \[
  \dfrac{\left|A_j\cap I_{m,k,L}\right|}{L} = \dfrac{1}{q^j}.
 \]
Since $j$ ranges from $1$ to $\ell$,~\eqref{eq:density} follows.
\end{proof}

\subsubsection{Ergodic components}

\begin{Prop}\label{erg-comp}
Let $\A$ be a finite alphabet, and let $\nu$ be a shift-invariant probability measure on $\A^{\NS}$. Let $n\geq 1$ and suppose that $(S^n,\nu)$ has $b\geq 1$ ergodic components. If
$\nu=\frac{1}{n}\sum_{s=0}^{n-1}(S^s)_{\ast}\eta$,
where $\eta$ is $S^n$-invariant then
$$
\nu=\frac{1}{n}\left(\left[\frac{n}{b!} \right]\sum_{s=0}^{b!-1}(S^s)_\ast(\eta)+\sum_{s=0}^{(n \bmod{b!})-1}(S^s)_\ast(\eta) \right).
$$
\end{Prop}
\begin{proof}
For $0\leq i\leq b-1$, let $\rho^{(i)}$ be the ergodic components of $(S^n,\nu)$, i.e.
$$
\nu=\alpha_0\rho^{(0)}+\dots+\alpha_{b-1}\rho^{(b-1)}
$$
for some $0< \alpha_0,\dots,\alpha_{b-1}<1$, $\sum_{i=0}^{b-1} \alpha_i=1$.
Then
$$
\nu=S_\ast(\nu)=\alpha_0S_\ast(\rho^{(0)})+\ldots \alpha_{b-1}S_\ast(\rho^{(b-1)}).
$$
For any $0\leq i\leq b-1$, $S_\ast(\rho^{(i)})$ is again an ergodic $S^n$-invariant measure. Since the ergodic decomposition of $(S^n,\nu)$ is unique, this yields a permutation
$\pi\colon \{0,1,\dots,b-1\}\to \{0,1,\dots,b-1\}$,
$$
\pi(i)=j \iff S_\ast(\rho^{(i)})=\rho^{(j)}.
$$
Clearly, $\pi^{b!}=\Id$. Now, since $\eta \ll \nu$,
$$
\eta=\beta_0\rho^{(0)}+\dots+ \beta_{b-1}\rho^{(b-1)}
$$
for some $0\leq \beta_i\leq 1$, $\sum_{i=0}^{b-1}\beta_i=1$, whence $(S^{b!})_\ast(\eta)=\eta$ and the assertion follows.
\end{proof}

\begin{Cor}\label{wn5}
Under the assumptions of Proposition~\ref{erg-comp}, whenever $n\geq2b!$, for any finite family of sets $\{A_i : i\in I\}$ with $\nu(A_i)>0$, there exists $0\leq {s}\leq b!-1$ such that
$$
\left|\left\{i\in I: (S^{s})_\ast(\eta) (A_i)\geq\frac{1}{2}\nu(A_i)\right\}\right|\geq \frac{|I|}{b!}.
$$
\end{Cor}
\begin{proof}
By Proposition~\ref{erg-comp}, for every $i\in I$, there exists $0\leq s_i\leq b!-1$ such that
$$
\frac{1}{n}\left(\left[\frac{n}{b!} \right] +1\right) (S^{s_i})_\ast(\eta)(A_i)\geq \frac{1}{b!}\nu(A_i).
$$
Since $\frac{1}{n}\left(\left[\frac{n}{b!} \right] +1\right)\leq\frac{2}{b!}$, we have $(S^{s_i})_\ast(\eta)(A_i)\geq \frac{1}{2}\nu(A_i)$ and the result easily follows by Fubini's argument.
\end{proof}

\begin{Prop}\label{ogolny}
Fix $\vep>0$. Let $\A$ be a finite alphabet, fix $w\in A^{\N^*}$ and suppose that the following conditions hold:
\begin{enumerate}[label=(\alph*)]
\item\label{ca2}
$w$ is quasi-generic for some shift-invariant measure ${\nu}$ for which
\item\label{cb2}
$H:=h_{top}(\text{supp}({\nu}))>0$,
\item\label{cc2}
there exist $q\geq 2$ and $b\geq 1$ such that, for all $m\geq 1$, the number of ergodic components of the action of $S^{q^m}$ on $(\A^{\NS}\!,{\nu})$ is bounded by $b$.
\end{enumerate}
Then, for all $\ell\geq 1$ large enough, there exists $\tau_\ell>0$ such that, for all $m>\ell$, we can find $0\leq {s}\leq b!-1$ satisfying (using as before the notation $L\egdef q^\ell$)
\begin{multline*}
\Biggl|\Biggl\{C\in \A^{L-b!}:\Biggr.\Biggr.\\
\Biggl.\Biggl. \limsup_{K\to \infty}\frac{1}{K}\sum_{0\le k < K}\raz_C\left(S^{(k+1)q^m-L+1+{s}}w\right)\geq \tau_\ell\Biggr\}\Biggr|\geq 2^{{H}(1-\vep)L}.
\end{multline*}
\end{Prop}

\begin{proof}
It follows by~\ref{cb2} that
\begin{equation}\label{town5}
\left| \left\{C\in \A^{L-b!}: \nu(C)>0 \right\}\right|
\geq 2^{H(1-\vep/3)(L-b!)}\geq 2^{H(1-2\vep/3)L},
\end{equation}
when $\ell$ (and hence also $L$) is large enough. Fix such an $\ell$, which additionally satisfies the following inequality:
\begin{equation}\label{townL}
\frac{2^{H(1-2\vep/3)L}}{b!}\geq 2^{H(1-\vep)L}.
\end{equation}
Let
\begin{equation}\label{TAU}
\tau_\ell:=\frac{1}{2}\min\left\{\nu(C): C\in \A^{L-b!},\ \nu(C)>0 \right\}
\end{equation}
and take $m>\ell$.

By~\ref{ca2}, we may find an increasing sequence $(N_j)$ such that
$$
\nu=\lim_{j\to \infty}\delta_{S,N_j,w}.
$$
Since $\left[\frac{N_j}{q_m}\right]q_m/N_j\to1$, by replacing $N_j$ with
$\left[\frac{N_j}{q_m}\right]q_m$ if necessary, we can assume that $q_m|N_j$ for $j\geq1$.
Passing to a subsequence if necessary, we can further assume the existence of
$$
\eta:=\lim_{j\to\infty}\delta_{S^{q_m},N_j/q_m,w}.
$$
Then $\nu=\frac{1}{q^m}\sum_{s=0}^{q^m-1}(S^s)_\ast(\eta)$. From~\eqref{TAU}, Corollary~\ref{wn5} applied to $\{C\in \A^{L-b!}: \nu(C)>0\}$,~\eqref{town5} and~\eqref{townL}, it follows that there exists $0\leq {s}\leq b!-1$ such that
\begin{multline}\label{town2}
 \left| \left\{C\in \A^{L-b!}\colon (S^{q^m-L+1+{s}})_\ast(\eta)(C)\geq \tau_\ell \right\}\right|\\
\geq \left| \left\{C\in \A^{L-b!}\colon (S^{q^m-L+1+{s}})_\ast(\eta)(C)\geq\frac{1}{2}\nu(C) \right\}\right|\\
\geq\frac{2^{H(1-2\vep/3)L}}{b!}\geq 2^{H(1-\vep)L}.
\end{multline}
Notice that, if $C$ is a cylinder such that $(S^{q^m-L+1+{s}})_\ast(\eta)(C)\geq a$ for some $a>0$, then
$$
\limsup_{K\to \infty}\frac{1}{K}\sum_{0\le k < K}\raz_C(S^{(k+1)q^m-L+1+{s}}w)\geq a.
$$
This and~\eqref{town2} imply
\begin{multline*}
\left| \left\{C\in \A^{L-b!}\colon \limsup_{K\to \infty}\frac{1}{K}\sum_{0\le k < K}\raz_C(S^{(k+1)q^m-L+1+{s}}w)\geq\tau_\ell\right\}\right|\\
\geq\frac{2^{H(1-2\vep/3)L}}{b!}\geq 2^{H(1-\vep)L},
\end{multline*}
which completes the proof.
\end{proof}

An immediate consequence of Proposition~\ref{ogolny} are the following two corollaries.
\begin{Cor}\label{cor:6}
Let $\vep>0$ and suppose that the assumptions~\ref{ca1},~\ref{cb1} and~\ref{cc1} of Theorem~\ref{tw:abs1} hold. Then for all $\ell\geq 1$ large enough, there exists $\tau_\ell>0$ such that, for all $m>\ell$, we can find $0\leq {s}\leq b!-1$ satisfying
\begin{multline*}
\biggl|\biggl\{C\in \{0,1\}^{L-b!}\colon\biggr.\biggr.\\
\biggl.\biggl. \limsup_{K\to \infty}\frac{1}{K}\sum_{0\le k < K}\raz_C\left(S^{(k+1)q^m-L+1+{s}}z^2\right)\geq \tau_\ell\biggr\}\biggr| \geq 2^{{H}(1-\vep/2)L}.
\end{multline*}
\end{Cor}

\begin{Cor}\label{cor:7}
Let $\vep>0$ and suppose that the assumptions~\ref{ca},~\ref{cb} and~\ref{cc} of Theorem~\ref{tw:abs2} hold. Then for all $\ell\geq 1$ large enough, there exists $\tau_\ell>0$ such that, for all $m>\ell$, we can find $0\leq {s}\leq b!-1$ satisfying
\begin{multline*}
\biggl|\biggl\{C\in \{-1,0,1\}^{L-b!}\colon\biggr.\biggr.\\
\biggl.\biggl. \limsup_{K\to \infty}\frac{1}{K}\sum_{0\le k < K}\raz_C\left(S^{(k+1)q^m-L+1+{s}}z\right)\geq \tau_\ell\biggr\}\biggr|\geq 2^{{H}(1-\vep/2)L}.
\end{multline*}
\end{Cor}

\subsubsection{Entropy estimates}
\begin{proof}[Proof of Theorem~\ref{tw:abs1}]
We will need the following notation: if $A=\{a_1<a_2<\cdots<a_r\}$ is a finite subset of $\NS$, and if $x=(x(n))_{n\in\NS}$ is a sequence in $\{0,1\}^{\N^\ast}$, we denote by $x(A)$ the finite sequence
 \[
  x(A) \egdef \bigl(x(a_1),\ldots, x(a_r)\bigr) \in\{0,1\}^r.
 \]

Fix $\vep>0$. Replacing $q$ by $q^r$ if necessary, for some large $r$ (which does not alter the validity of \ref{cc1}), we can assume that $q$ is large enough
to satisfy the assertion of Proposition~\ref{pr:7}, and also that
\begin{equation}
 \label{eq:q_large}
 \dfrac{1}{q-1} < \frac{\vep}{2}H.
\end{equation}
Let $\ell$ be an integer large enough to satisfy the assertion of Corollary~\ref{cor:6}, and set $L\egdef q^\ell$. Then, by Proposition~\ref{prop:not-good-k}, we can take $m$ large enough so that the upper density of the set of integers $k$ which are not good is strictly less than $\tau_\ell$. Let $0\le {s}\leq b!-1$ be given by Corollary~\ref{cor:6}. Then, for any $C\in \{0,1\}^{L-b!}$ satisfying
$$
\limsup_{K\to \infty}\frac{1}{K}\sum_{0\leq k<K}\raz_C\left(S^{(k+1)q^m-L+1+s}z^2\right)\geq \tau_\ell,
$$
there exist infinitely many good integers $k$ such that the block corresponding to the cylinder set $C$ appears at position ${s}$ of $I_{m,k,L}$ in the sequence $z^2$. Since, by Corollary~\ref{cor:6}, the number of such cylinder sets is at least $2^{{H}(1-\vep/2)L}$, we can deduce that
 \begin{equation}
 \label{eq:many1}
  \bigl|
  \{ z^2 (I_{m,k,L}): k\text{ good}\}
  \bigr|
  \ge 2^{{H}(1-\vep/2)L}.
 \end{equation}

We will show now that for $\ell$ sufficiently large, the number of blocks of length $L$ in $t$ is at least $2^{{H}(1-\vep)L}$, more precisely,
we claim that
$$
\left|\left\{t(I_{m,k,L}): k\text{ is good}\right\}\right|\geq 2^{{H}(1-\vep)L}.
$$
We have
$$
I_{m,k,L} = A_{m,k,L}\sqcup B_{m,k,L},
$$
where
\begin{align*}
&A_{m,k,L}:=\{n\in I_{m,k,L}: n\in\cn\},\\
&B_{m,k,L}:=\{n\in I_{m,k,L}: n\notin\cn\}.
\end{align*}
By Proposition~\ref{prop:density},
$$
|A_{m,k,L}|=L\left(\frac{1}{q}+\dots+\frac{1}{q^l}\right)<\frac{L}{q-1} \quad \text{ whenever }k\text{ is good},
$$
whence
\beq\label{blisko}
|\{z^2(A_{m,k,L}): k\text{ is good}\}|\leq 2^{\frac{L}{q-1}}.\eeq

Observe also that, when $k$ is good, the relative positions of $A_{m,k,L}$ and $B_{m,k,L}$ inside $I_{m,k,L}$ are always the same. Hence,
\begin{multline}\label{v1}
 \bigl|
  \{ z^2 (I_{m,k,L}): k\text{ is good}\}
  \bigr|\\
  \le
  \bigl|
  \{ z^2 (A_{m,k,L}): k\text{ is good}\}
  \bigr|
  \cdot
  \bigl|
  \{ z^2 (B_{m,k,L}): k\text{ is good}\}
  \bigr|.
\end{multline}
Therefore, in view of~\eqref{blisko}, \eqref{eq:many1}
and~\eqref{eq:q_large}, we obtain
\begin{multline}\label{v2}
|\{z^2(B_{m,k,L})\colon k\text{ is good}\}|\geq\frac{|\{ z^2 (I_{m,k,L}): k\text{ is good}\}|}{|\{ z^2 (A_{m,k,L}): k\text{ is good}\}|}\\
\geq\frac{1}{2^{\frac{L}{q-1}}}|\{ z^2 (I_{m,k,L}): k\text{ is good}\}|\geq\frac{2^{{H}(1-\vep/2)L}}{2^{\frac{L}{q-1}}}\geq 2^{{H}(1-\vep)L}.
\end{multline}
Finally
\begin{multline}\label{v3}
\left|\left\{t(I_{m,k,L}): k\text{ is good}\right\}\right|\geq \left|\{t(B_{m,k,L}): k\text{ is good}\}\right|\\
=|\{z(B_{m,k,L}): k\text{ is good}\}|\geq |\{z^2(B_{m,k,L}): k\text{ is good}\}|\geq 2^{{H}(1-\vep)L}
\end{multline}
and the result follows.
\end{proof}

\begin{proof}[Proof of Theorem~\ref{tw:abs2}]
The proof goes along the same lines as the proof of Theorem~\ref{tw:abs1} (instead of $\{0,1\}$, we consider the alphabet $\{-1,0,1\}$). First, (using Corollary~\ref{cor:7} instead of Corollary~\ref{cor:6}) we show that
 \begin{equation}
  \bigl|
  \{ z (I_{m,k,L}): k\text{ good}\}
  \bigr|
  \ge 2^{{H}(1-\vep/2)L}
 \end{equation}
(cf.\ formula~\eqref{eq:many1}). The formulas~\eqref{v1} and~\eqref{v2} are still valid, with $z$ playing now the role of $z^2$. In~\eqref{v3} it suffices to remove the part involving $z^2$ to obtain the result.
\end{proof}

\subsection{Applications}

\subsubsection{$\mob$ and its generalizations: $h_{top}(z^2)>0$}
\label{zastM}

Let $\mathscr{B}=\{b_k: k\geq 1\}$ be a set of pairwise coprime numbers with $b_k=a_k^2$ and let $z(n)=\mob_\mathscr{B}(n)$ be given by formula~\eqref{gen-mob}. Then the following is true:
\begin{enumerate}[label=(\alph*)]
\item
The point $z^2$ is generic for some measure $\nu$. Moreover, for any block $C$ appearing in $z^2$, $\nu(C)>0$
(for $z=\mob$, see~\cite{Peckner} and for the general case, see~\cite{B-Free}).
\item
$\frac{1}{N}\sum_{n\leq N}z^2(n)\tend{N}{\infty} h_{top}(\text{supp}(\nu))= h_{top}(z^2)>0$.
\item
$(S,\nu)$ has purely discrete spectrum. Moreover, for $q$ prime:
\begin{itemize}
\item
if $q\nmid b_k$ for all $k\geq 1$ then $(S^q,\nu)$ is ergodic,
\item
if $q\mid b_k$ for some $k\geq 1$ then such $k$ is unique and for any $m\ge1$, $S^{q^m}$ has at most $b_k$ ergodic components
\end{itemize}
(see Theorem~4.4 in~\cite{B-Free}).
\end{enumerate}
Thus, we can apply Theorem~\ref{tw:abs1} to $z$:
\begin{Cor}
Fix $\vep>0$. For $z=\mob_{\mathscr{B}}$ (including the case $z=\mob$), there exists a Toeplitz sequence $t$ which correlates with $z$, such that $h_{top}(t)\geq (1-\vep)h_{top}(z^2)$.
\end{Cor}

\begin{Remark}
It would be interesting to know, whether we can find a Toeplitz sequence $t$ so that $t$ correlates with $\mob_{\mathscr{B}}$ and, moreover,
$h_{top}(t)\geq (1-\vep)h_{top}(z)$ (cf.\ Remark~\ref{siodemki}).
\end{Remark}

\begin{Remark}
Recall that $h_{top}(\mob^2)= 6/\pi^2$. Therefore, in view of Proposition~\ref{pr:7} and Theorem~\ref{tw:abs1}, in case $z=\mob$, it suffices to take $q=5$ in the construction of $t$, in order to obtain $h_{top}(t)>0$.
\end{Remark}

\subsubsection{Sturmian sequences}
Let $\eta\in \{0,1\}^{\N^*}$ be a Sturmian sequence and let $u\in \{-1,0,1\}^{\N^*}$
be a generic point for some Bernoulli measure $\mathbb{B}$. Let $z:=\eta\cdot u$. Let $\nu$ be the measure from Remark~\ref{stu}.
Denote $\rho=m_\ast(\nu\ot\mathbb{B})$, see~\eqref{eq:map_m}.

\begin{Lemma}\label{l15}
$h_{top}(\text{supp}(\rho))=h_{top}(z)>0$.
\end{Lemma}
\begin{proof}
In view of Remark~\ref{rk17}, it suffices to show that $h_{top}(\text{supp}(\rho))= h_{top}(z)$.

Clearly, whenever $B$ is such that $\rho(B)>0$, then $B$ appears in $z$. Let now $B$ be a block which appears in $z$. Then $B=B_1\cdot B_2$ (the multiplication is to be understood coordinatewise) for some block $B_1$ which appears in $\eta$ and some block $B_2$ which appears in $u$. Therefore
\begin{multline*}
\rho(B)=\rho(B_1\cdot B_2)=\nu\otimes \mathbb{B}(m^{-1}(B_1\cdot B_2))\\
\geq \nu\otimes \mathbb{B}(B_1\times B_2)=\nu(B_1)\cdot \B(B_2)>0
\end{multline*}
by Remark~\ref{stu} and since $\B(C)>0$ for any block $C$. This ends the proof.
\end{proof}

\begin{Lemma}\label{l16}
$\lim_{N\to \infty}\frac{1}{N}\sum_{n\leq N}z^2(n)>0$.
\end{Lemma}
\begin{proof}
It follows from Remark~\ref{stu} that $(S,X_{\eta},\nu)\perp (S,X_{u},\B)$, whence $(\eta,u)$ is generic for $\nu\otimes \B$. Therefore,
\begin{align*}
\frac{1}{N}\sum_{n\leq N}z^2 (n)&=\frac{1}{N}\sum_{n\leq N}\eta^2(n)\cdot u^2 (n)=\frac{1}{N}\sum_{n\leq N}\eta(n)\cdot u^2 (n)\\
&=\frac{1}{N}\sum_{n\leq N}(\raz_{\{w : w(1)=1\}}\otimes \raz_{\{w : w(1)=\pm 1\}})((S\times S)^n(\eta ,u ))\\
&\tend{N}{\infty}\int \raz_{\{w : w(1)=1\}}\otimes \raz_{\{w : w(1)=\pm 1\}}\ d(\nu\otimes \B)\\
&=\nu(\{w : w(1)=1\})\cdot \B(\{w : w(1)=\pm 1\})>0.
\end{align*}
\end{proof}

\begin{Remark}\label{ttt}
Notice that $(S,X_z,\rho)$ with $\rho=m_\ast (\nu\otimes \mathbb{B})$ is a factor of $(S,X_\eta,\nu)\times (S,X_u,\B)$. Therefore, $z$ is generic for $\rho$. Moreover, the eigenvalues of $(S,X_\eta,\nu)\times (S,X_u,\B)$ and $(S,X_\eta,\nu)$ are the same, and there exists some $\lambda\in\C$ with $|\lambda|=1$ such that these eigenvalues are of the form $\lambda^n$, $n\in\Z$. In particular, any eigenvalue of $(S,X_z,\rho)$ is of the form $\lambda^n$ for some $n\in\Z$.
\end{Remark}

\begin{Lemma}\label{l17}
There exists $b\geq 1$ such that $(S^r,X_z,\rho)$ has at most $b$ ergodic components for any $r\geq 1$.
\end{Lemma}
\begin{proof}
If $(S,X_z,\rho)$ is totally ergodic, the assertion of the lemma is true. Assume now that $(S,X_z,\rho)$ is not totally ergodic. Let $\lambda\in \C$ be as in Remark~\ref{ttt} and let $n_0\geq 1$ be the smallest natural number such that $\lambda^{n_0}=1$. It follows that $(S^r,X_z,\rho)$ has at most $n_0$ ergodic components for any $r\geq 1$, which ends the proof.
\end{proof}

\begin{Cor}
Fix $\vep>0$. For $z$ defined as above, there exists a Toeplitz sequence $t$ which correlates with $z$ and such that
$$
h_{top}(t)\geq (1-\vep)h_{top}(z).
$$
\end{Cor}
\begin{proof}
In view of Lemma~\ref{l15}, Lemma~\ref{l16}, Remark~\ref{ttt} and Lemma~\ref{l17}, the assumptions of Theorem~\ref{tw:abs2} are satisfied for $z$ and the assertion follows.
\end{proof}

\begin{appendix}

\section{Possible pairs of entropies}

The purpose of this appendix is to study the set of possible values of the pair $(h_{top}(z^2),h_{top}(z))$ for $z\in\{-1,0,1\}^{\N^\ast}$ (we do not assume here that $z$ satisfies~\eqref{Ch}). Recall the definition of function $H$ from Lemma~\ref{lemma:Shields}:
$$
H(x)=-x\cdot\log x-(1-x)\cdot \log(1-x) \text{ for }x\in(0,1)
$$
and consider $H_1:=H|_{(0,1/2]},\ H_2:=H|_{[1/2,1)}$. Moreover, define $f_1,f_2\colon [0,1]\to \R$ by
$$
f_1(x):=x+H_1^{-1}(x),\ f_2(x):=\begin{cases}
x+H_2^{-1}(x),& \text{ for }x<H(2/3),\\
\log 3,&\text{ for }x\geq H(2/3),
\end{cases}
$$
see Figure~\ref{fig:2} and~\ref{fig:3}. 
\begin{Remark}\label{inc}
Elementary calculation shows that $x\mapsto x+H_2^{-1}(x)$ is increasing on $(0,H(2/3)]$.
\end{Remark}

We have the following restrictions for $(h_{top}(z^2),h_{top}(z))$:
\begin{figure}
    \begin{subfigure}{0.5\textwidth}
	 \includegraphics[height=6cm]{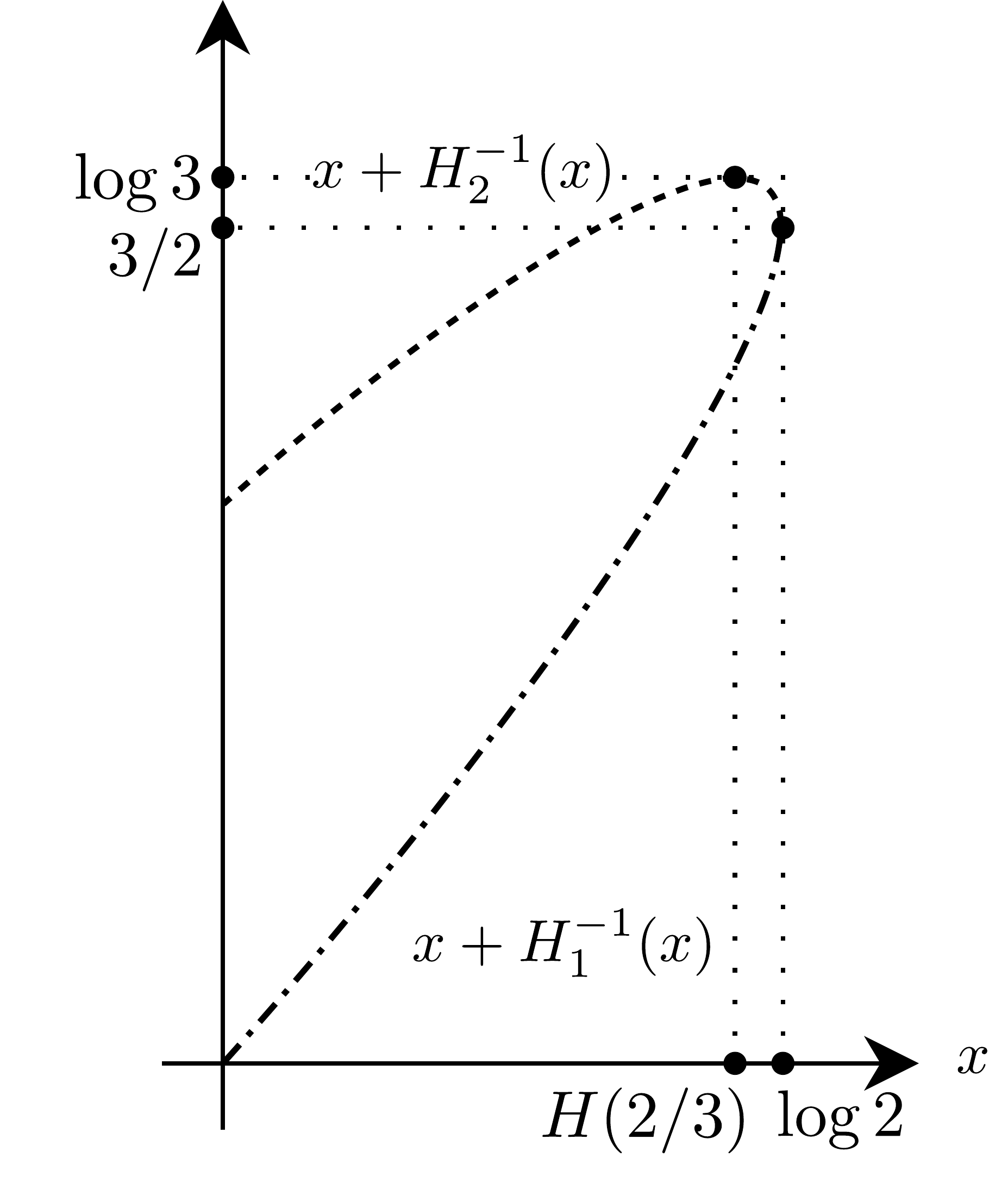}
	 \caption{Functions $x+H_1^{-1}(x)$ and $x+H_2^{-1}(x)$.}
	 \label{fig:2}
    \end{subfigure}
    \begin{subfigure}{0.5\textwidth}
		 \includegraphics[height=6cm]{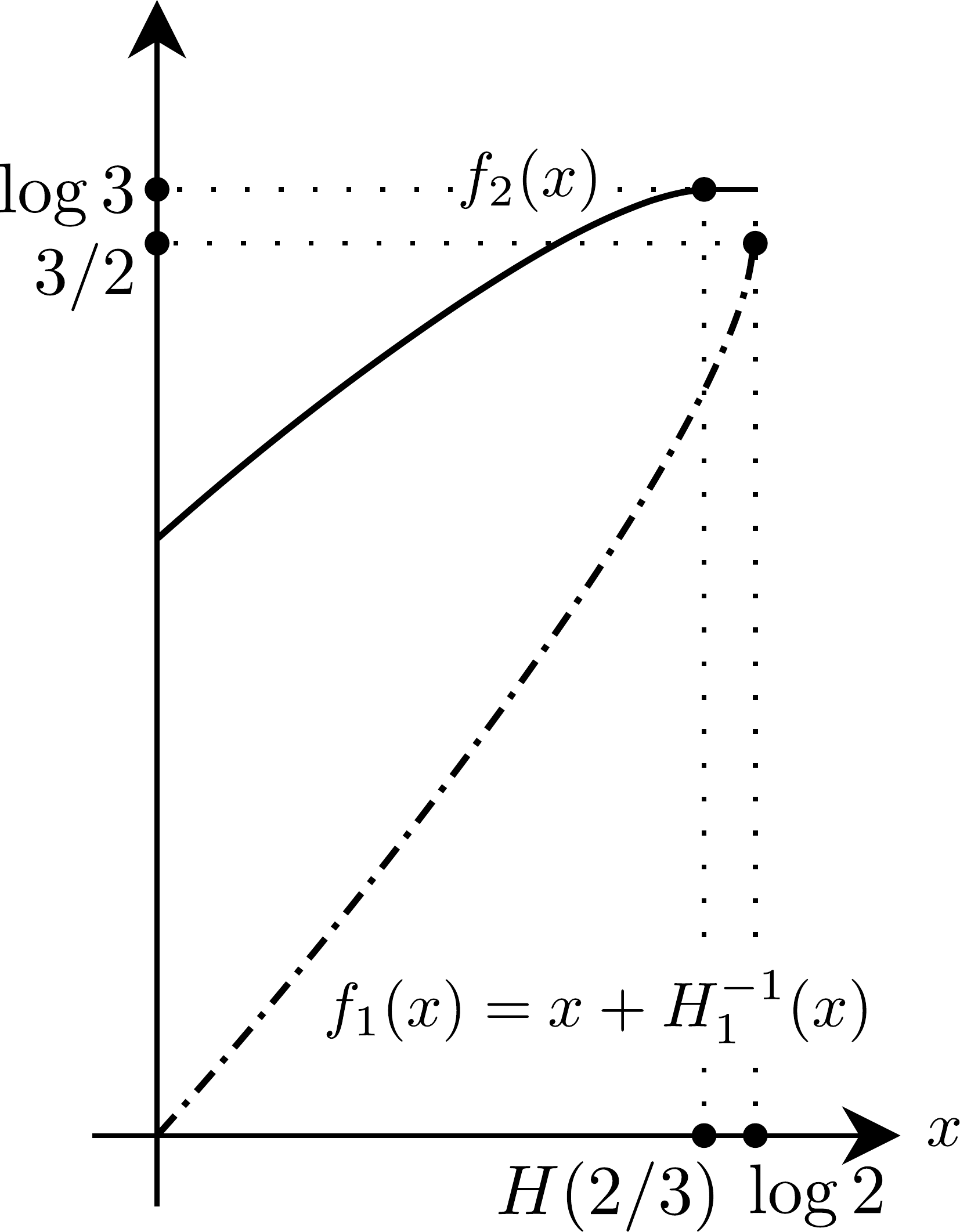}
		 \caption{Functions $f_1$ and $f_2$.}
		 \label{fig:3}
    \end{subfigure}
    \caption{Graphs of the functions used in the bounds for measure-theoretic entropy (left) and topological entropy (right).}
  \end{figure}

\begin{Prop}\label{wniosek3}
Let $z\in\{-1,0,1\}^{\N^\ast}$. Then
$$
h_{top}(z) \leq f_2(h_{top}(z^2) )
$$
with the equality $h_{top}(z) =f_2(h_{top}(z^2) )$ if and only if $h_{top}(z)=\log 3$. If additionally
\begin{equation}\label{e:recur}
X_z=\pi^{-1}(X_{z^2}),
\end{equation}
then
$f_1(h_{top}(z^2) )<h_{top}(z).$
\end{Prop}

One can also show that the bounds given by Proposition \ref{wniosek3} cannot be improved:
\begin{Prop}\label{p:notimpro}
For any $h_{z^2}\in (0,1]$ and $\vep>0$ there exists $z\in\{-1,0,1\}^{\N^\ast}$ such that
\begin{itemize}
\item
$h_{top}(z^2)\in [h_{z^2}-\vep,h_{z^2}+\vep]$,
\item
$h_{top}(z)\geq f_2(h_{z^2})-\vep$.
\end{itemize}
Moreover, there exists $z\in \{-1,0,1\}^{\N^\ast}$ satisfying condition~\eqref{e:recur} and such that
\begin{itemize}
\item
$h_{top}(z^2)\in [h_{z^2}-\vep,h_{z^2}+\vep]$,
\item
$h_{top}(z)\leq f_1(h_{z^2})+\vep$.
\end{itemize}
\end{Prop}

The main idea in the proof of Proposition~\ref{wniosek3} is to study the pairs $(h(\nu),h(\widehat{\nu}))$ for $\nu\in\mathcal{P}_S(\{0,1\}^\Z)$. We have the following:
\begin{Prop}\label{p:mes}
Let $\nu\in\mathcal{P}^e_S(\{0,1\}^\Z)$ and set $d:=\nu([1])$. Then
$$
h(\hat{\nu})=h(\nu)+d, \text{ and }d\in [H_1^{-1}(h(\nu)),H_2^{-1}(h(\nu))].
$$
Moreover, the following conditions are equivalent:
\begin{enumerate}[label=(\alph*)]
\item\label{803a}
$d=H_i^{-1}(h(\nu))$ for $i=1$ or $i=2$,
\item\label{803b}
$\nu=B(1-d,d)$.
\end{enumerate}
\end{Prop}

The proof of these results starts with the following simple observation:
\begin{Lemma}\label{pr:1a}
Let $z\in \{-1,0,1\}^{\N^\ast}$. Then
\[
h_{top}(z^2)\leq h_{top}(z)\leq \min(h_{top}(z^2)+1,\log 3).
\]
\end{Lemma}
\begin{proof}
Clearly, if $z^2(k)=1$ for some $k\in \Z$ then $z(k)\in \{-1,1\}$. Hence
$
p_{z}(n)\leq 2^n\cdot p_{z^2}(n).
$
It follows immediately that
\begin{multline*}
h_{top}(z)=\lim_{n\to \infty}\frac{1}{n} \log p_z(n)\leq \lim_{n\to \infty}\frac{1}{n} \log (2^n\cdot p_{z^2}(n))\\
=1+\lim_{n\to \infty}\frac{1}{n}\log p_{z^2}(n)=1+h_{top}(z^2).
\end{multline*}
This completes the proof as clearly $h_{top}(z^2)\leq h_{top}(z)\leq \log 3$.
\end{proof}

\subsection{Measure-theoretical setting}

The restrictions on the pairs $(h(\nu),h(\widehat{\nu}))$ for $\nu\in\mathcal{P}_S(\{0,1\}^\Z)$ are described in Proposition~\ref{p:mes}. For its proof, we will need the following lemma.

\begin{Lemma}\label{lemma:nowy}
Let $d\in[0,1]$, $n\ge1$ and $\vep>0$. Let
$$
\mathcal{A}_{n,d,\vep}:=\Big\{C\in\{0,1\}^n : \left|\frac{|\text{supp} (C)|}{n}-d \right|<\vep\Big\}.
$$
Then
\begin{equation}
 \label{eq:Andep}|\mathcal{A}_{n,d,\vep}| \le 2^{n\sup\{H(d'):d-\vep\le d' \le d+\vep\}}.
\end{equation}
Moreover, if
$\nu\in\mathcal{P}_S^e(\{0,1\}^\Z)$ with $\nu([1])=d$, then for $n$ sufficiently large there exists $\mathcal{C}_n\subset \mathcal{A}_{n,d,\vep}$ such that
$$
\nu\left(\bigcup_{C\in\mathcal{C}_n}[C]\right)>1-\vep \text{ and }|\mathcal{C}_n| \geq 2^{n(h(\nu)-\vep)}.
$$
\end{Lemma}
\begin{proof}
If $d-\vep\le 1/2\le d+\vep$, \eqref{eq:Andep} is obvious because $H(1/2)=1$. If $d+\vep<1/2$, this is a consequence of Lemma~\ref{lemma:Shields}. Otherwise, we have $1-d+\vep<1/2$, so we can apply Lemma~\ref{lemma:Shields} to
$\mathcal{A}_{n,1-d,\vep}$, and observe that $H$ is symmetric, and that $|\mathcal{A}_{n,d,\vep}|=|\mathcal{A}_{n,1-d,\vep}|$.

Now, let $\nu\in\mathcal{P}_S^e(\{0,1\}^\Z)$ with $\nu([1])=d$. Set $h:=h(\nu)$.
It follows from the Shannon-McMillan theorem and from the mean ergodic theorem that for $n$ sufficiently large there exists $\mathcal{C}_n\subset \mathcal{A}_{n,d,\vep}$
such that
$$
\nu\left(\bigcup_{C\in\mathcal{C}_n}[C]\right)>1-\vep/2,
$$
and, for all $C\in\mathcal{C}_n$,
\begin{equation*}
2^{-(h+\vep)n}<\nu([C])<2^{-(h-\vep)n}.
\end{equation*}
It follows that
$$
1-\vep/2 <\sum_{C\in\mathcal{C}_n}\nu([C]) <|\mathcal{C}_n|\cdot 2^{-(h-\vep)n},
$$
whence
$
|\mathcal{C}_n|>(1-\vep/2)2^{n(h-\vep)}\geq 2^{n(h-\vep)}.
$
\end{proof}

\begin{proof}[Proof of Proposition~\ref{p:mes}]
Let $\vep>0$, set $h:=h(\nu)$ and $d:=\nu([1])$.
Let $\mathcal{A}_{n,d,\vep}$ be as in Lemma~\ref{lemma:nowy}. Then
$$
2^{(h-\vep)n}\leq 2^{n\sup\{H(d'):d-\vep\le d'\le d+\vep\}}
$$
for all $n$ sufficiently large. Thus, by the continuity of $H$, we obtain $h\leq H(d)$, which implies that
\begin{equation}
\label{eq:dH}
d\in [H_1^{-1}(h),H_2^{-1}(h)].
\end{equation}
We are now ready to estimate $h(\widehat{\nu})$. We have
\begin{multline*}
\frac{1}{n}\sum_{\substack{B\in \{-1,0,1\}^n \\ \widehat{\nu}(B)>0}}-\widehat{\nu}(B)\log \widehat{\nu}(B)\\
=\frac{1}{n}\sum_{C\in\mathcal{A}_{n,d,\vep}}\sum_{\substack{B\in\{-1,0,1\}^n \\ B^2=C}}-\frac{\nu(C)}{2^{|\text{supp}(C)|}}\cdot\log\frac{\nu(C)}{2^{|\text{supp}(C)|}}\\
+\frac{1}{n}\sum_{\substack{B\in\{-1,0,1\}^n\\ B^2\in\mathcal{A}_{n,d,\vep}^{c}\\ \nu(B^2)>0}}-\widehat{\nu}(B)\log\widehat{\nu}(B).
\end{multline*}
Let
$
\delta:=\widehat{\nu}\Big(\bigcup_{\substack{B\in\{-1,0,1\}^n\\ B^2\in\mathcal{A}_{n,d,\vep}^{c}}} B\Big)=\nu\Big( \bigcup_{C\in\mathcal{A}_{n,d,\vep}^{c}}C\Big)<\vep.
$
Then
\begin{align*}
\sum_{\substack{B\in\{-1,0,1\}^n\\ B^2\in\mathcal{A}_{n,d,\vep}^{c}\\ \nu(B^2)>0}}(-\widehat{\nu}(B))\log\widehat{\nu}(B)
&=
\delta \sum_{\substack{B\in\{-1,0,1\}^n\\ B^2\in\mathcal{A}_{n,d,\vep}^{c}\\ \nu(B^2)>0}}-\frac{\widehat{\nu}(B)}{\delta}\cdot\log \frac{\widehat{\nu}(B)}{\delta}
-\delta\log \delta\\
&\le \delta n \log 3-\delta\log \delta.\\
\end{align*}
It follows that
\begin{align*}
h(\widehat{\nu})&\le
\lim_{n\to\infty}\frac{1}{n}\sum_{C\in\mathcal{A}_{n,d,\vep}}\sum_{\substack{B\in\{-1,0,1\}^n \\ B^2=C}}-\frac{\nu(C)}{2^{|\text{supp}(C)|}}\cdot\log\frac{\nu(C)}{2^{|\text{supp}(C)|}}+\delta\log3\\
&\le\lim_{n\to\infty}\left(\frac{1}{n}\sum_{C\in\mathcal{A}_{n,d,\vep}}(-\nu(C)\log\nu(C))+\frac{1}{n}\sum_{C\in\mathcal{A}_{n,d,\vep}}\nu(C)|\text{supp}(C)| \right) + \vep \log 3\\
&\le h(\nu)+ d+\vep + \vep \log 3.
\end{align*}
Since $\vep$ is arbitrarily small, and remembering that~\eqref{eq:dH} holds, we obtain
$$h(\widehat{\nu}) \le h(\nu)+ d \le h(\nu)+H_2^{-1}(h(\nu)).$$
On the other hand, by similar arguments,
$$
h(\widehat{\nu})\geq h(\nu)+d \geq h(\nu)+H_1^{-1}(h(\nu)),
$$
and the first part of the assertion follows.

Now, consider the partition $\{[0],[1]\}$. This is a generating partition and it follows that $h(\nu)\leq H(d)$; the inequality is sharp, unless $\nu=B(1-d,d)$. In other words,~\ref{803a} does not hold, unless~\ref{803b} holds. Clearly,~\ref{803b} implies~\ref{803a}. 
\end{proof}
\subsection{Topological setting}
The restrictions on the pairs $(h_{top}(z^2),h_{top}(z))$ are listed in Proposition~\ref{wniosek3}. For the proof, we will need some auxiliary lemmas.

\begin{Lemma}\label{lm:17}
Let $\nu\in \mathcal{P}_S(\{0,1\}^\Z)$. Then
$$
h(\widehat{\nu})=\max\{h(\rho) : \rho\in\mathcal{P}_S(\{-1,0,1\}^\Z), \pi_\ast(\rho)=\nu\}.
$$
\end{Lemma}
\begin{proof}
For any $\rho\in\mathcal{P}_S(\{-1,0,1\}^\Z)$ satisfying $\pi_\ast(\rho)=\nu$, we have
\begin{align*}
-h(\rho)&=\lim_{n\to\infty} \frac{1}{n}\sum_{B\in\{-1,0,1\}^n}\rho(B)\log \rho(B)=\lim_{n\to\infty} \frac{1}{n}\sum_{C\in \{0,1\}^n}\sum_{B^2=C}\rho(B)\log \rho(B) \\
&\geq \lim_{n\to\infty} \frac{1}{n}\sum_{C\in \{0,1\}^n}\sum_{B^2=C}\hat{\nu}(B)\log \hat{\nu}(B)=-h(\hat{\nu})
\end{align*}
(the inequality follows from the fact that $\sum_{i=1}^{n}a_i\log (a_i)\geq \sum_{i=1}^{n}\frac{a}{n}\log\frac{a}{n}$, where $a=\sum_{i=1}^{n}a_i$).
\end{proof}

\begin{Remark}\label{doda}
Suppose that $z\in \{-1,0,1\}^{\N^\ast}$ satisfies~\eqref{e:recur}, that is, for each $B$ appearing in $z^2$ and each $C$ such that $C^2=B$, the block $C$ appears in $z$.
Then clearly
$$
\nu\in \mathcal{P}_S(X_{z^2}) \iff \hat{\nu}\in\mathcal{P}_S(X_z).
$$
It follows that if $\rho\in \mathcal{P}_S(X_z)$ is such that $h_{top}(z)=h(\rho)$ then, by Lemma~\ref{lm:17},
$$
h_{top}(z)=h(\rho)\leq h(\hat{\pi_\ast(\rho)})\leq h_{top}(z),
$$
i.e. $h_{top}(z)=h(\hat{\pi_\ast(\rho)})$.
\end{Remark}

\begin{Lemma}\label{lm:18}
Let $z\in \{-1,0,1\}^{\N^\ast}$. Then
$$
h_{top}(z)\leq \sup\{h(\widehat{\nu}): \nu\in\mathcal{P}_S(X_{z^2})\}=\max\{h(\widehat{\nu}): \nu\in\mathcal{P}_S(X_{z^2})\}.
$$
\end{Lemma}
\begin{proof}
Let $X:=\pi^{-1}(X_{z^2})$. Let $\rho\in\mathcal{P}_S(X)$ be such that $h(\rho)=h_{top}(X)$. Then by Lemma~\ref{lm:17}
\[
 h_{top}(z) \le h_{top}(X) = h(\rho) \le h (\widehat{\pi(\rho)}) \le h_{top}(X).
\]
But $\pi(\rho)\in\mathcal{P}_S(X_{z^2})$, hence
$$h (\widehat{\pi(\rho)}) \le \sup\{h(\widehat{\nu}): \nu\in\mathcal{P}_S(X_{z^2})\}.$$
On the other hand, if $\nu\in\mathcal{P}_S(X_{z^2})$, $\widehat{\nu}\in \mathcal{P}_S(X)$, and
$
h(\widehat{\nu})\le h(\rho)= h_{top}(X)$. It follows that
the supremum in the statement of the lemma is equal to $h (\widehat{\pi(\rho)})$, and the result is proved.
\end{proof}

\begin{proof}[Proof of Proposition~\ref{wniosek3}]
By Lemma~\ref{lm:18},
$$ h_{top}(z) \le \max\{h(\widehat{\nu}):\nu\in\mathcal{P}_S(X_{z^2})\}. $$
 Let $\kappa\in\mathcal{P}_S(X_{z^2})$ be such that
$$ h(\widehat\kappa) = \max\{h(\widehat{\nu}):\nu\in\mathcal{P}_S(X_{z^2})\}. $$
By Proposition~\ref{p:mes},
$$ h(\widehat\kappa)\le f_2(h(\kappa)) \le \sup \{f_2(h(\nu)):\nu\in\mathcal{P}_S(X_{z^2})\}. $$
By Remark~\ref{inc}, the latter expression is equal to $f_2(h_{top}(z^2))$,
whenever $h_{top}(z^2)< H(2/3)$. But if $h_{top}(z^2)\ge H(2/3)$, then $f_2(h_{top}(z^2))=\log 3\ge f_2(h(\nu))$ for any $\nu\in\mathcal{P}_S(X_{z^2}h_{top}(z)$. Thus, we have obtained
$$
\sup \{f_2(h(\nu)):\nu\in\mathcal{P}_S(X_{z^2})\}(z)\leq f_2(h_{top}(z^2)).
$$
If all of the above inequalities are equalities, then in particular
$$ h(\widehat\kappa) = f_2(h(\kappa)), $$
which by Proposition~\ref{p:mes} happens only if $\kappa$ is Bernoulli. This implies
$ h_{top}(z^2)=\log2=1. $
But then
$$ h_{top}(z) = f_2(h_{top}(z^2)) = f_2 (1) = \log 3. $$
On the other hand, if $h_{top}(z)=\log3$, then $h_{top}(z^2)=\log2$ and we obtain
$$ h_{top}(z)=\log3 = f_2(h_{top}(z^2)). $$

Suppose now that~\eqref{e:recur} holds and let $\nu\in\mathcal{P}_S(X_{z^2})$ be such that $h_{top}(z^2) =h(\nu)$. Then, by Proposition~\ref{p:mes},
$$
h_{top}(z) \geq h(\widehat{\nu})\geq f_1(h(\nu))=f_1(h_{top}(z^2)).
$$
We claim that the second inequality is sharp. Indeed, assume it is an equality. Then by Proposition~\ref{p:mes}, $\nu$ is Bernoulli. This, together with~\eqref{e:recur} gives $h_{top}(z)=\log 3$. Then $h_{top}(z^2)=1$. But for these values, we have
$ \log3 > \dfrac{3}{2} = f_1(\log2).$
\end{proof}

Proposition~\ref{p:notimpro} shows that the bounds given by Proposition \ref{wniosek3} cannot be improved.

\begin{proof}[Proof of Proposition~\ref{p:notimpro}]
The proof of both parts of the assertion goes along the same lines and we will provide details only for the first part.

Let $d\ge 1/2$ be such that $H(d)=h_{z^2}$. Fix $\vep>0$ and for $n\in\N$ let $\mathcal{A}_{n,d,\vep}$ be as in Lemma~\ref{lemma:nowy}. Then,
$$
|\mathcal{A}_{N,d,\vep}|\leq 2^{N\sup\{(H(d'): d-\vep\le d'\le d+\vep\} }.
$$
Applying the second part of Lemma~\ref{lemma:nowy} with $\nu:=B(1-d,d)$, for $N$ sufficiently large, we get
$$
2^{(H(d)-\vep)N}\leq |\mathcal{A}_{N,d,\vep}|.
$$
Fix such $N$ and let $X\subset\{0,1\}^{\N^\ast}$ be the subshift consisting of these points $x$, for which any block appearing in $x$ is a subword of a concatenation of some words form $A_{N,d,\vep}$. Let
$$
\mathcal{C}_{nN}:=\{B\in X : |B|=nN\}.
$$
Then
$$
|\mathcal{A}_{N,d,\vep}|^n\leq |\mathcal{C}_{nN}|\leq N\cdot |A_{N,d,\vep}|^{n+1},
$$
whence
$$
H(d)-\vep \leq h_{top}(X)\leq \sup\{H(d'): d-\vep\le d'\le d+\vep\}.
$$

Moreover, if $n$ is sufficiently large, then, for $B\in \mathcal{C}_{nN}$ we have
$$
\left|\frac{|\text{supp}(B)|}{nN}-d \right|<2\vep.
$$
For $n\in\N$, let
$$
\mathcal{D}_{nN}:=\{C\in \{-1,0,1\}^{nN} : C^2\in \mathcal{C}_{nN}\}.
$$
It follows that
$$
|\mathcal{D}_{nN}|\geq 2^{(d-2\vep)nN}|\mathcal{C}_{nN}|\geq 2^{(d-2\vep)nN}|\mathcal{A}_{N,d,\vep}|^n\geq 2^{(d-2\vep)nN}\cdot 2^{(H(d)-\vep)nN}.
$$
Then
$$
H(d)+d-\vep-2\vep\leq h_{top}(\pi^{-1}({X})).
$$
By the choice of $d$, $d=H_2^{-1}(h_{z^2})$. Then, by continuity of $H$,
to complete the proof, it suffices to pick $z\in \pi^{-1}(X)$ such that any block appearing in $\pi^{-1}(X)$ also appears in $z$.
\end{proof}

\end{appendix}

\small
\bibliography{biblio_ChS}

\providecommand{\bysame}{\leavevmode\hbox to3em{\hrulefill}\thinspace}
\providecommand{\MR}{\relax\ifhmode\unskip\space\fi MR }
\providecommand{\MRhref}[2]{%
  \href{http://www.ams.org/mathscinet-getitem?mr=#1}{#2}
}
\providecommand{\href}[2]{#2}
\begin{thebibliography}{10}

\bibitem{Ab-Ka-Le}
H.~{\noopsort{Abdalaoui}}El~Abdalaoui, S.~Kasjan, and M.~Lema\'{n}czyk,
  \emph{0-1 sequences of the {T}hue-{M}orse type and {S}arnak's conjecture}, To
  appear in {P}roc. {A}mer. {M}ath. {S}oc.

\bibitem{B-Free}
H.~{\noopsort{Abdalaoui}}El~Abdalaoui, M.~Lema\'{n}czyk, and T.~de~la Rue,
  \emph{A dynamical point of view on the set of $\mathscr{B}$-free integers},
  International Mathematics Research Notices \textbf{2015} (2015), no.~16,
  7258--7286.

\bibitem{Ap}
T.~M. Apostol, \emph{Introduction to analytic number theory}, Springer-Verlag,
  New York-Heidelberg, 1976, Undergraduate Texts in Mathematics.

\bibitem{Ce-Si}
F.~Cellarosi and Ya.~G. Sinai, \emph{Ergodic properties of square-free
  numbers}, J. Eur. Math. Soc. \textbf{15} (2013), no.~4, 1343--1374.

\bibitem{Cho}
S.~Chowla, \emph{The {R}iemann hypothesis and {H}ilbert's tenth problem},
  Mathematics and Its Applications, Vol. 4, Gordon and Breach Science
  Publishers, New York, 1965.

\bibitem{Da}
H.~Davenport, \emph{On some infinite series involving arithmetical functions.
  {II}}, Quart. J. Math. Oxford \textbf{8} (1937), 313--320.

\bibitem{Do0}
T.~Downarowicz, \emph{The {C}hoquet simplex of invariant measures for minimal
  flows}, Israel J. Math. \textbf{74} (1991), no.~2-3, 241--256.

\bibitem{Do1}
\bysame, \emph{Survey of odometers and {T}oeplitz flows}, Algebraic and
  topological dynamics, Contemp. Math., vol. 385, Amer. Math. Soc., Providence,
  RI, 2005, pp.~7--37.

\bibitem{Do}
\bysame, \emph{Entropy in {D}ynamical {S}ystems}, New Mathematical Monographs,
  vol.~18, Cambridge University Press, Cambridge, 2011.

\bibitem{Do-Ka}
T.~Downarowicz and S.~Kasjan, \emph{Odometers and {T}oeplitz subshifts
  revisited in the context of {S}arnak's conjecture}, to appear in Studia
  Math., \url{http://arxiv.org/abs/1502.02307}.

\bibitem{Fogg}
N.~P. Fogg, \emph{Substitutions in {D}ynamics, {A}rithmetics and
  {C}ombinatorics}, Lecture Notes in Mathematics, vol. 1794, Springer-Verlag,
  Berlin, 2002, Edited by V. Berth{{\'e}}, S. Ferenczi, C. Mauduit and A.
  Siegel.

\bibitem{Fur}
H.~Furstenberg, \emph{Disjointness in ergodic theory, minimal sets, and a
  problem in {D}iophantine approximation}, Math. Systems Theory \textbf{1}
  (1967), 1--49.

\bibitem{Ja-Ke}
K.~Jacobs and M.~Keane, \emph{{$0-1$}-sequences of {T}oeplitz type}, Z.
  Wahrscheinlichkeitstheorie und Verw. Gebiete \textbf{13} (1969), 123--131.

\bibitem{Ka}
T.~Kamae, \emph{Subsequences of normal sequences}, Israel J. Math. \textbf{16}
  (1973), 121--149.

\bibitem{Ke-Li}
D.~Kerr and H.~Li, \emph{Independence in topological and {$C^*$}-dynamics},
  Math. Ann. \textbf{338} (2007), no.~4, 869--926.

\bibitem{Kw}
D.~Kwietniak, \emph{Topological entropy and distributional chaos in hereditary
  shifts with applications to spacing shifts and beta shifts}, Discrete Contin.
  Dyn. Syst. \textbf{33} (2013), no.~6, 2451--2467.

\bibitem{Ma-Ra-Ta}
K.~Matom\"aki, M.~Radziwi{\l}{\l}, and T.~Tao, \emph{Sign patterns of the
  {L}iouville and {M}\"obius functions}, \url{http://arxiv.org/abs/1509.01545}.

\bibitem{Mi}
L.~Mirsky, \emph{Arithmetical pattern problems relating to divisibility by
  {$r$}th powers}, Proc. London Math. Soc. (2) \textbf{50} (1949), 497--508.

\bibitem{Or}
D.~Ornstein, \emph{Factors of {B}ernoulli shifts are {B}ernoulli shifts},
  Advances in Math. \textbf{5} (1970), 349--364 (1970).

\bibitem{Parry}
W.~Parry, \emph{Entropy and generators in ergodic theory}, W. A. Benjamin,
  Inc., New York-Amsterdam, 1969.

\bibitem{Peckner}
R.~Peckner, \emph{Uniqueness of the measure of maximal entropy for the
  squarefree flow}, to appear in Israel J. Math.,
  \url{http://arxiv.org/abs/1205.2905}.

\bibitem{Sarnak}
P.~Sarnak, \emph{Three lectures on the {M}{\"o}bius function, randomness and
  dynamics}, \url{http://publications.ias.edu/sarnak/}.

\bibitem{Shields}
P.~C. Shields, \emph{The ergodic theory of discrete sample paths}, Graduate
  Studies in Mathematics, vol.~13, American Mathematical Society, Providence,
  RI, 1996.

\bibitem{Tao}
T.~Tao, \emph{The {C}howla conjecture and the {S}arnak conjecture}, What's new
  (blog),
  \url{http://terrytao.wordpress.com/2012/10/14/the-chowla-conjecture-and-the-sarnak-conjecture}.

\bibitem{Thouvenot}
J.-P. Thouvenot, \emph{Une classe de syst{\`e}mes pour lesquels la conjecture
  de {P}insker est vraie}, Israel J. Math. \textbf{21} (1975), no.~2-3,
  208--214, Conference on Ergodic Theory and Topological Dynamics (Kibbutz
  Lavi, 1974).

\bibitem{Titchmarsh}
E.~C. Titchmarsh, \emph{The theory of the {R}iemann zeta-function}, second ed.,
  The Clarendon Press Oxford University Press, New York, 1986, Edited and with
  a preface by D. R. Heath-Brown.

\bibitem{Wa}
P.~Walters, \emph{An introduction to ergodic theory}, Graduate Texts in
  Mathematics, vol.~79, Springer-Verlag, New York, 1982.

\bibitem{We}
B.~Weiss, \emph{Normal sequences as collectives}, Proc. Symp. on Topological
  Dynamics and ergodic theory, Univ. of Kentucky, 1971.

\bibitem{We2}
\bysame, \emph{Single orbit dynamics}, CBMS Regional Conference Series in
  Mathematics, vol.~95, American Mathematical Society, Providence, RI, 2000.

\bibitem{Wi}
S.~Williams, \emph{Toeplitz minimal flows which are not uniquely ergodic}, Z.
  Wahrsch. Verw. Gebiete \textbf{67} (1984), no.~1, 95--107.

\end{thebibliography}

\bigskip
\footnotesize

\noindent
El Houcein El Abdalaoui\\
\textsc{Laboratoire de Math\'{e}matiques Rapha\"{e}l Salem, Normandie Universit\'{e}, Universit\'{e} de Rouen, CNRS -- Avenue de l'Universit\'{e} -- 76801 Saint Etienne du Rouvray, France}\par\nopagebreak
\noindent
\texttt{elhoucein.elabdalaoui@univ-rouen.fr}

\medskip

\noindent
Joanna Ku\l aga-Przymus\\
\textsc{Institute of Mathematics, Polish Acadamy of Sciences, \'{S}niadeckich 8, 00-956 Warszawa, Poland}\\
\textsc{Faculty of Mathematics and Computer Science, Nicolaus Copernicus University, Chopina 12/18, 87-100 Toru\'{n}, Poland}\par\nopagebreak
\noindent
\texttt{joanna.kulaga@gmail.com}

\medskip

\noindent
Mariusz Lema\'{n}czyk\\
\textsc{Faculty of Mathematics and Computer Science, Nicolaus Copernicus University, Chopina 12/18, 87-100 Toru\'{n}, Poland}\par\nopagebreak
\noindent
\texttt{mlem@mat.umk.pl}

\medskip

\noindent
Thierry de la Rue\\
\textsc{Laboratoire de Math\'{e}matiques Rapha\"{e}l Salem, Normandie Universit\'{e}, Universit\'{e} de Rouen, CNRS -- Avenue de l'Universit\'{e} -- 76801 Saint Etienne du Rouvray, France}\par\nopagebreak
\noindent
\texttt{Thierry.de-la-Rue@univ-rouen.fr}

\end{document}